\documentclass[12pt]{article}
\usepackage{amsmath,amssymb,amsfonts,amsthm}
\usepackage[all]{xy}

\newcounter{item}[section]
\newcounter{kirshr}
\newcounter{kirsha}
\newcounter{kirshb}
\newenvironment{enumroman}{\setcounter{kirshr}{1}
\begin{list}{(\roman{kirshr})}{\usecounter{kirshr}} }{\end{list}}
\newenvironment{enumarab}{\setcounter{kirshb}{1}
\begin{list}{(\arabic{kirshb})}{\usecounter{kirshb}} }{\end{list}}
\newenvironment{athm}[1]{\vskip3mm\par\noindent
{\bf #1 }. \slshape }
{\upshape\par\vskip10pt minus3pt}
\newtheorem{theorem}{Theorem}

\newtheorem{lemma}[theorem]{Lemma}
\newtheorem{corollary}[theorem]{Corollary}
\newenvironment{demo}[1]{\noindent{\bf #1.}\upshape\mdseries}
{\nopagebreak{\hfill\rule{2mm}{2mm}\nopagebreak}\par\normalfont}
\theoremstyle{definition}

\newtheorem{example}[theorem]{Example}
\newtheorem{definition}[theorem]{Definition}

\def\C{{\mathfrak{C}}}

\def\At{{\sf At}}
\def\N{{\cal N}}
\def\Nr{{\mathfrak{Nr}}}
\def\Fr{{\mathfrak{Fr}}}
\def\Sg{{\mathfrak{Sg}}}

\def\A{{\mathfrak{A}}}
\def\B{{\mathfrak{B}}}
\def\C{{\mathfrak{C}}}
\def\D{{\mathfrak{D}}}
\def\F{{\mathfark{F}}}
\def\M{{\mathfrak{M}}}

\def\Ig{{\mathfrak{Ig}}}
\def\CA{{\bf CA}}
\def\QA{{\bf QA}}
\def\SC{{\bf SC}}
\def\QEA{{\bf QEA}}

\def\RQA{{\bf RQA}}

\def\Lf{{\bf Lf}}
\def\Dc{{\bf Dc}}

\def\K{{\bf K}}
\def\K{{\bf K}}
\def\RK{{\bf RK}}
\def\RCA{{\bf RCA}}

\def\Rd{{\mathfrak{Rd}}}
\def\(R)RA{{\bf (R)RA}}
\def\RA{{\bf RA}}

\def\Dc{{\bf Dc}}
\def\R{\mathbb{R}}
\def\Dc{{\bf Dc}}

\def\QRA{{\bf QRA}}

\def\Sc{{\bf Sc}}

\def\c #1{{\cal #1}}
 \def\CA{{\sf CA}}
\def\B{{\sf B}}
\def\G{{\sf G}}

\def\K{{\sf K}}
 \def\Cm{{\mathfrak{Cm}}}
\def\Nr{{\mathfrak{Nr}}}
\def\SNr{{\bf S}{\mathfrak{Nr}}}

\def\Ra{{\mathfrak{Ra}}}
\def\Ca{{\mathfrak{Ca}}}
\def\set#1{\{#1\} }
\def\Ra{{\mathfrak{Ra}}}
\def\Nr{{\mathfrak{Nr}}}

\def\A{{\mathfrak{A}}}
\def\B{{\mathfrak{B}}}
\def\C{{\mathfrak{C}}}
\def\D{{\mathfrak{D}}}

\def\CA{{\bf CA}}
\def\RA{{\bf RA}}

\def\RCA{{\bf RCA}}
\def\Kn{{\bf Kn}}
\def\G{{\bf G}}

\def\R{\cal{R}}

\def\PEA{{\bf PEA}}
\def\PA{{\bf PA}}

\def\Los{\L{o}\'{s}}

\def\RA{{\bf RA}}

\def\At{{\mathfrak{At}}}

\def\G{{\mathfrak{G}}}

\def\SC{{\bf SC}}
\def\K{{\bf K}}

\def\RQEA{{\bf RQEA}}

\def\tr{{\bf tr}}

\def\c #1{{\cal #1}}

\def\Ca{{\mathfrak{Ca}}}
\def\Rl{{\mathfrak{Rl}}}

\def\A{{\cal{A}}}
\def\B{{\mathfrak{B}}}
\def\C{{\mathfrak{C}}}
\def\D{{\mathfrak{D}}}
\def\P{{\mathfrak{P}}}
\def\Bl{{\mathfrak{Bl}}}
\def\Ra{{\mathfrak{Ra}}}
\def\Nr{{\mathfrak{Nr}}}
\def\F{{\mathfrak{F}}}
\def\CA{{\bf CA}}
\def\RCA{{\bf RCA}}
\def\FPA{{\bf FPA}}

\def\c#1{{\mathcal #1}}

\def\Ca{{\mathfrak Ca}}

\def\Cm{{\mathfrak Cm}}
\def\Sg{{\mathfrak Sg}}
\def\H{{\mathfrak H}}
\def\CPA{{\sf CPA}}
\def\A{{\mathfrak A}}
\def\Str{{\mathfrak Str}}
\def\Los{\L{o}\'{s}}

\def\At{{\sf At}}
\def\WFPA{{\sf WFPA}}

\def\rng{{\sf rng}}

\def\dim{{\sf dim}}

\def\Los{\L{o}\'{s}}

\title{What is the spirit of the cylindric paradigm, as opposed to that of the polyadic one?}
\author{Tarek Sayed Ahmed}

\begin{document}
\maketitle
 
\begin{abstract} The question posed in the title is more of a philosophical nature. 
Even more,  departing from the materialistic point of view that spirits do not exist, it is 
meta-physical, and so from the positivist point of view it is meaningless.

However, we approach the problem in a completely rigorous mathematical way; 
and we give an answer that to our mind is quite satisfactory using deep concepts in algebraic logic 
wrapped in the language of arrows, better known as category theory.

For a start, the following two questions are investigated for cylindric -like algebras:
\begin{enumarab}
\item Given ordinals $\alpha<\beta$ and an algebra of dimension $\alpha$, 
does it (neatly) embed into the $\alpha$ reduct of a $\beta$ dimensional algebra? 
And if it does, does it neatly embed into the $\alpha$ reduct of a $\beta+k$
dimensional algebra for some $k\geq 1$.

\item Suppose that $\A$ has the neat embedding property, so that $\A$ actually embeds into the neat $\alpha$ reduct
of an algebra $\A$ in $\omega$ extra dimensions, is this last algebra, called a dilation, uniquely determined by $\A$ in some sense? 
\end{enumarab}
For the first question we show that the answer is no for many cylindric like algebras of relations (like quasi-polyadic algebras), 
for both finite and infinite 
dimensions.
We give a categorial answer to the question in the title, encompassing an answer to the second question.
We show that the uniqueness of the minimal dilation obtained when the small algebra generates the dilation, 
depends on the adjointness of the neat embedding 
operator, viewed as a functor. For polyadic algebras the neat reduct functor (that has to do with compressing dimensions) 
is strongly invertible, while for cylindric algebras, and its likes,  it does not 
even have a right adjoint (a functor that stretches dimensions.)
\end{abstract}

\section{Introduction}

In our treatment of (notation and concepts on) neat reducts we follow \cite{Sayed} whose notation is consistent with \cite{HMT1}.
We have only strayed from this principle only when we felt there was a compelling reason, and this happens only once.
We denote the restriction of a function $f$ to a set $X$ by $f\upharpoonright X$ and not the other way round as done in \cite{HMT1}. 
Let $\K$ be a cylindric -like class (like for instance quasi-polyadic algebras), so that for all $\alpha$, $\K_{\alpha}$ is a 
variety consisting of $\alpha$ dimensional algebras, and for
$\alpha<\beta$, $\Nr_{\alpha}\B$ and $\Rd_{\alpha}\B$ are defined, so that the former, the $\alpha$ neat reduct of $\B$,
is a subalgebra of the latter, the $\alpha$ reduct of $\B$ and the latter  is in $\K_{\alpha}$.

We address the following two questions on neat embeddings:
\begin{enumarab}

\item Given ordinals $\alpha<\beta$ and $\A\in \K_{\alpha}$ is there a $\B\in \K_{\beta}$ such that 

(i) $\A$  embeds into the $\alpha$ reduct of $\B$? 

(ii) $\A$ embeds into the $\alpha$ {\it neat} reduct of $\B$?

\item Assume that it does neatly embed into $\B$, and assume further that $\A$ (as a set) generates $\B$, is then $\B$ unique up to isomorphisms 
that fix $\A$ pointwise?
 \end{enumarab}


If $\B$ and $\A$ are like in the second item, then $\B$ is called a {\it minimal dilation} of $\A$.

Our second question is motivated by the following quote of Henkin Monk and Tarski \cite{HMT1}:


Unless specified to the contrary ordinals considered 
are always infinite. 
{\bf It will be shown in Part II that for each $\alpha$, $\beta$
such that $\beta\geq \alpha\geq \omega$ there is a $\CA_{\alpha}$ $\A$ and a $\CA_{\beta}$ $\B$ such that 
$\A$ is a generating subreduct of $\B$ different from $\Nr_{\alpha}\B$; in fact, both $\A$ and $\B$ can be taken to
be representable. Thus $\Dc_{\alpha}$ cannot be replaced by $\CA_{\alpha}$ in Theorem 2.6.67 (ii); 
it is known that this replacement
also cannot be made in certain consequences of 2.6.67, namely 2.6.71 and 2.6.72.}

This result was not proved in \cite{HMT2} as promised, 
but it was proved by the present author with a precursor; a joint  publication with 
Istvan N\'emeti \cite{SL}. The solution is announced in \cite{neat} and presented briefly in \cite{Sayed}.

The main result in \cite{neat} is that minimal dilations for representable algebras that are not dimension complemented
are not unique up to isomorphisms that fix the base algebra pointwise. We will see that this question has an elegant categorial formulation and answer.

Much of the beauty of mathematics, and in particular, category theory,
is that it affords abstraction. Not only does it allow one to see the forest rather than the trees, 
but it also offers the possibility for study of the structure of the entire forest in preparation for the next stage of abstraction 
comparing forests and then, perhaps even,  comparing forests of forests.

Category theory provides an entirely new language, 
a language that provides economy of thought and expression as well as allowing 
easior communication among investigators in different areas, it 
is a language that brings to the forefront the common basic ideas underlying ostensibly unrelated theorems 
and hence a language that gives a new context in which to 
view old problems. In a wider perspective, this new context allows vieweing the old problem in a wider framework, with new insight.

As category theory was not mature enough, at the time part one of the monograph cylindric algebras was published,
at the present time the question raised by Henkin et all can be formulated more succintly using categorial jargon. 
The question, in retrospect, is essentially equivalent to the more elegant and concise question 
as to whether the neat reduct operator viewed (in a  natural way) 
as functor, that {\it compresses} dimensions, has - using categorial jargon - a right adjoint, or using polyadic algebra jargon, if you like, 
a dilation. The answer is no.

Here we investigate the analogous  question for many cylindric -like algebras. We will discover that the 
answer depends essentially on the {\it invertibility} of the neat reduct operator, viewed as a functor.
For cylindric-like algebras this functor is not even weakly invertible but for polyadic-like ones it is, and strongly so. 

The paper intends to replace trips into algebraic territory by the use of 
category theory. Even so, the trade between algebraic logic, logic, and category theory remains interesting, 
even when it is not a matter of applying concrete theorems, but exporting more universal ideas.

Category theory has also the supreme advantage of putting 
many existing results scattered in the literature, in their proper perspectives highlighting interconnections, illuminating differences and similarities, 
despite the increasing tendencies toward fragmentation and specializtion, in mathematical logic  in general, 
and in even more specialized fields like algebraic logic. 

Throughout this article, the high level of abstraction embodied in category theory and  in dealing with highly abstract notions, like 
systems of varieties definable by a schema, is motivated and exemplified by well known concrete examples, 
so that this level of abstraction can be kept from becoming a high 
level of obfuscation.

The categorial approach adopted here is not merely a formal wrapping, on the contrary, 
it is an emphasis that general mathematical sophistication, or essayistic common sense, is the more appropriate road towards insight
than elaborate logical formal systems.
Insights found in category theory really live at some higher generic abstraction level that can often be brought 
out better in an approach originating from algebraic logic, and indeed from 
most concrete techniques available in connection to the notion of representability.

The notion of neat reducts \cite{Sayed} are intimately related to representability (the most important notion in algebraic logic).
Henkin's completeness proof, has radically influenced model theory and for that matter algebraic logic, and more.
This technique which exists almost everywhere, whenever we encounter  completeness or interpolation, or an omitting types setting,
for various predicate logics,
is now called simply {\it a Henkin construction}. It has since unfolded into a 
sophisticated and versatile proof technique, in many branches in (algebraic) logic, and beyond.
Its re-incarnation in algebraic logic has come under the name of the {\it Neat Embedding Theorem}, 
which is an algebraic version of a completeness theorem for certain fairly standard extensions for first order logic that are in some sense more 
basic.
 

On the border line, there are cylindric-like algebras for which the neat reduct  functor is strongly invertible, too.
Some of these are classical in the sense that, like cylindric algebras, their cylindrfiers commute, 
others are obtained by removing the 'Rosser condition' of commutativity
of cylindrifiers, a theme that can be traced back to the Andr\'eka-Resek- Thompson result, 
inspired by Leon Henkin (in analogy to considering two sorted first order logic to provide semantics for second 
order logic) and severly boosted by
Ferenczi, in his recent inspiring  work on 
neat embeddings of non -commutative algebras that are representable only by relativized set algebras, cf. \cite{Fer1},\cite{Fer2}, \cite{Fer3}, 
\cite{Fer4}, \cite{Fer}. 
Such results can be also viewed as a fruitful contact between neat embedding theorems
and relativized representations. The algebraisation process is a powerful strategy, and it works modulo 
modest requirements on the base logic. But as with general models, the conspicuous possibilities lie in between. This typically involves varying the 
'semantic parameter'; this was started by Leon Henkin and his student Resek, and has culminated in incredibly sophisticated representation theorems
\cite{Fer}.

Let us start from the very beginning. The first natural question that can cross one's mind is: 
Is it true that every algebra neatly embeds into another algebra having only one extra dimension?, 
having $k$ extra dimension, $k>1$ ($k$ could be infinite) ? 
And could it possibly happen that an $\alpha$ dimensional algebra 
neatly embeds into $\alpha+k$ dimensions but does not neatly
embed into $\alpha+k+1$ extra dimension? These are all fair questions, and indeed difficult to answer. Such questions have provoked
extensive research that  
have engaged algebraic logicians for years,
and they were all (with the exception of the infinite dimensional case solved here for cylindric algebras using existing finite dimensional constructions) 
settled by the turn of the millenuim after thorough 
dedicated trials, and dozens of publications providing 
partial answers.
We will show that this is indeed the case for finite dimensions $\geq 3$, this is a known result for cylindric algebras due to
Hirsch, Hodkinson and Maddux, as well as for infinite dimensions, 
which will follow from the finite dimensional case using an ingenious lifting argument of Monk's.
The infinite dimensional case was also proved by Robin Hirsch and Sayed Ahmed 
for other algebras (Like Pinter;s substitution algebras) for all dimensions, using the same lifting argument here to pass from the finite
to the transfinite.

To make our argument as general as much as possible, we introduce the new notion of a system of varieties 
of {\it Boolean algebras} with operators definable by a schema. 
It is like the definition of a Monk's schema, except that we integrate finite dimensions, in such a way that
the $\omega$ dimensional case, uniquely determining higher dimensions, is a natural limit of all $n$ dimensional varieties for finite $n$. 
This is crucial for our later investigations.
The definition is general enough to handle our algebras, and narrow enough to prove what we 
need. At the final section we will define a system definable by a schema that covers also polyadic algebras. 

But for the time being as we are only dealing with the cylindric paradigm, 
we slightly generalized Monk's schemas allowing finite dimensions, but not necessarily all,
so that our systems are indexed by all ordinals $\geq m$ and $m$ could be finite. 
(We can allow also proper infinite subsets of $\omega$, but we do not need that much.)
The main advantage in this approcah is that it shows that a lot of results proved for infinite 
dimensions (like non-finite schema axiomatizability of the representable algebras)
really depend on the analogous result proved for every finite dimension starting at a certain finite $n$ which is usually $3$. 
The hard work is done for the finite dimensional case. The rest is a purely syntactical ingenious lifting process invented by Monk.

\begin{definition}
\begin{enumroman}
\item Let $2\leq m\in \omega.$ A finite $m$ type schema is a quadruple $t=(T, \delta, \rho,c)$ such that
$T$ is a set, $\delta$ and $\rho$ maps $T$ into $\omega$, $c\in T$, and $\delta c=\rho c=1$ and $\delta f\leq m$ for all $f\in T$.
\item A type schema as in (i) defines a similarity type $t_{n}$ for each $n\geq m$ as follows. The domain $T_{n}$ of $t_{n}$ is
$$T_{n}=\{(f, k_0,\ldots k_{\delta f-1}): f\in T, k\in {}^{\delta f}n\}.$$
For each $(f, k_0,\ldots k_{\delta f-1})\in T_{n}$ we set $t_{n}(f, k_0\ldots k_{\delta f-1})=\rho f$.
\item A system $(\K_{n}: n\geq m)$ of classes 
of algebras is of type schema $t$ if for each $n\geq m$ $\K_{n}$ is a class of algebras of type 
$t_{n}$.
\end{enumroman}
\end{definition}

\begin{definition} Let $t$ be  a finite $m$ type schema. 
\begin{enumroman}
\item With each $m\leq n\leq \beta$ we associate a language $L_{n}^t$ of type $t_{n}$: for each $f\in T$ and 
$k\in {}^{\delta f}n,$ we have a function symbol $f_{k0,\ldots k(\delta f-1)}$ of rank $\rho f$ 
\item Let $m\leq \beta\leq n$, and let $\eta\in {}^{\beta}n$ be an injection. 
We associate with each term $\tau$ of $L_{\beta}^t$ a term $\eta^+\tau$ of $L_{n}^t$.
For each $\kappa,\omega, \eta^+ v_k=v_k$. if $f\in T, k\in {}^{\delta f}\alpha$, and $\sigma_1\ldots \sigma_{\rho f-1}$ are terms of $L_{\beta}^t$, then
$$\eta^+f_{k(0),\ldots k(\delta f-1)}\sigma_0\ldots \sigma_{\rho f-1}=f_{\eta(k(0)),\ldots \eta(k(\delta f-1))}\eta^+\sigma_0\ldots \eta^+\sigma_{\rho f-1}.$$
Then we associate with each equation $e=\sigma=\tau$ of $L_{\beta}^t$ the equation $\eta^+\sigma=\eta^+\tau$ of $L_{\alpha}^t$, which we denote by 
$\eta^+(e)$.

\item A system $\K=(\K_n: n\geq m)$ of finite $m$ type schema $t$ 
is a complete system of varieties definable by a schema, if there is a system $(\Sigma_n: n\geq m)$ of equations such that $Mod(\Sigma_n)=\K_n$, and 
for $n\leq m<\omega$ if $e\in \Sigma_n$ and
$\rho: n\to m$ is an injection, then $\rho^+e\in \Sigma_m$; $(\K_{\alpha}: \alpha\geq \omega)$ is a system of varieties definable by schemes
and $\Sigma_{\omega}=\bigcup_{n\geq m}\Sigma_n$.
\end{enumroman}
\end{definition}
\begin{definition} 
\begin{enumarab}
\item Let $\alpha, \beta$ be ordinals, $\A\in \K_{\beta}$ and $\rho:\alpha\to \beta$ be an injection. 
We assume for simplity of notation that in addition to cylindrfiers, we have only one unary function symbol $f$ such that $\rho(f)=\delta(f)=1.$ 
(The arity is one, and $f$ has only one index.)
Then $\Rd_{\alpha}^{\rho}\A$ is the $\alpha$ dimensional 
algebra obtained for $\A$ by definining for $i\in \alpha$, $f_i$ by $f^{\B}_{\rho(i)}$. 
$\Rd_{\alpha}\A$ is $\Rd_{\alpha}^{\rho}\A$ when $\rho$ is the inclusion. 

\item As in the first part we assume only the existence of one unary operator with one index.
Let $\A\in \K_{\beta}$, and $x\in A$. The dimension set of $x$, denoted by $\Delta x$, is the set $\Delta x=\{i\in \alpha: c_ix\neq x\}$.
We assume that if $\Delta x\subseteq \alpha$, then $\Delta f(x)\leq \alpha$. 
Then $\Nr_{\alpha}\B$ is the subuniverse of $\Rd_{\alpha}\B$ consisting
only of $\alpha$ dimensional elements.

\item For $K\subseteq \K_{\beta}$ and an injection $\rho:\alpha\to \beta$, 
then $\Rd_{\alpha}^{\rho}K=\{\Rd^{\rho}_{\alpha}\A: \A\in K\}$ and $\Nr_{\alpha}K=\{\Nr_{\alpha}\A: \A\in K\}$
\end{enumarab}
\end{definition}
The class $S\Nr_{\alpha}\K_{\alpha+\omega}$ has special significance since it cincides in the most known cases to 
the class of representable algebras.
In the next theorem, we show how properties that hold for all finite reducts of an infinite dimensional algebra forces it 
to have the neat embedding property.
The proof does not use any properties not formalizable in systems of varieties definable by a schema; 
it consists of non-trivial manipulation of reducts and neat reducts
via ultraproduct constructions, used to `stretch' dimensions.

In the following theorem, we use a very similar argument of lifting to solve problem 2.12 in \cite{HMT1} for infinite dimensions.
So let us warm up by the first lifting argument; for the second, though in essence very similar, will be more involved technically.

\begin{theorem} Let $\A\in \K_{\alpha}$ such that for every finite injective map $\rho$ into $\alpha$, 
and for every  $x,y\in A$, $x\neq y$, there is a function $h$ and $k<\alpha$ such
that $h$ is an endomorphism of $\Rd^{\rho}\A$, $k\in \alpha\sim Rng(\rho)$, $c_k\circ h=h$ and $h(x)\neq h(y)$.
then $\A\in UpS\Nr_{\alpha}\K_{\alpha+\omega}$. 
\end{theorem}
\begin{demo}{Proof}  We first prove that the following holds for any $l<\omega$.
For every $k<\omega$, for every injection $\rho:k\to \alpha$, and every $x, y\in A$, $x\neq y$, 
there exists $\sigma,h$ such that $\sigma:k+l\to \alpha$ is an injection, $\rho\subseteq \sigma$,
$h$ is an endomorphism of $\Rd_k^{\rho}\A$, $c_{\sigma_u}\circ h=h$, whenever $k\leq u\leq k+l$, and $h(x)\neq h(y)$.

We proceed by induction on $l$. This holds trivially for $l=0$, and it is easy to see that it is true for $l=1$
Suppose now that it holds for given $l\geq 1$. Consider $k, \rho$, and $x,y$ satisfying the premisses. By the induction hypothesis there are 
$\sigma, h$, $\sigma:\alpha+l\to \alpha$ an injection, $\rho\subseteq \sigma$, $h$ is an endomorphism of $\Rd_k^{\rho}\A$, $c_{\sigma_u}\circ h=h$ 
whenever  $k\leq u< k+l$, and $h(x)\neq h(y).$ But then there exist $k,v$ such that $k$ is an endomorphism of 
$\Rd_{k+l}^{\sigma}\A$, $v\in \alpha\sim Rg\sigma$, $c_v\circ k=k$ and $k\circ h(x)\neq k\circ h(y)$. 
Let $\sigma'$ be defined by $\sigma'\upharpoonright \alpha+(l+k)=\sigma$ and $\sigma'(k+l)=v$, and let $h'-=\circ h$. It is easy to check that
$\sigma'$ and $h'$ complete the induction step.

We have $h\in Hom(\Rd_k^{\rho}\A, \Nr_{k}\B)$ where $\B=\Rd_{k+l}^{\sigma}\A$. Then $\Rd_k^{\rho}\A\in S\Nr_k\K_{\alpha+\omega}.$
For brevity let $\D=\Rd_k^{\rho}\A$. For each $l<\omega,$ let $\B_l\in \K_{k+l}$ such that $\D\subseteq \Nr_k\B_l$. 
For all such $l$, let $C_l$ be an algebra have the same similarity type as of $\K_{\omega}$ be such $\B_l=\Rd_{k+l}\C_l$.
Let $F$ be a non-principal ultrafilter on $\omega$, and let $\G=\Pi_{\l<\omega}\C_l/F$. Let 
$$G_n=\{\Gamma\cap (\omega\sim n):\Gamma\in F\}.$$
Then for all $\mu<\omega$, we have
\begin{align*}
\Rd_{k+\mu}\G
&=\Pi_{\eta<\omega} \Rd_{k+\mu}\C_{\eta}/F\\
&\cong \Pi_{\mu\leq \eta<\omega}\Rd_{k+\mu}\C_{\eta}/G_{\mu}\\
&=\Pi_{\mu\leq \eta<\omega}\Rd_{\beta+\mu}\B_{\eta}/G_{\mu}.
\end{align*}
We have shown that $\G\in \K_{\omega}$. 
Define $h$ from $\D$ to $\G$, via
$$x\to (x: \eta<\omega)/F$$
Then $h$ is an injective homomorphism from $\D$ into $\Nr_k\G$.
We have $\G\in \K_{\omega}$ We now show that there exists $\B$ in $\K_{\alpha+\omega}$ such that $\D\subseteq \Nr_k\B$.
(This is a typical instance where reducts are used to 'stretch dimensions', not to compress them).
One proceeds inductively, at successor ordinals (like $\omega+1$) as follows. Let $\rho:\omega+1\to \omega$ be an injection such that  $\rho(i)=i$, for each $i\in \alpha$. Then 
$\Rd^{\rho}\G\in \K_{\omega+1}$ and $\G=\Nr_{\omega}{\Rd^{\rho}}\G$. At limits one uses ultraproducts like above.

Thus $\Rd_k^{\rho}\A\subseteq \Nr_k\B$ for some $\B\in \K_{\alpha+\omega}$. Let $\sigma$ be a permutation of 
$\alpha+\omega$ such that $\sigma\upharpoonright k=\rho$ and $\sigma(j)=j$ for all $j\geq \omega$. 
Then
$$\Rd_k^{\rho}\A\subseteq \Nr_k\B=\Nr_k\Rd_{\alpha+\omega}^{\sigma}\Rd_{\alpha+\omega}^{\sigma^{-1}}\B.$$
Then for any $u$ such that $\sigma[k]\subseteq u\subseteq \alpha+\omega$, we have
$$\Nr_k\Rd_{\alpha+\omega}^{\sigma}\Rd_{\alpha+\omega}^{\sigma^{-1}}\B\subseteq \Rd_k^{\sigma|k}\Nr_{uu}\Rd_{\alpha+\omega}^{\sigma^{-1}}\B.$$
Thus $\Rd_k^{\rho}\A\in S\Rd^{\rho}\Nr_{\alpha}\K_{\alpha+\omega},$ and this holds for any injective finite sequence $\rho$.

Let $I$ be the set of all finite one to one sequences with range in $\alpha$.
For $\rho\in I$, let $M_{\rho}=\{\sigma\in I:\rho\subseteq \sigma\}$.
Let $U$ be an ultrafilter of $I$ such that $M_{\rho}\in U$ for every $\rho\in I$. Exists, since $M_{\rho}\cap M_{\sigma}=M_{\rho\cup \sigma}.$
Then for $\rho\in I$, there is $\B_{\rho}\in \K_{\alpha+\omega}$ such that
$\Rd^{\rho}\A\subseteq \Rd^{\rho}\B_{\rho}$. Let $\C=\Pi\B_{\rho}/U$; it is in ${Up}\K_{\alpha+\omega}$.
Define $f:\A\to \Pi\B_{\rho}$ by $f(a)_{\rho}=a$, and finally define $g:\A\to \C$ by $g(a)=f(a)/U$.
Then $g$ is an embedding, and we are done.
\end{demo}

\section{The first question}

Our first question addresses reducts and neat reducts. Handling reducts are usually easier. Problems concerning 
neat reducts tend to be messy, in the positive sense. 

So lets get over with the easy part of reducts. 
The infinite dimensional case follows from the definition of a system of varieties definable by schema, namely, for any such system 
we have $\K_{\alpha}= HSP\Rd^{\rho}\K_{\beta}$
for any pair of infinite ordinals $\alpha<\beta$ and any injection $\rho:\alpha\to \beta$. 
But for all algebras considered $S\Rd_{\alpha}\K_{\beta}$ is a variety, hence the
desired conclusion; which is that every algebra is a subreduct of an algebra in any preassigned higher dimension.
We can strengthen this to:

\begin{theorem} For any pair of infinite ordinals $\alpha<\beta$, we have $\K_{\alpha}=El\Rd_{\alpha}\K_{\beta}$
\end{theorem}
\begin{proof} For simplicity we asume that we have one unary opeartion $f$, with $\rho(f)=1$.
The general case is the same.
Let $\A\in \K_{\alpha}$. Let $I=\{\Gamma: \Gamma\subseteq \beta, |\Gamma|<\omega\}$. 
Let $I_{\Gamma}=\{\Delta \subseteq I, \Gamma\subseteq \Delta\}$, and let $F$ be an ultrafilter such that
$I_{\Gamma}\in F$ for all $\Gamma\in I$. Notice that $I_{\Gamma_1}\cap I_{\Gamma_2}=I_{\Gamma_1\cup \Gamma_2}$ 
so this ultrafilter exists.
For each $\Gamma\in I$, let $\rho(\Gamma)$ be an injection from $\Gamma$ into $\alpha$ 
such that $Id\upharpoonright \Gamma\cap \alpha\subseteq \rho(\Gamma)$, and let $\B_{\Gamma}$ be an algebra having same similarity type 
as $\K_{\beta}$such that for $k\in \Gamma$
$f_k^{\B_{\Gamma}}= f_{\rho(\Gamma)[k]}^{\A}$. Then 
$\D=\prod \B_{\Gamma}/F\in \K_{\beta}$ and 
$f:\A\to \Rd_{\alpha}\D$ defined via $a\mapsto (a: \Gamma\in I)/F$
is an elementary embedding.
\end{proof}

Things are different for finite dimensions. Here we give an example for quasi-polyadic equality algebras, modelled 
on a construction of Henkin for cylindric algebras reported in \cite{HMT1}. The construction essentially depends on
the presence of diagonal elements. We do not know whether an analagous result hold for quasi-polyadic algebras.

\begin{example}\label{reduct}

Let $n\geq 2$ and $\A\in \QEA_n$. Let $e$ be the equation defined in lemma 2.6.10, in \cite{HMT2}.
Lemma 2.6.13, provides a cylindric algebr $\C$  in $\Rd_{\alpha}\CA_{\beta}$ for any finite $\beta$, 
such that $\C$ is generated by a set with cardinality $\beta$, and
$e$ fails in this algebra. 

Now this algebra, is based on the algebra constructed in lemma 2.6.12; so we need to define substitutions on this last algebra, 
which is a product of the Boolean  part two cylindric algebras.
One simply sets $p_{ij}(x,y)=(p_{ij}z, p_{ij}y)$; but then the algebra 
$\C$ is in $\Rd_{\alpha}PEA_{\beta}$. 

The proofs of theorems 2.6.14 and 2.6.16, work verbatim by replacing cylindric algebras with polyadic algebras.
And so we have the proper inclusions, for $2\leq \alpha<\beta\leq \omega$:
$$HSP\Rd_{\alpha}\QEA_{\beta+1}\subset HSP\Rd_{\alpha}\QEA_{\beta}\subset \QEA_{\alpha}$$
(The inclusion follows from the fact that a reduct of a reduct is a reduct).
\end{example}

The neat reduct part, as we shall see, is profoundly more involved. The following is our  main result in this section.
It is not the case that every algebra in $\CA_m$ is the neat reduct of an algebra in $\CA_n$, 
nor need it even be a subalgebra of a neat reduct of an algebra in $\CA_n$.  Furthermore, $S\Nr_m\CA_{m+k+1}\neq S\Nr_m\CA_m$, 
whenever $3\leq m<\omega$ and $k<\omega$. 

The hypothesis in the following theorem presupposes the existence of certain finite
dimensional algebras, not chosen haphazardly at all, but are  rather an abstraction of cylindric algebras existing in the literature witnessing the last 
proper inclusions.  The main idea, that leads to the conclusion of the theorem, 
is to use such finite dimensional algebras to obtain an an analogous result for the infinite dimensional case.
Accordingly, we streamline Monk's argument who did exactly that for cylindric algebras, but we do it in 
the wider context of systems of varieties definable by a schema. (Strictly speaking Monk's lifting argument is weaker, 
the infinite dimensional constructed algebras are merely non -representable, in our case they are not only non-representable, 
but are also subneat reducts of  algebras in a given pre 
asighned dimension; this is a technical difference, that needs some non-trivial fine tunning in the proof).
The inclusion of finite dimensions in our formulation, was therefore not a luxuary, nor was it motivated by 
aesthetic reasons, and nor was it merely an artefect of Monk's definition. 
It is motivated by the academic worthiness of the result 
(for infinite dimensions).

\begin{theorem}\label{2.12} Let $(\K_{\alpha}: \alpha\geq 2)$ be a complete system of varieties definable by a schema.
Assume that for $3\leq m<n<\omega$, 
there is $m$ dimensional  algebra $\C(m,n,r)$ such that
\begin{enumarab}
\item $\C(m,n,r)\in S\Nr_m\K_n$
\item $\C(m,n,r)\notin S\Nr_m\K_{n+1}$
\item $\prod_{r\in \omega} \C(m, n,r)\in S\Nr_m\K_n$
\item For $m<n$ and $k\geq 1$, there exists $x_n\in \C(n,n+k,r)$ such that $\C(m,m+k,r)\cong \Rl_{x}\C(n, n+k, r).$
\end{enumarab}
Then for any ordinal $\alpha\geq \omega$, $S\Nr_{\alpha}\K_{\alpha+k+1}$ is not axiomatizable by a finite schema over $S\Nr_{\alpha}\K_{\alpha+k}$
\end{theorem}

\begin{proof} The proof is a lifting argument essentially due to Monk, by 'stretching' dimensions  using only properties of 
reducts  and ultraproducts, formalizable in the context of a system of varieties definable by a schema.

It is divided into 3 parts:

\begin{enumarab}

\item  Let $\alpha$ be an infinite ordinal, 
let $X$ be any finite subset of $\alpha$, let $I=\set{\Gamma:X\subseteq\Gamma\subseteq\alpha,\; |\Gamma|<\omega}$.  
For each $\Gamma\in I$ let $M_\Gamma=\set{\Delta\in I:\Delta\supseteq\Gamma}$ and let $F$ be any ultrafilter over $I$ 
such that for all $\Gamma\in I$ we have $M_\Gamma\in F$ 
(such an ultrafilter exists because $M_{\Gamma_1}\cap M_{\Gamma_2} = M_{\Gamma_1\cup\Gamma_2}$).  
For each $\Gamma\in I$ let $\rho_\Gamma$ be a bijection from $|\Gamma|$ onto $\Gamma$.   
For each $\Gamma\in I$ let $\c A_\Gamma, \c B_\Gamma$ be $\K_\alpha$-type algebras.  
If for each $\Gamma\in I$ we have 
$\Rd^{\rho_\Gamma}\c A_\Gamma=\Rd^{\rho_\Gamma}\c B_\Gamma$ then $\Pi_{\Gamma/F}\c A_\Gamma=\Pi_{\Gamma/F}\c B_\Gamma$.  
Standard proof, by \Los' theorem.  
Note that the base of $\Pi_{\Gamma/F}\c A_\Gamma$ is 
identical with the base of $\Pi_{\Gamma/F}\Rd^{\rho_\Gamma}\c A_\rho$ 
which is identical with the base of $\Pi_{\Gamma/F}\c B_\Gamma$, by the assumption in the lemma.  
Each operator $o$ of $\K_\alpha$ is the same for both ultraproducts because $\set{\Gamma\in I:\dim(o)\subseteq\rng(\rho_\Gamma)} \in F$.  

Furthermore, if $\Rd^{\rho_\Gamma}\c A_\Gamma \in \K_{|\Gamma|}$, for each $\Gamma\in I$  then $\Pi_{\Gamma/F}\c A_\Gamma\in \K_\alpha$.
For this, it suffices to prove that each of the defining axioms for $\K_\alpha$ holds for $\Pi_{\Gamma/F}\c A_\Gamma$.  
Let $\sigma=\tau$ be one of the defining equations for $\K_{\alpha}$, the number of dimension variables is finite, say $n$. 
Take any $i_0, i_1,\ldots  i_{n-1}\in\alpha$, we must prove that 
$\Pi_{\Gamma/F}\c A_\Gamma\models \sigma(i_0,\ldots i_{n-1})=\tau(i_0\ldots  i_{n-1})$.  
If they are all in $\rng(\rho_\Gamma)$, say $i_0=\rho_\Gamma(j_0), \; i_1=\rho_\Gamma(j_1), \;\ldots i_{n-1}=\rho_\Gamma(j_{n-1})$,  
then $\Rd^{\rho_\Gamma}\c A_\Gamma\models \sigma(j_0, \ldots ,j_{n-1})=\tau(j_0, \ldots j_{n-1})$, 
since $\Rd^{\rho_\Gamma}\c A_\Gamma\in\K_{|\Gamma|}$, so $\c A_\Gamma\models\sigma(i_0\ldots , i_{n-1})=\tau(i_0\ldots i_{n-1}$.  
Hence $\set{\Gamma\in I:\c A_\Gamma\models\sigma(i_0, \ldots, i_{n-1}l)=\tau(i_0, \ldots,  i_{n-1})}\supseteq\set{\Gamma\in I:i_0,\ldots,  i_{n-1}
\in\rng(\rho_\Gamma}\in F$, 
hence $\Pi_{\Gamma/F}\c A_\Gamma\models\sigma(i_0,\ldots  i_{n-1})=\tau(i_0, \ldots,  i_{n-1})$.  
Thus $\Pi_{\Gamma/F}\c A_\Gamma\in\K_\alpha$.

\item Let $k\in \omega$. Let $\alpha$ be an infinite ordinal. 
Then $S\Nr_{\alpha}\K_{\alpha+k+1}\subset S\Nr_{\alpha}\K_{\alpha+k}.$
Let $r\in \omega$. 
Let $I=\{\Gamma: \Gamma\subseteq \alpha,  |\Gamma|<\omega\}$. 
For each $\Gamma\in I$, let $M_{\Gamma}=\{\Delta\in I: \Gamma\subseteq \Delta\}$, 
and let $F$ be an ultrafilter on $I$ such that $\forall\Gamma\in I,\; M_{\Gamma}\in F$. 
For each $\Gamma\in I$, let $\rho_{\Gamma}$ 
be a one to one function from $|\Gamma|$ onto $\Gamma.$
Let ${\c C}_{\Gamma}^r$ be an algebra similar to $\K_{\alpha}$ such that 
\[\Rd^{\rho_\Gamma}{\c C}_{\Gamma}^r={\c C}(|\Gamma|, |\Gamma|+k,r).\]
Let  
\[\B^r=\prod_{\Gamma/F\in I}\c C_{\Gamma}^r.\]
We will prove that 
\begin{enumerate}
\item\label{en:1} $\B^r\in S\Nr_\alpha\K_{\alpha+k}$ and 
\item\label{en:2} $\B^r\not\in S\Nr_\alpha\K_{\alpha+k+1}$.  \end{enumerate}

The theorem will follow, since $\Rd_\K\B^r\in S\Nr_\alpha \K_{\alpha+k} \setminus S\Nr_\alpha\K_{\alpha+k+1}$.

For the first part, for each $\Gamma\in I$ we know that $\c C(|\Gamma|+k, |\Gamma|+k, r) \in\K_{|\Gamma|+k}$ and 
$\Nr_{|\Gamma|}\c C(|\Gamma|+k, |\Gamma|+k, r)\cong\c C(|\Gamma|, |\Gamma|+k, r)$.
Let $\sigma_{\Gamma}$ be a one to one function 
 $(|\Gamma|+k)\rightarrow(\alpha+k)$ such that $\rho_{\Gamma}\subseteq \sigma_{\Gamma}$
and $\sigma_{\Gamma}(|\Gamma|+i)=\alpha+i$ for every $i<k$. Let $\c A_{\Gamma}$ be an algebra similar to a 
$\K_{\alpha+k}$ such that 
$\Rd^{\sigma_\Gamma}\c A_{\Gamma}=\c C(|\Gamma|+k, |\Gamma|+k, r)$.  By the second part   
with  $\alpha+k$ in place of $\alpha$,\/ $m\cup \set{\alpha+i:i<k}$ 
in place of $X$,\/ $\set{\Gamma\subseteq \alpha+k: |\Gamma|<\omega,\;  X\subseteq\Gamma}$ 
in place of $I$, and with $\sigma_\Gamma$ in place of $\rho_\Gamma$, we know that  $\Pi_{\Gamma/F}\A_{\Gamma}\in \K_{\alpha+k}$.

We prove that $\B^r\subseteq \Nr_\alpha\Pi_{\Gamma/F}\c A_\Gamma$.  Recall that $\B^r=\Pi_{\Gamma/F}\c C^r_\Gamma$ and note 
that $C^r_{\Gamma}\subseteq A_{\Gamma}$ 
(the base of $C^r_\Gamma$ is $\c C(|\Gamma|, |\Gamma|+k, r)$, the base of $A_\Gamma$ is $\c C(|\Gamma|+k, |\Gamma|+k, r)$).
 So, for each $\Gamma\in I$,
\begin{align*}
\Rd^{\rho_{\Gamma}}\C_{\Gamma}^r&=\c C((|\Gamma|, |\Gamma|+k, r)\\
&\cong\Nr_{|\Gamma|}\c C(|\Gamma|+k, |\Gamma|+k, r)\\
&=\Nr_{|\Gamma|}\Rd^{\sigma_{\Gamma}}\A_{\Gamma}\\
&=\Rd^{\sigma_\Gamma}\Nr_\Gamma\A_\Gamma\\
&=\Rd^{\rho_\Gamma}\Nr_\Gamma\A_\Gamma
\end{align*}
By the first part of the first part we deduce that 
$\Pi_{\Gamma/F}\C^r_\Gamma\cong\Pi_{\Gamma/F}\Nr_\Gamma\A_\Gamma\subseteq\Nr_\alpha\Pi_{\Gamma/F}\A_\Gamma$,
proving \eqref{en:1}.

Now we prove \eqref{en:2}.
For this assume, seeking a contradiction, that $\B^r\in S\Nr_{\alpha}\K_{\alpha+k+1}$, 
$\B^r\subseteq \Nr_{\alpha}\c C$, where  $\c C\in \K_{\alpha+k+1}$.  
Let $3\leq m<\omega$ and  $\lambda:m+k+1\rightarrow \alpha +k+1$ be the function defined by $\lambda(i)=i$ for $i<m$ 
and $\lambda(m+i)=\alpha+i$ for $i<k+1$.
Then $\Rd^\lambda(\c C)\in \K_{m+k+1}$ and $\Rd_m\B^r\subseteq \Nr_m\Rd^\lambda(\c C)$.
For each $\Gamma\in I$,\/  let $I_{|\Gamma|}$ be an isomorphism 
\[{\c C}(m,m+k,r)\cong \Rl_{x_{|\Gamma|}}\Rd_m {\c C}(|\Gamma|, |\Gamma+k|,r).\]
Let $x=(x_{|\Gamma|}:\Gamma)/F$ and let $\iota( b)=(I_{|\Gamma|}b: \Gamma)/F$ for  $b\in \c C(m,m+k,r)$. 
Then $\iota$ is an isomorphism from $\c C(m, m+k,r)$ into $\Rl_x\Rd_m\B^r$. 
Then $\Rl_x\Rd_{m}\B^r\in S\Nr_m\K_{m+k+1}$. 
It follows that  $\c C (m,m+k,r)\in S\Nr_{m}\K_{m+k+1}$ which is a contradiction and we are done.
\end{enumarab}
\end{proof}
 
\subsubsection{Monk's algebras}

Monk's seminal result proved in 1969, showing that the class of representable cylindric algebras is not finitely axiomtizable had a 
shatterring effect on  algebraic logic, in many respects. The conclusions drawn from this result, were that either the 
extra non-Boolean basic operations of cylindrifiers and diagonal elements were not properly chosen, or that the notion of 
representability was inappropriate; for sure it was concrete enough, but perhaps this is precisely the reason, it is far {\it too concrete.}

Research following both paths, either by changing the signature or/and altering the notion of concrete representability have been pursued 
ever since, with amazing success.  Indeed there are  two conflicting but complementary facets 
of such already extensive  research referred to in the literature, as 'attacking the representation problem'.
One is to delve deeply in investigating the complexity of potential axiomatizations for existing varieties 
of representable algebras, the other is to try to sidestep 
such wild unruly complex axiomatizations, often referred to as {\it taming methods}. 

Those taming methods can either involve passing to (better behaved) expansions of the algebras considered, 
or even completely change the signature  bearing in mind that the essential operations like cylindrifiers are 
term definable or else change the very  notion of representatiblity involved, as long as it remains concrete enough.

The borderlines are difficult to draw, we might not know what is {\it not} concrete enough, but we can
judge that a given representability notion is satisfactory, once we have one. 

One can find well motivated appropriate notions of semantics by first locating them while giving up classical semantical prejudices. 
It is hard to give a precise mathematical underpinning to such intuitions. What really counts at the end of the day
is a completeness theorem stating a natural fit between chosen intuitive concrete-enough, but not too concrete, 
semantics and well behaved axiomatizations. 
The move of altering semantics has radical phiosophical repercussions, taking us away from the conventional 
Tarskian semantics captured by Fregean-Godel-like axiomatization; the latter completeness proof is effective 
but highly undecidable; and this property is inherited by finite varibale fragments of first order logic as long as we insist on Tarskian semantics.

Monk defined the required algebras, witnessing the non finite axiomtizability of $\RCA_n$ $n\geq 3$,
via their atom structure. An $n$ dimensional atom structure is a triple 
$\G=(G, T_i, E_{ij})_{i,j\in n}$
such that $T_i\subseteq G\times G$ and $E_{ij}\subseteq G$, for all $i, j\in n$. An atom structure so defined, is a cylindric atom structure if 
its complex algebra $\Ca\G\in \CA_n$. $\Ca\C$ is the algebra 
$$(\wp(G), \cap, \sim T_i^*, E_{ij}^*)_{i,j\in n},$$ where
$$T_i^*(X)=\{a\in G: \exists b\in X: (a,b)\in T_i\}$$
and
$$E_{i,j}^*=E_{i,j}.$$
Cylindric algebras are axiomatized by so-called Sahlqvist equations, and therefore it is easy to spell out first order correspondants
to such equations characterizing 
atom structures of cylindric algebras.

\begin{definition}
For $3 \leq m\leq n < \omega$, ${\G}_{m, n}$ denotes
the cylindric atom structure such that ${\G}_{m, n} = (G_{m, n},
T_i, E_{i,j})_{i, j < m} $ of dimension
$m$ which is defined as follows:
$G_{m, n}$ consists of all pairs $(R, f)$ satisfying
the following conditions:
\begin{enumarab}
\item $R$ is equivalence relation on $m$,
\item $f$ maps $\{ (\kappa, \lambda) : \kappa, \lambda < n,
\kappa  \not{R} \lambda\}$ into $n$,
\item for all $\kappa, \lambda < m$, if $\kappa \not{R} \lambda$
then $f_{\kappa \lambda } = f_{\lambda \kappa}$,
\item for all $\kappa, \lambda, \mu < m$, if $\kappa \not{R}
\lambda R \mu$ then $f_{\kappa \lambda } = f_{\kappa \mu}$,
\item for all $\kappa, \lambda, \mu < n$, if $\kappa \not{R}
\lambda \not{R} \mu \not{R} \kappa$ then $|f_{\kappa \lambda },
f_{\kappa \mu}, f_{\lambda \mu}| \neq 1.$
\end{enumarab}
For $\kappa < m$ and $(R, f), (S, g) \in G(m,n)$ we define
\begin{eqnarray*}
&(R, f) T_\kappa (S, g) ~~ \textrm{iff} ~~ R \cap {}^2(n
\smallsetminus
\{\kappa\}) = S \cap {}^2(m \smallsetminus \{\kappa\}) \\
& \textrm{and for all} ~~ \lambda, \mu \in m \smallsetminus
\{\kappa\}, ~~ \textrm{if} ~~ \lambda \not{R} \mu~~ \textrm{then} ~~
f_{\lambda \mu } = g_{\lambda \mu}.
\end{eqnarray*}
For any $ \kappa, \lambda <m$, set
$$ E_{\kappa \lambda} = \{ ( R, f) \in G(m,n) : \kappa R \lambda \}.$$
\end{definition}
Monk proves that this indeed defines a cylindric atom structure, he defines
the $m$ dimensional cylindric algebra $\C(m,n)=\Ca(\G(m,n),$ then he proves: 

\begin{theorem} 
\begin{enumarab}
\item For $3\leq m\leq n<\omega$ and $n-1\leq \mu< \omega$, $\Nr_m\C(n,\mu)\cong \C(m,\mu)$.
In particular, $\C(m, m+k)\cong \Nr_m(\C(n, n+k)$. 
\item Let $x_n=\{(R,f)\in G_{n, n+k}; R=(R\cap ^2n) \cup (Id\upharpoonright {}^2(n\sim m))\\
\text { for all $u, v$,} uRv, f(u,v)\in n+k,
\text { and 
for all } \mu\in n\sim m, v<\mu,\\ f(\mu, v)=\mu+k\}.$ 

Then 
$\C(n, n+k)\cong \Rl_x\Rd_n\C(m, m+k).$
\end{enumarab}
\end{theorem}
\begin{demo}{Proof} \cite{HMT2}, theorems 3.2.77 and 3.2.86.
\end{demo}
\begin{theorem} The class $\RCA_{\alpha}$ is not axiomatized by a finite schema.
\end{theorem}
\begin{demo}{Proof} By $\RCA_{\alpha}=S\Nr_{\alpha}\CA_{\alpha+\omega}.$ Let $r\in \omega$. Then $\B^r$, call it $\B_k$ constructed above, 
from the finite dimensional algebras increasing in dimension, is in 
$S\Nr_{\alpha}\CA_{\alpha+k}$ but it is not in $S\Nr_{\alpha}\CA_{\alpha+k+1}$ least representable. 
Then the ultraproduct of the $\B_k$'s over a non-principal ultrafilter will be in 
$S\Nr_{\alpha}\CA_{\alpha+\omega},$ hence will be representable.
\end{demo}

Johnsson defined a polyadic atom structure based on the $\G_{m,n}$. First a helpful piece of notation: 
For relations $R$ and $G$, $R\circ G$ is the relation 
$$\{(a,b): \exists c (a,c)\in R, (c, b)\in S\}.$$
Now Johnson extended the atom structure $\G(m,n)$ by

$(R,f)\equiv_{ij}(S,g)$ iff $f(i,j)=g(j,i)$ and if $(i,j)\in R$, then $R=S$, if not, then $R=S\circ [i,j]$, as composition of relations. 
 
Strictly speaking, Johnsson did not define substitutions quite in this way; because he has all finite transformations, not only transpositions.
Then, quasipolyadic algebras was not formulated in schematizable form, a task accomplished by Sain and Thompson \cite{ST} much later.

\begin{theorem}(Sain-Thompson) $\RQA_{\alpha}$ and $\RQEA_{\alpha}$ is not finite schema axiomatizable
\end{theorem}
\begin{demo}{Proof} One proof uses the fact that $\RQA_{\alpha}=S\Nr_{\alpha}\QA_{\alpha+\omega}$, and that the diagonal free reduct
Monk's algebras (hence their infinite dilations) are not representable. Another proof uses a result of Robin Hirsch and Tarek Sayed 
Ahmed that there exists finite dimensional quasipolyadic algebras satisfying the hypothesis of theorem \ref{2.12}. 
A completely analogous result holds for Pinters algebras, using also finite dimensional Pinters algebras satisfying the hypothesis
of theorem \ref{2.12}.
\end{demo}

\section{The methods of splitting applied to quasi-polyadic equality algebras}

More severe negative results on potential universal axiomatizations of cylindric and quasi polyadic equality were obtained by Andr\'eka and 
Sayed Ahmed, we give one in what follows. Such results use a  different technique called splitting, although there are similarities with Monk's ideas.

The idea, traced back to Jonsson for relation algebras, consists of constructing for every finite $k\in \omega$ 
a non-representable algebra, all of whose $k$ -generated subalgebras are representable.

Andr\'eka ingeniously transferred such an idea to cylindric algebras, and to fully implement it, she invented
the nut cracker method of splitting. 
The subtle splitting technique invented by Andr\'eka can be summarized as follows. 
In the presence of only finitely many substitutions, we take a fairly simple representable algebra generated by an atom, and we break 
up or {\it split} the atom into enough (finitely many)  $k$ atoms,
forming a larger algebra, that is in fact non-representable; in fact, its cylindric reduct will not be representable, due to the  incompatibility 
between the number of atoms, and the number of elements in the domain of a representation. However, the ''small" subalgebras 
namely, those generated by $k$ elements of such an algebra will be 
representable.

This does have affinity to Monk's construction witnessing non finite axiomatizability for 
the class of representable cylindric algebras.
The key idea of the construction of a Monk's algebra is not so hard. 
Such algebras are finite, hence atomic, more precisely their Boolean reducts are atomic.
The atoms are given colours, and 
cylindrifications and diagonals are defined by stating that monochromatic triangles 
are inconsistent. If a Monk's algebra has many more atoms than colours, 
it follows from Ramsey's Theorem that any representation
of the algebra must contain a monochromatic traingle, so the algebra
is not representable.

For $\CA$'s, for each $k$ only one splitting into $k$ atoms are required, as done by Andr\'eka, 
For $\RQEA$, things are more complicated, one has to perform infinitely many finite splittings 
(that is into a pre assigned finite $k$), one for every reduct containing only finitely many 
substitutions.  (not just one which is done in \cite{ST}; relative though to infinitely many atoms which is much more than needed), 
increasing in number but always finite, 
constructing infinitely many algebras, whose similarity types contain only finitely many 
substitutions. Such constructed non-representable algebras, 
form a chain, and our desired algebra will be their 
directed union. The easy thing to do is to show that ``small" subalgebras of every non-representable 
algebra in the chain is representable; the hard thing to do is to show that ``small" subalgebras of the non-representable 
limit remain 
representable. (The error in Sain's Thompson paper is claiming that the small subalgebras 
of the non-representable algebra, obtained by performing only one splitting into infinitely many atoms, are representable; this is not necessarily 
true).

The cylindric reduct of the algebras forming the chain is of $\CA_{\omega}$ type; in particular, it contains infinitely many 
cylindrifications and diagonal elements.
The  combinatorial argument of counting depends essentially on the presence of infinitely many diagonal elements.
Indeed, it can be shown that the splitting 
technique adopted to prove complexity results concerning axiomatizations of $\RQEA_{\omega}$ 
simply does not work in the absence of diagonals. This can be easily destilled from our proof since our constructed 
non-representable  quasipolyadic equality algebras, in fact have a representable 
quasipolyadic reduct. 
An open problem here, that can be traced back to to Sain's and Thompson's paper \cite{ST}, 
is whether $\RQA_{\omega}$ {\it can be} 
axiomatized by a necessarily infinite) set of formulas using only finitely many 
variables. This seems to be a hard problem, and the author tends to believe 
that there are axiomatizations that contain only finitely many variables, but further research is needed in this area.

On the other hand, the algebra constructed by this method of splitting is `almost representable', in the sense that if we 
enlarge the potential domain of a representation, then various reducts of the algebra, obtained by discarding some of the operations 
(for example diagonal elements or infinitely many cylindrifications), turn out 
representable; and this gives {\it relative} non-finitizability results.
Here we are encountered by a situation where we cannot have our cake and eat. If we want a quasipolyadic equality 
algebras that is only barely representable,
then we cannot obtain non-representability of some of its strict reducts like its quasipolyadic reduct.
Throughout, we will be tacitly assuming that quasipolyadic (equality) algebras are not only 
term-definitionally with finitary polyadic (equality) algebras as proved in \cite{ST} p.546, but that they are actually
the same. This means that in certain places we consider only substitutions 
corresponding to transpositions rather then all substitutions corresponding to finite transformations which is perfectly 
legitimate. Also we understand representability of reducts of quasipolyadic equality algebras, when we discard some of the substitution operations, 
in the obvious sense. 


If $\A$ has a cylindric reduct, then $\Rd_{ca}\A\in \RCA_{\alpha}$ denotes this reduct. 
Our next theorem corrects the error mentioned above in Sain's Thompson's seminal paper \cite{ST},
generalizes Theorem 6 in \cite{Andreka} p. 193 to infinitely many dimensions, 
and answers a question by Andre\'ka in op cit also on p. 193. 

\begin{theorem} The variety $\RQEA_{\omega}$ cannot be axiomatized with with a set $\Sigma$ of quantifier free
formulas containing finitely many variables. In fact, for any $k<\omega$, and any set of quantifier free formulas 
$\Sigma$ axiomatizing $\RQEA_{\omega}$, $\Sigma$ contains a formula with more than $k$ variables in which some diagonal 
element occurs.
\end{theorem}

\begin{demo}{Proof}  The proof consists of two parts.  In the first part we construct algebras $\A_{k,n}$ with certain properties, for each 
$n,k\in \omega\sim \{0\}$. In the second part we form a limit of such algebras as $n$ tends to infinity, obtaining an algebra $\A_k$ that is not representable, though its 
$k$-generated subalgebras are representable. This algebra wil finish the proof.

{\bf Part I}

Let $k, n\in \omega\sim \{0\}$.
Let $G_n$ be the symmetric group on $n$.
$G_n$ is generated by the set of all transpositions $\{[i,j]: i,j\in n\}$ and for $n\leq m$, we can consider $G_n\subseteq G_m$.
We shall construct an algebra $\A_{k,n}=(A_{k,n}, +, \cdot ,-, {\sf c}_i, {\sf s}_{\tau}, {\sf d}_{ij})_{i,j\in \omega, \tau\in G_n}$ 
with the following properties.
\begin{enumroman}
\item $\Rd_{ca}\A_{k,n}\notin \RCA_{\omega}$. 
\item Every $k$-generated subalgebra of $\A_{k,n}$ is representable.
\item 
There is a one to one mapping $h:\A_{k,n}\to (\B(^{\omega}W), {\sf c}_i, {\sf s}_{\tau}, {\sf d}_{ij})_{i,j<\omega, \tau\in G_n}$ 
such that $h$ is a homomorphism with respect to all operations of $\A_{k,n}$ except for the 
diagonal elements. 
\end{enumroman}

Here $k$-generated means generated by $k$ elements.
The proof for finite reducts uses arguments very similar to the proof of Andr\'eka of Theorem 6 in \cite{Andreka}, and has affinity  with the proof 
of theorem 3.1 in \cite{c}. 
However, there are two major differences.
Our cylindric reducts are infinite dimensional, and our proof is more direct and, in fact, far easier to grasp. 
The proof of the above cited theorem of Andr\'eka's goes through the route of certain finite expansions by so-called 
permutation invariant unary operations that are also modalities (distributive over the boolean join), and these are more general 
than substitutions. Substitutions are more concrete, and therefore our proof is less abstract.
\begin{enumarab}
\item Let $m\geq 2^{k.n!+1}$, $m<\omega$ 
and let $\langle U_i:i<\omega\rangle$ be a system
of disjoint sets such that $|U_i|=m$ for $i\geq 0$ and $U_0=\{0,\ldots m-1\}$.
Let $$U=\bigcup\{U_i: i\in \omega\},$$
let $$R=\prod _{i<\omega}U_i=\{s\in {}^{\omega}U: s_i\in U_i\},$$
and let $\A'$ be the subalgebra of 
$\langle \B({}^{\omega}U), {\sf c}_i, {\sf d}_{ij}, {\sf s}_{\tau}\rangle_{\tau\in G_n}$
generated by $R$.
Then ${\sf s}_{\tau}R$ is an atom of $\A'$ for any $\tau\in G_n$. Indeed for any two sequences $s,z\in R$ there is a 
permutation $\sigma:U\to U$ of $U$ taking $s$ to $z$ and fixing $R$, i.e
$\sigma\circ s=z$ and $R=\{\sigma\circ p:p\in R\}$. 
$\sigma$ fixes all the elements generated by $R$ because the operations 
are permutation invariant. Thus if $a\in A'$ and $s\in a\cap R$ then $R\subseteq a$
showing that $R$ is an atom of $\A'$. Since $\tau$ is a bijection, it follows that ${\sf s}_{\tau}R$ is also an atom of $\A$
and, it is easy to see that  all these atoms are pairwise disjoint.
That is if $\tau_1\neq \tau_2$, then ${\sf s}_{\tau_1}R\cap {\sf s}_{\tau_2}R=\emptyset$.
We now split each ${\sf s}_{\tau}R$ into abstract atoms ${\sf s}_{\tau}R_j$, $j\leq m$ and $\tau\in G_n$.
Let $(R_j:j\leq m)$ be a set of $m+1$ distinct elements, and let $\A_{k,n}$ be an algebra
such that
\begin{enumerate}
\item $\A'\subseteq \A_{k,n},$ the Boolean part of $\A_{k,n}$ is a Boolean algebra,
\item $R=\sum\{R_j:j\leq m\},$
\item ${\sf s}_{\tau}R_j$ are pairwise distinct atoms of $\A_k$ for each $\tau\in G_n$ and $j\leq m$ and 
${\sf c}_i{\sf s}_{\tau}R_j={\sf c}_i{\sf s}_{\tau}R$ for all $i<\omega$ and all $\tau\in G_n,$
\item each element of $\A_{k,n}$ is a join of element of $\A'$ and of some ${\sf s}_{\tau}R_j$'s,
\item ${\sf c}_i$ distributes over joins,
\item The ${\sf s}_{\tau}$'s are Boolean endomorphisms such that ${\sf s}_{\tau}{\sf s}_{\sigma}a={\sf s}_{\tau\circ \sigma}a$.
\end{enumerate}

The existence of such algebra is easy to show; furthermore they are unique up to isomorphim, see \cite{Andreka}, the comment right 
after the 
definition on p.168. Now we show that $\Rd_{ca}\A_{k,n}$ cannot be representable. This part of the proof is identical to 
Andr\'eka's proof but we include it for the sake of 
completeness. The idea is that we split $R$ into $m+1$ distinct atoms but $U_0$ has only $m$ elements, and those two 
conditions are incompatible in case there is a 
representation. The substitutions have to do with permuting the atoms and they do not contribute to this part of the proof.
For $i,j<\omega$, $i\neq j,$ recall that ${\sf s}_j^ix={\sf c}_i({\sf d}_{ij}\cdot x)$. Let 
$$\tau(x)=\prod_{i\leq m} {\sf s}_i^0{\sf c}_1\ldots {\sf c}_mx\cdot \prod_{i<j\leq m} 
-{\sf d}_{ij}$$
Then $\A'\models \tau(R)=0.$
Indeed we have  
$${\sf c}_1\ldots {\sf c}_mR={}^mU\times U_{m+1}\times \ldots $$
$${\sf s}_i^0{\sf c}_1\ldots {\sf c}_mR=U\times \ldots U_0\times 
U\times U_{m+1}\ldots $$
$$\bigcap {\sf s}_i^0{\sf c}_1\ldots {\sf c}_mR={}^{m+1}U_0\times U_{m+1}\times .$$
Then by $|U_0|\leq m$ there is no repitition free sequence in $^{m+1}U_0$. Thus
as claimed $\A'\models \tau(R)=0$.
Then $\A_{k,n}\models \tau(R)=0$. Assume that $\A_{k,n}$ is represented somehow.
Then there is a homomorphism 
$h:\A_{k,n}\to \langle \B(^{\omega}W), {\sf c}_i, {\sf d}_{ij}\rangle_{i,j<\omega}$ 
for some set $W$ such that $h(R)\neq \emptyset$.

By $h(R)\neq \emptyset$ there is some $s\in h(R)$. By $R\leq {\sf c}_0R_i$ we have
$h(R)\subseteq h(R_i)$, so there is a $w_i$ such that $s(0|w_i)\in h(R_i)$ 
for all $i\leq m$.
These $w_i$'s are distinct since the $R_i$'s are pairwise disjoint (they are distinct atoms)
and so are the 
$h(R_i)$'s.
Consider the sequence
$$z=\langle w_0, w_1, \ldots w_m, s_{m+1}, \ldots \rangle.$$
We show that $z\in \tau(h(R))$. Indeed let $i, j\leq m$, $i\neq j$, then $z\in -{\sf d}_{ij}$
by $w_i\neq w_j$. Next we show that $z\in {\sf s}_i^0{\sf c}_1\ldots {\sf c}_mh(R)$. 
By definition, $\langle w_i, s_1\ldots \rangle\in h(R_i)\subseteq h(R)$
so $\langle w_i, w_1,\ldots w_m, s_{m+1}, \rangle\in {\sf c}_1\ldots {\sf c}_mh(R)$ 
and thus $z\in {\sf c}_0({\sf d}_{0i}\cap {\sf c}_1\ldots {\sf c}_mh(R))={\sf s}_i^0{\sf c}_1\ldots {\sf c}_mh(R)$.
This contradicts that $\A_{k,n}\models \tau(R)=0$.

Next we show that the $k$ generated subalgebras of $\A_{k,n}$ are representable.
Let $G$ be given such that $|G|\leq k$. The idea is to use $G$ and define a ``small" subalgebra of $\A_{k,n}$ 
that contains $G$
and is representable. 
Define $R_i\equiv R_j$ iff
$$(\forall g\in G)(\forall \tau\in G_n)({\sf s}_{\tau}R_i\leq g\Longleftrightarrow {\sf s}_{\tau}R_j\leq g).$$
This is similar to the equivalence relation defined by Andr\'eka \cite{Andreka} p. 157; 
the difference is that substitutions have to come to the picture \cite{Andreka}p.189.
Then $\equiv$ is an equivalence relation on $\{R_j: j\leq m\}$ which has $\leq 2^{k.n!}$ blocks by $|G|\leq k$ and $G_n=n!.$
Let $p$ denote the number of blocks of $\equiv$, that is $p=|\{R_j/\equiv:j\leq m\}|\leq 2^{k.n!}\leq m$.
Now that $R$ is split into $p< m+1$ atoms, the incompatibility condition above no longer holds. 
Indeed, let
$$B=\{a\in A_{k,n}: (\forall i,j\leq m)(\forall \tau\in G_n)(R_i\equiv R_j\text { and }{\sf s}_{\tau}R_i\leq a\implies {\sf s}_{\tau}R_j\leq a\}.$$

We first show that $B$ is closed under the operations of $\A_{k,n}$, then we show 
that, unlike $\A_{k,n}$,  $B$ is the universe of a representable algebra. Let $i<l<\omega$
Clearly $B$ is closed under the Boolean operations. The diagonal element
${\sf d}_{il}\in \B$ since ${\sf s}_{\tau}R_j\nleq {\sf d}_{il}$ for all $j\leq m$ and $\tau\in G_n$.
Also $A'\subseteq B$ since ${\sf s}_{\tau}R$ is an atom of $\A'$ and ${\sf c}_ia\in A'$ for all $a\in A_{k,n}$.
Thus ${\sf c}_ib\in B$ for all $b\in B$.
Assume that $a\in B$ and let $\tau\in G_n$. Suppose that $R_i\equiv R_j$ and ${\sf s}_{\sigma}R_i\leq {\sf s}_{\tau}a$. Then 
${\sf s}_{\tau}{\sf s}_{\sigma}R_i\leq a$, so ${\sf s}_{\tau\circ \sigma}R_i\leq a$. Since $a\in B$ we get
that ${\sf s}_{\tau\circ \sigma}R_j={\sf s}_{\tau}{\sf s}_{\sigma}R_j\leq a$, and so ${\sf s}_{\sigma}R_j\leq{\sf s}_{\tau}a$.
Thus $B$ is also closed under substitutions.
Let $\B\subseteq \A_{k,n}$ be the subalgebra of $\A_{k,n}$ with universe $B$. Since $G\subseteq B$
it suffices to show that $\B$ is representable.
Let $\{y_j;j<p\}=\{\sum(R_j/\equiv):j\leq m\}$. 
Then $\{y_j:j<p\}$ is a partition of $R$ in $\B,$ ${\sf c}_iy_j={\sf c}_iR$ for all $j<p$ and 
$i<\omega$ and every element of $\B$ is a join of some element of $\A'$ and of finitely many of 
${\sf s}_{\tau}y_j$'s. Recall that $p\leq m$.
We now split $R$ into $m$ `real' atoms, cf. \cite{Andreka} p.167, lemma 2. We define an equivalence relation on $R$. For any $s,z\in R$
$$s\sim z\Longleftrightarrow |\{i\in \omega: s_i\neq z_i\}|<\omega.$$
Let $S\subseteq R$ be a set of reprsentatives of $\sim$. 
Consider the group $Z_m$ of integers modulo $m$. 
(Any finite abelian group with $m$ elements 
will do.) 
For any $s\in S$ and $i\in \omega$ let $f_i^s:U_i\to Z_m$ be an onto map such that $f_i^s(s_i)=0$.
For $j<m$ define
$$R_j^s=\{z\in R: \sum\{f_i^s:i\in \omega\}=j\}$$
and 
$$R_j"=\bigcup \{R_j^s: s\in S\}.$$
Then $\{R_0'', \ldots  R_{m-1}''\}$ is a partition of $R$ such that ${\sf c}_iR_j''= {\sf c}_iR$ for all $i<\omega$ and $j<m$.
Let $\A''$ be the subalgebra of 
$\langle \B(^{\omega}U), {\sf c}_i, {\sf d}_{ij}, {\sf s}_{\tau}\rangle_{i,j<\omega, \tau\in G_n}$
generated by $R_0'',\ldots R_{m-1}''.$
Let 
$${\cal R}=\{{\sf s}_{\sigma}R_j'':\sigma\in G_n, j<m\}.$$
Let $$H=\{a+\sum X: a\in A', X\subseteq_{\omega}\cal R\}.$$
Clearly $H\subseteq A''$ and $H$ is closed under the boolean operations. Also
because transformations considered are bijections we have
$${\sf c}_i{\sf s}_{\sigma}R_j={\sf c}_i{\sf s}_{\sigma}R\text { for all $j<m$ and }\sigma\in G_n.$$
Thus $H$ is closed under ${\sf c}_i.$ Also $H$ is closed under substitutions.
Finally ${\sf d}_{ij}\in A'\subseteq H.$
We have proved that $H=A''$. This implies that every element of $\R$ is an atom of $\A''$.
We now show that $\B$ is embeddable in $\A''$, and hence will be representable.
Define for all $j<p-1$,
$$R_j'=R_j'',$$
and
$$R_{p-1}'=\bigcup\{R_j'': p-1\leq j<m\}$$
Then define for $b\in B$:
$$h(b)=(b-\sum_{\tau\in G_n}{\sf s}_{\tau}R)\cup\bigcup\{{\sf s}_{\tau}R_j': \tau\in G_n,  j<p, {\sf s}_{\tau}y_j\leq b\}.$$
It is clear that $h$ is one one, 
preseves the Boolean operations and the diagonal elements and is the identity on $A'$.
Now we check cylindrifications and substitutions.
\begin{equation*}
\begin{split}
{\sf c}_ih(b)&={\sf c}_i[(b-\sum {\sf s}_{\tau}R)\cup\bigcup\{{\sf s}_{\tau}R_j': \tau\in G_n, j<p, {\sf s}_{\tau}y_j\leq b\}]\\
&={\sf c}_i(b-\sum_{\tau\in G_n} {\sf s}_{\tau}R)\cup\bigcup\{{\sf c}_i{\sf s}_{\tau}R_j':\tau\in G_n, j<p, {\sf s}_{\tau}y_j\leq b\}\\
&={\sf c}_i(b-\sum_{\tau\in G_n} {\sf s}_{\tau}R\cup\bigcup\{{\sf c}_i{\sf s}_{\tau}y_j: \tau\in G_n, j<p, {\sf s}_{\tau}y_j\leq b\}\\
&={\sf c}_i[(b-\sum_{\tau\in G_n} {\sf s}_{\tau}R)\cup\bigcup\{{\sf s}_{\tau}y_j,\tau\in G_n, j<p, {\sf s}_{\tau}y_j\leq b\}]\\
&={\sf c}_ib\\
\end{split}
\end{equation*}
On the other hand
$$h{\sf c}_i(b)=({\sf c}_ib-\sum_{\tau\in G_n} {\sf s}_{\tau}R)\cup\bigcup\{{\sf s}_{\tau}R_j': {\sf s}_{\tau}y_j\leq {\sf c}_ib\}={\sf c}_ib.$$
Preservation of substitutions follows from the fact that the substitutions are Boolean endomorphisms. In more detail, let $\sigma\in G_n$, then:
\begin{equation*}
\begin{split}
{\sf s}_{\sigma}h(b)&={\sf s}_{\sigma}[(b-\sum_{\tau\in G_n} {\sf s}_{\tau}R)\cup\bigcup\{{\sf s}_{\tau}R_j':\tau\in G_n, j<p, {\sf s}_{\tau}y_j\leq b\}]\\
&=({\sf s}_{\sigma}b-\sum_{\tau\in G_n} {\sf s}_{\sigma}{\sf s}_{\tau}R)\cup\bigcup\{{\sf s}_{\sigma}
{\sf s}_{\tau}R_j': \tau\in G_n, j<p, {\sf s}_{\tau}y_j\leq b\}]\\
&=({\sf s}_{\sigma}b-\sum_{\tau\in G_n} {\sf s}_{\sigma\circ \tau}R)\cup\bigcup\{{\sf s}_{\sigma\circ \tau}R_j': \tau\in G_n, j<p, {\sf s}_{\tau}y_j\leq b\}]\\
&=({\sf s}_{\sigma}b-\sum_{\tau\in G_n} {\sf s}_{\tau}R)\cup\bigcup\{{\sf s}_{\tau}R_j':  \tau\in G_n, j<p, {\sf s}_{\tau}y_j\leq b\}]\\
\end{split}
\end{equation*}

But for fixed $\sigma,$ we have $\{\sigma\circ \tau: \tau\in G_n\}=G_n$ and so
$${\sf s}_{\sigma}h(b)=h({\sf s}_{\sigma}(b)).$$

For every $k,n<\omega$ we have constructed an algebra $\A_{k,n}$ such that 
$\Rd_{\CA}\A_{k,n} \notin \RCA_{\omega}$ and the $k$-generated subalgebras of 
$\A_{k,n}$ are representable.  
We should point out that the  ``finite dimensional version" of the $\A_{k,n}$'s 
were constructed in \cite{c}, and their construction can be recovered from 
the proof of Theorem 6 in \cite{Andreka} which addresses the finite dimensional case but in a more general 
setting allowing arbitrary unary additive permutation invariant operations expanding 
those of $\RCA_n$. We note that the latter result does not survive the infinite dimensional case. There are easy examples, cf. \cite{Andreka} p.192 
and \cite{ex}. 

\item We show that $\A_{k,n}$ has a representation which preserves all operations except for the diagonal elements. 
That is, its quasipolyadic reduct is 
representable.
The proof is analogous to that of Andr\'eka's on of Claim 16 on p.194 of \cite{Andreka}.
Let $\A_{k,n}$ be the algebra obtained by splitting the atom $R$ in $\A'$ as in the above proof. Then
$\A_{k,n}$ is not representable,  but its $k$ generated subalgebras are representable.
We show that there is a representation of $\A_{k,n}$ in which all operations are preserved except for the diagonal elements.
Let $U_i$, $i<\omega$ be a sequence of pairwise disjoint sets such that $|U_0|=m\geq 2^{k.n!+1}$  and $|U_i|\geq m+1$. Let $R$ be as above
except that it is defined via the new $U_i$'s.   Let $(R_j:j\leq m)$ be the splitting of $R$ in $\A_{k,n}$. 
Let $W\supset U$ (properly).
Let $W_0=U_0\cup (W\sim U)$, and $W_i=U_i$ for $0<i<\omega$.
First we define a function $h:\wp(^{\omega}U)\to \wp(^{\omega}W)$ with the desired properties and $h(R)=\prod_{i<\omega}W_i$.
Let $t: W\to U$ be a surjective function which is the identity on $U$ and which maps $W_0$ to $U_0$. Define $g:{}^{\omega}W\to {}^{\omega}U$ 
by $g(s)=t\circ s$ for all $s\in {}^{\omega}W$ and for all $x\subseteq {}^{\omega}U,$ define
$$h(x)=\{s\in {}^{\omega}U: g(s)\in x\}.$$
Since $|W_i|\geq m+1$ for all $i<\omega$, the incompatibity condition between the number of atoms 
splitting $R$ and the number of elements in $|W_0|$
used in the representation vanishes, 
so there is a real partition $(S_j: j\leq m)$ 
of $S=\prod_{i<\omega}W_i$ such that ${\sf c}_iS_j={\sf c}_iS$ for all $ i<\omega$ and $j\leq m$.
Then $({\sf s}_{\sigma}S_j: j\leq m)$ is an analogous partition of ${\sf s}_{\sigma}S$ for $\sigma\in G_n$.
Let $X_{\sigma,j}={\sf s}_{\sigma}^{\A_{k,n}}R_j$ for $j\leq m$.
Define $\bar{h}:\A_{k,n}\to \wp(^{\omega}W)$ by
$$\bar{h}(a)=h(a), \ \ a\in A'$$
$$\bar{h}(X_{\sigma j})={\sf s}_{\sigma}S_j, \sigma \in G_n, j\leq m$$
and
$$\bar{h}(x+y)=\bar{h}(x)+\bar{h}(y), x,y\in A_{k,n}.$$
It is easy to check using lemma  (iv) in \cite{Andreka} that $\bar{h}$ is as desired. In fact $\bar{h}$ 
preserves all the quasipolyadic operations including substitutions 
corresponding to replacements, which are now no longer definable, because we have discarded diagonal elements.
The reasoning is as follows \cite{Andreka} p.194. 
For $i,j\in n$, 
the quantifier free formula $x\leq -{\sf d}_{ij}\to {\sf s}_i^jx$ is valid in representable algebras 
hence it is valid in $\A_{k,n}$ since its  $k$ generated subalgebra are representable. 
Let $\sigma\in G_n$, $l\leq m$. Then ${\sf s}_i^j(X_{\sigma, l})=0$ in $\A_{k,n}$.
Now
$$\bar{h}({\sf s}_i^j(X_{\sigma, l}))=\bar{h}(0)=0={\sf s}_i^j {\sf s}_{\sigma}S_j={\sf s}_i^jh(X_{\sigma l}).$$ 
Assume that $a\in A_{k,n}$. Then
$$\bar{h}({\sf s}_i^ja)=h({\sf s}_i^ja)={\sf s}_i^jh(a)={\sf s}_i^j\bar{h}(a).$$ 
Since both $\bar{h}$ and ${\sf s}_i^j$ are additive we get the required. 
\end{enumarab}

{\bf Part II}
\begin{enumarab}
\item Here is where we really start the non-trivial modification of Andr\'eka's splitting. 
For $n\in \omega$ and $m=2^{k.n!+1}$, we denote $\A_{k,n}$ by $split(\A', R, m, n)$.
This is perfectly legitimate since the algebra $\A_{k,n}$ is determined uniquely by $R,$ $\A'$, $m$ and $n$. 
Recall that $m$ is the number of atoms splitting $R$, while $n$ is the finite number of substitutions available.
For $n_1<n_2$, we denote by $\Rd_nsplit(\A', R, m, n_2)$ 
the reduct of $split(\A', R, m, n_2)$ obtained by restricting substitutions to $G_{n_1}.$ 
Let $m_1<m_2$ and $n_1<n_2$. Then we claim that
$$split(\A', R, m_1,n_1)\text { embeds into }\Rd_{n_1}split(\A', R, m_2, n_2).$$
This part of the proof is analogous to Andr\'eka's proofs in \cite{Andreka}, lemma 3, on splitting elements in cylindric algebras.
Indeed, let $$\chi: m_1\to m_2$$ 
be such that the set $\chi(j),$ $j<m_1$ are non empty and pairwise disjoint, and
$$\bigcup \{\chi(j):j<m_1\}=m_2.$$
For $x\in split(\A, R, m_1, n_1),$ let 
$$J_{\tau}(x)=\{j<m_1: {\sf s}_{\tau}R_j\leq x\}.$$
Let $(R_i: i\leq m_2)$ be the splitting of $R$ in $split(\A', R, m_2, n_2)$.
Define 
$$h(x)=(x-\sum {\sf s}_{\tau}R)+\sum\{{\sf s}_{\tau}R_i: \tau\in G_{n_1}, i\in \bigcup \{\chi(j): j\in J_{\tau}(x)\}\}.$$
Here we are considering $G_{n_1}$ as a subset of $G_{n_2}$.
It is easy to check that $h(x)$ is a Boolean homomorphism and that $h(x)\neq 0$ whenever $0\neq x\leq {\sf s}_{\tau}R,$ for $\tau\in G_{n_1}$.
Thus $h$ is one to one.
Let $i\in \omega$ and $x\in split(\A, R, m_1, n_1)$. If
$x\cdot {\sf s}_{\tau}R=0$ for all $\tau,$ then $x\in A'$, hence $h({\sf c}_ix)={\sf c}_ih(x)$. 
So assume that there is a $\tau\in G_n$ such that $x\cdot {\sf s}_{\tau}R\neq 0$. Then ${\sf c}_i(x\cdot {\sf s}_{\tau}R)={\sf c}_i{\sf s}_{\tau}R$ and 
${\sf c}_ih(x\cdot {\sf s}_{\tau}R))={\sf c}_i{{\sf s}_{\tau} R}$ by $0\neq h(x\cdot {\sf s}_{\tau}R)\leq {\sf s}_{\tau}R$.
Now
\begin{equation*}
\begin{split}
h({\sf c}_ix)&=h({\sf c}_i(x-{\sf s}_{\tau}R)+{\sf c}_i(x\cdot {\sf s}_{\tau}R))\\
&={\sf c}_i(x-{\sf s}_{\tau}R)+{\sf c}_iR.\\
\end{split}
\end{equation*}
\begin{equation*}
\begin{split}
{\sf c}_ih(x)&={\sf c}_i(h(x-{\sf s}_{\tau}R)+x\cdot {\sf s}_{\tau}R))\\
&={\sf c}_i(h(x-{\sf s}_{\tau}R))+h(x\cdot {\sf s}_{\tau}R)\\
&={\sf c}_i(x-{\sf s}_{\tau}R)+{\sf c}_iR.\\
\end{split}
\end{equation*}
We have proved that 
$${\sf c}_ih(x)=h({\sf c}_ix).$$ Now we turn to substitutions. Let $\sigma\in G_{n_1}$. Then we have
\begin{equation*}
\begin{split}
{\sf s}_{\sigma}h(x)&={\sf s}_{\sigma}[(x-\sum_{\tau\in G_{n_1}} {\sf s}_{\tau}R)+\sum\{{\sf s}_{\tau}R_i, \tau\in G_{n_1}, i\in \bigcup\{\chi(j): j\in J_{\tau}x\}\})]\\
&=({\sf s}_{\sigma}x-\sum_{\tau\in G_{n_1}}{\sf s}_{\sigma\circ \tau}R)+
\sum\{{\sf s}_{\sigma\circ \tau}R_i:\tau\in G_{n_1}, i\in \bigcup\{\chi(j): j\in J_{\tau}x\}\}.\\
\end{split}
\end{equation*}
Since $\{\sigma\circ \tau:\tau\in G_{n_1}\}=G_{n_1}$ then we have:
$${\sf s}_{\sigma}h(b)=h({\sf s}_{\sigma}(b).)$$

\item We have a sequence of algebras $(\A_{k,i}:i\in \omega\sim{0})$ such that for $n<m,$ 
we can assume by the embeddings proved to exist in the previous item
that $\A_{k,n}$ is a subreduct (subalgebra of a reduct) 
of $\A_{k,m}$.  
Form the natural direct limit of such algebras which is the (reduct directed) union
call it $\A_k$. That is $A_k=\bigcup_{n\in \omega}A_{k,n}$, and the operations are defined the obvious way.
For example if $i<\omega$, and $a\in A_k$, then $i\in n$ and $a\in A_{n,k}$ for some $n$; set ${\sf c}_i^{\A_k}a={\sf c}_i^{\A_{k,n}}a$.
These are well defined. The other operations are defined analogously, where we only define the ${\sf s}_{[i,j]}$'s for $i,j\in \omega$. 
Clearly, $\Rd_{ca}\A_k$ is not representable, for else $\Rd_{ca}\A_{k,n}$ would be representable for all $n\in \omega$.

\item Let $|G|\leq k$. Then $G\subseteq \A_{k,n}$ for some $n$. If $\Sg^{\A_{k}}G$ is not representable then there exists 
$l\geq n$ such that $G\subseteq \A_{k,l}$ and 
$\Sg^{\A_{k,l}}G$ is not representable, contradiction.
To see this we can show directly that $\Sg^{\A_k}G$ has to be representable. 
We show that every equation $\tau=\sigma$  valid in the variety
$\RQEA_{\omega}$ is valid in $\Sg^{\A_k}G$. 
Let $v_1,\ldots v_k$ be the variables occuring in this equation, and let $b_1,\ldots b_k$ be arbitrary elements of $\Sg^{\A_k}G$. 
We show that $\tau(b_1,\ldots b_k)=\sigma(b_1\ldots b_k)$. Now there are terms 
$\eta_1\ldots \eta_k$ written up from elements of $G$ such that $b_1=\eta_1\ldots b_k=\eta_k$, then we need to show that 
$\tau(\eta_1,\ldots \eta_k)=\sigma(\eta_1, \ldots \eta_k).$ 
This is another equation written up from elements of $G$, which is also valid in $\RQEA_{\omega}$. 
Let $n$ be an upper bound for the indices occuring in this equation and let $l>n$ be such that $G\subseteq \A_{k,l}$. 
Then the above equation is valid in $\Sg^{\Rd_n\A}G$ since the latter is representable. 
Hence the equation $\tau=\sigma$ holds in $\Sg^{\A_k}G$ at the evaluation $b_1,\ldots b_k$ of variables.

\item Let $\Sigma_n^v$ be the set of universal formulas using only $n$ substitutions and $k$ variables valid in $\RQEA_{\omega}$, 
and let $\Sigma_n^d$ be the set of universal formulas
using only $n$ substitutions and no diagonal elements valid in $\RQEA_{\omega}$.  By $n$ substitutions we understand the set 
$\{{\sf s}_{[i,j]}: i,j\in n\}.$
Then $\A_{k,n}\models \Sigma_n^v\cup \Sigma_n^d$. $\A_{k,n}\models \Sigma_n^v$ because the $k$ generated subalgebras
of $\A_{k,n}$ are representable, while $\A_{k,n}\models \Sigma_n^d$ because $\A_{k,n}$ has a representation that preserves all operations except
for diagonal elements.  Indeed, let $\phi\in \Sigma_n^d$, then there is a representation of $\A_{k,n}$ in which all operations 
are the natural ones except for the diagonal elements. 
This means that (after discarding the diagonal elements) there is a one to one homomorphism 
$h:\A^d\to \P^d$ where $\A^d=(A_{k,n}, +, \cdot , {\sf c}_k, {\sf s}_{[i,j]}, {\sf s} _i^j)_{k\in \omega, i,j\in n}\text { and } 
\P^d=(\B(^{\omega}W), {\sf c}_k^W, {\sf s}_{[i,j]}^W, {\sf s}_{[i|j]}^W)_{k\in \omega, i,j\in n},$ 
for some infinite set $W$. 
Now let $\P=(\B(^{\omega}W), {\sf c}_k^W, {\sf s}_{[i,j]}^W, {\sf s}_{[i|j]}^W, {\sf d}_{kl}^W)_{k,l\in \omega, i,j\in n}.$
Then we have that $\P\models \phi$ because $\phi$ is valid 
and so  $\P^d\models \phi$ due to the fact that  no diagonal elements  occur in $\phi$. 
Then $\A^d\models \phi$ because $\A^d$ is isomorphic to a subalgebra of $\P^d$ and $\phi$ is quantifier free. Therefore 
$\A_{k,n}\models \phi$.
Let $$\Sigma^v=\bigcup_{n\in \omega}\Sigma_n^v 
\text { and }\Sigma^d=\bigcup_{n\in \omega}\Sigma_n^d$$ 
Hence $\A_k\models \Sigma^v\cup \Sigma^d.$ For if not then there exists a quantifier free  formula $\phi(x_1,\ldots x_m)\in \Sigma^v\cup \Sigma^d$,
and $b_1,\ldots b_m$ such that $\phi[b_1,\ldots b_n]$ does not hold in $\A_k$. We have $b_1\ldots b_m\in \A_{k,i}$ for some $i\in \omega$. 
Take $n$ large enough $\geq i$ so that
$\phi\in \Sigma_n^v\cup \Sigma_n^d$.   Then $\A_{k,n}$ does not model $\phi$, a contradiction.
Now let $\Sigma$ be  a set of quantifier free formulas axiomatizing  $\RQEA_{\omega}$, then $\A_k$ does not model $\Sigma$ since $\A_k$ is not 
representable, so there exists a formula $\phi\in \Sigma$ such that
$\phi\notin \Sigma^v\cup \Sigma^d.$ Then $\phi$ contains more than $k$ variables and a diagonal constant occurs in $\phi$.
\end{enumarab}
\end{demo}

We immediately get the following answer to Andr\'eka's question formulated on p. 193 of \cite{Andreka}.

\begin{corollary} The variety $\RQEA_{\omega}$ is not axiomatizable over $\RQA_{\omega}$ 
with a set of universal formulas containing infinitely many variables.
\end{corollary}
One can show, using the modified method of splitting here, that all theorems in \cite{Andreka} on complexity of axiomatizations 
generalize to $\RQEA_{\omega}$ with the sole exception of Theorem 5, 
which is false for infinite dimensions \cite{ex}. As a sample we give the following theorem 
which can be proved by some modifications of the cited theorems
in the proof; this modifications are not hard, most of them can be found in \cite{c}, proof of theorem 3.1. The basic idea 
in the proof is to show that certain cylindric homomorphisms
(between cylindric algebras) remain to be quasipolyadic equality algebra homomorphisms 
(between their corresponding natural quasipolyadic equality expansions), that is when we add substitutions corresponding to transpositions.

\begin{theorem} Let $\Sigma$ be a set of equations 
axiomatizing $\RQEA_{\omega}.$  Let $l,k, k' <\omega$. 
Then $\Sigma$ contains
infinitely equations in which $-$ occurs, one of $+$ or $\cdot$ occurs  a diagonal or a permutation with index $l$ occurs, more
than $k'$ cylindrifications and more than $k$ variables occur.
\end{theorem}
\begin{demo}{Sketch of Proof} Let $n,k\in \omega\sim\{0\}$. Let $\A_{k,n}$ be the non-representable algebra constructed above 
obtained by splitting $\A'$ into $m\geq 2^{k.n!+1}$ atoms and we require that $m\geq k'$ as well. One then shows that 
the complementation free reduct  $\A_{k,n}^{-}$ of $\A_{k,n}$ is a homomorphic image of a subalgebra $\C$ of 
the complemention free reduct of $\P^{-}$ of a  a representable $\P$, \cite{Andreka}, cf. Theorem 7, p.163. The algebra
$\A_{k,n}$ can be represented such that every operation except for $\cup$ and $\cap$ are the natural ones, cf. \cite{Andreka} p.200 and \cite{c} for the necessary 
modifications.
For any $I\subseteq \omega$, $|I|=m$ there is an infinite set $W$ an an embedding  
from $\A_{k,n}\to (\B(^{\omega}W), {\sf c}_i, {\sf d}_{ij}, {\sf s}_{\tau})_{i,j\in \omega, \tau\in G_n}$ 
which is a homomorphism with respect to all operations of $\A_{k,n}$ except for ${\sf c}_i$ $i\notin I$, cf. \cite{Andreka} p.172, Theorem 3.
One just has to show that the map $h:\A'\to \wp(^{\omega}W)$ defined on p. 174 preserves substitutions, which is 
straightforward from the definition of the map $g$ defined on 
p.173.
There is an infinite set $W,$ such that there is an embedding $h:\A\to (\B(^{\omega}W), {\sf c}_i, {\sf d}_{ij}, {\sf s}_{\tau})_{i,j<n, \tau\in G_n}$ 
such that $h$ is a homomorphism preserving all operations except for ${\sf d}_{il}$ and ${\sf s}_{[i,l]}$ if $i,l\in n$, cf. p.176 Claim 6. 
This will prove the theorem because of the following reasoning.  
By $n$ substitutions we understand the set $\{{\sf s}_{[i,j]}: i,j\in n\}$. 
Let $\Sigma_n^-$ denote the set of equations  without complementation in which only $n$ substitutions occur, 
$\Sigma_n^v$ be the set of equations
which contains at most $k$ variables in which only $n$ substitutions occur, $\Sigma_n^c$ be the set of equations 
in which only $k'$ cylindrifications and $n$ substitutions occur,
$\Sigma_n^{ds}$ be the set of equations in which at most $n$ substitutions occur and no diagonal nor substitutions with index $l$ occurs, 
and $\Sigma_n^{Bool}$ the set of equations 
that does not contain $\cdot$ nor  $+$ and $n$ substitutions occur, all valid in $\RQEA_{\omega}$.  Then 
$\A_{k,n}\models \Sigma_n^-\cup \Sigma_n^v\cup \Sigma_n^c\cup \Sigma_n^{ds}\cup \Sigma_n^{Bool}.$ 
Indeed, the algebra $\A_{k,n}\models \Sigma_n^-$ because of the following reasoning.
Let $\C$ and $\P$ be as above. Then $\P\models \Sigma_n^-$ because $\P$ is representable. So $\P^-\models \Sigma_n^-$ 
because $-$ does not occur in 
$\Sigma_n^-$. Now  $\C\models \Sigma_n^-$ by 
$\C\subseteq \P^{-}$, and so $\A_{k,n}^-\models \Sigma_n^-$ since $\A^{-}_{k,n}$ is a homomorphic image of $\C$ and $\Sigma_n^-$ 
consists of equations. 
Then $\A_{k,n}\models \Sigma_n^-$. 
This together with previous reasoning proves that $\A_{k,n}\models  \Sigma_n$ where $\Sigma_n$ is the above (union) of formulas. 
In more detail $\A_{k,n}\models \Sigma^v_n$ because its $k$ generated subalgebras are representable, 
$\A_{k,n}\models \Sigma_n^{dp}$ because it has a representation that preserves all elements except for 
diagonals and substitutions with index $l$, and so forth. Then we can infer that  
$\A_k\models \bigcup_{n\in \omega} \Sigma_n=\Gamma.$
For if not, then we can choose $n$ large enough such that $\A_{k,n}$ does not model 
$\Sigma_n$. But $\A_k$ is not representable hence any equational 
axiomatization of the representable algebras contain a formula that is outside $\Gamma$.
Thus the required follows. 
\end{demo}
For axiomatization with universal formulas, using the same ideas as above, except for those involving complementation,
we obtain the slightly weaker:
\begin{theorem}
Let $\Sigma$ be a set of quantifier free formulas 
axiomatizing $\RQEA_{\omega}.$  Let $l,k, k' <\omega$. 
Then $\Sigma$ contains
infinitely equations in which  one of $+$ or $\cdot$ occurs  a diagonal or a permutation with index $l$ occurs, more
than $k'$ cylindrifications and more than $k$ variables occur
\end{theorem}
Our corollary 2 is  evidence that $\RQA_{\omega}$ can be axiomatized by an infinite set of universal formulas containing only 
finitely many variables.

Next we generalize Theorem 2 in \cite{Andreka} to the class $S\Nr_{\omega}\QEA_{\omega+p}$ for $p\geq 2.$
Let $n\in \omega$ and $m=2^{k.n!+1}$. Let 
$$e_n=\prod_{i\leq m}{\sf c}_0(x\cdot x_i\cdot \prod_{i\neq j\leq m}-x_j)\leq {\sf c_0}
\ldots {\sf c}_m(\prod_{i,j\leq m, i\neq j}{\sf s}_i^0{\sf c}_1\ldots {\sf c}_mx.-{\sf d}_{ij}).$$ Note that $e_n$ is equivalent to the above set of equations.
Then $\A_{k,n}$ does not model $e_n$, and so $\A_k$ does not model $e_n$, for all $n\in \omega$, but 
$S\Nr_{\omega}\QEA_{\omega+2}\models e_n$ for all $n\in \omega$. This is done like the $\CA$ case since every equation that holds in 
$S\Nr_n\CA_{n+2}$ holds in $S\Nr_n\QEA_{n+2}$. Using the reasoning on p.163, one obtains:

\begin{theorem} Let $p\geq 2$. Then $S\Nr_{\omega}\QEA_{\omega+p}$ is 
not  axiomatizable with any set of quantifier free formulas containing only finitely 
many variables. 
\end{theorem}

\subsubsection{Monks algebras modified, by Hirsch and Hodkinson}

Now we prove the conclusion of theorem \ref{2.12}, for cylindric algebras and quasipolyadic equality, 
solving the infinite dimensional version of the famous 2.12 problem in algebraic logic.
The finite dimensional algebras we use are constructed by Hirsch and Hodkinson; and 
they based on a relation algebra construction.
Such combinatorial algebras have affinity with Monk's algebras.
Related algebras were constructed by Robin Hirsch and the present author (together with the above lifting argument)
to prove theorem the analogue of theorem \ref{2.12} holds for various 
equality free algebraisations of first order logic.


We recall the construction of Hirsch and Hodkinson. They prove their result for cylindric algebras.
Here, by noting that their atom structures are also symmetric; it permits expansion by substitutions, 
we slightly extend the result to polyadic equality algebras.
Define  relation algebras $\A(n,r)$ having two parameters $n$ and $r$ with $3\leq n<\omega$ and $r<\omega$.
Let $\Psi$ satisfy $n,r\leq \Psi<\omega$. We specify the stom structure of $\A(n,r)$.
\begin{itemize}
\item The atoms of $\A(n,r)$ are $id$ and $a^k(i,j)$ for each $i<n-1$, $j<r$ and $k<\psi$.
\item All atoms are self converse.
\item We can list te forbidden triples $(a,b,c)$ of atoms of $\A(n,r)$- those such that
$a.(b;c)=0$. Those triples that are not forbidden are the consistent ones. This defines composition: for $x,y\in A(n,r)$ we have
$$x;y=\{a\in At(\A(n,r)); \exists b,c\in At\A: b\leq x, c\leq y, (a,b,c) \text { is consistent }\}$$
Now all permutations of the triple $(Id, s,t)$ will be inconsistent unless $t=s$.
Also, all permutations of the following triples are inconsistent:
$$(a^k(i,j), a^{k'}(i,j), a^{k''}(i,j')),$$
if $j\leq j'<r$ and $i<n-1$ and $k,k', k''<\Psi$.
All other triples are consistent.
\end{itemize}

Hirsch and Hodkinson invented means to pass from relation algebras to $n$ dimensional cylindric algebras, 
when the relation algebras in question have what they call a hyperbasis.

Unless otherwise specified, 
$\A=(A,+,\cdot, -,  0,1,\breve{} , ;, Id)$ 
will denote an arbitrary relation algeba with $\breve{}$ standing for converse, and $;$ standing for composition, and 
$Id$ standing for the identity relation.

\begin{definition} Let $3\leq m\leq n\leq k<\omega$, and let $\Lambda$ be a non-empty set.
An $n$ wide $m$ dimensional $\Lambda$ hypernetwork over $\A$ is a map
$N:{}^{\leq n}m\to \Lambda\cup At\A$ such that $N(\bar{x})\in At\A$ if $|\bar{x}|=2$ and $N(\bar{x})\in \Lambda$ if $|\bar{x}|\neq 2$,
with the following properties:
\begin{itemize}
\item $N(x,x)\leq Id$ ( that is $N(\bar{x})\leq Id$ where $\bar{x}=(x,x)\in {}^2n.)$

\item $N(x,y)\leq N(x,z);N(z,y)$
for all $x,y,z<m$
\item If $\bar{x}, \bar{y}\in {}^{\leq n}m$, $|\bar{x}|=|\bar{y}|$ and $N(x_i,y_i)\leq Id$ for all $i<|\bar{x}|$, then $N(\bar{x})=N(\bar{y})$
\item when $n=m$, then $N$ is called an $n$ dimensional $\Lambda$ hypernetwork.

\end{itemize}
\end{definition}

\begin{definition} Let $M,N$ be $n$ wide $m$ dimensional $\Lambda$ hypernetworks.
\begin{enumarab}
\item For $x<m$ we write $M\equiv_xN$ if $M(\bar{y})=N(\bar{y})$ for all $\bar{y}\in {}^{\leq n}(m\sim \{x\})$
\item More generally, if $x_0,\ldots x_{k-1}<m$ we write $M\equiv_{x_0,\ldots,x_{k-1}}N$
if $M(\bar{y})=N(\bar{y})$ for all $\bar{y}\in {}^{\leq n}(m\sim \{x_0,\ldots x_{k-1}\}).$
\item If $N$ is an $n$ wide $m$ dimensional $\Lambda$ -hypernetork over $\A$, and $\tau:m\to m$ is any map, then 
$N\circ \tau$ denotes the $n$ wide $m$ dimensional $\Lambda$ hypernetwork over $\A$ with labellings defined by
$$N\circ \tau(\bar{x})=N(\tau(\bar{x})) \text { for all }\bar{x}\in {}^{\leq n}m.$$
That is
$$N\circ \tau(\bar{x})=N(\tau(x_0),\ldots ,\tau(x_{l-1}))$$
\end{enumarab}
\end{definition}

\begin{lemma} Let $N$ be an $n$ dimensional $\Lambda$ hypernetwork over $\A$ and $\tau:n\to n$ be a map. 
Then $N\circ \tau$ is also a network.
\end{lemma}
\begin{demo}{Proof} \cite{HH} lemma 12.7
\end{demo}

\begin{definition} The set of all $n$ wise $m$ dimensional hypernetworks will be denoted by $H_m^n(\A,\Lambda)$.
An $n$ wide $m$ dimensional $\Lambda$ 
hyperbasis for $\A$ is a set $H\subseteq H_m^n(\A,\lambda)$ with the following properties:
\begin{itemize}
\item For all $a\in At\A$, there is an $N\in R$ such that $N(0,1)=a$
\item For all $N\in R$ all $x,y,z<n$ with $z\neq x,y$ and for all $a,b\in At\A$ such that
$N(x,y)\leq a;b$ there is $M\in R$ with $M\equiv_zN, M(x,z)=a$ and $M(z,y)=b$
\item For all $M,N\in H$ and $x,y<n$, with $M\equiv_{xy}N$, there is $L\in H$ such that
$M\equiv_xL\equiv_yN$
\item For a $k$ wide $n$ dimensional hypernetwork $N$, we let $N|_m^k$ the restriction of the map $N$ to $^{\leq k}m$.
For $H\subseteq H_n^k(\A,\lambda)$ we let $H|_k^m=\{N|_m^k: N\in H\}$.
\item When $n=m$, $H_n(\A,\Lambda)$ is called an $n$ dimensional hyperbases.
\end{itemize}
We say that $H$ is symmetric, if whenever $N\in H$ and $\sigma:m\to m$, then $N\circ\sigma\in H$.
\end{definition}
We note that $n$ dimensonal hyperbasis are extensions of Maddux's notion of cylindric basis.

\begin{theorem} If $H$ is a $m$ wide $n$ dimensional $\Lambda$ symmetric 
hyperbases for $\A$, then $\Ca H\in \PEA_n$.
\end{theorem}
\begin{demo}{Proof} Let $H$ be the set of $m$ wide $n$ dimensinal $\Lambda$ symmetric hypernetworks for $\A$.
The domain of $\Ca(H)$ is $\wp(H)$. 
The Boolean operations are defined as expeced (as complement and union of sets). For $i,j<n$ the diagonal is defined by
$${\sf d}_{ij}=\{N\in H: N(i,j)\leq Id\}$$
and for $i<n$ we define the cyylindrifier ${\sf c}_i$ by
$${\sf c}_iS=\{N\in H: \exists M\in S(N\equiv_i M\}.$$
Now the polyadic operations are defind by
$${\sf p}_{ij}X=\{N\in H: \exists M\in S(N=M\circ [i,j])\}$$

Then $\Ca(H)\in \PEA_n$. Furthermore, $\A$ embeds into 
$\Ra(\Ca(H))$ via 
$a\mapsto \{N\in H: N(0,1)\leq a\}.$
\end{demo}

\begin{theorem} Let $3\leq m\leq n\leq k<\omega$ be given.
Then $\Ca(H|^k_m)\cong \Nr_m(\Ca(H))$
\end{theorem}
\begin{demo}{Proof}\cite{HH} 12.22
\end{demo}

The set $C=H_n^{n+1}(\A(n,r), \Lambda)$ aff all $(n+1)$ wide $n$ 
dimensional $\Lambda$ hypernetworks over $\A(n,r)$ is an $n+1$ wide $n$ 
dimensional {\it symmetric} $\Lambda$ hyperbasis. 
$H$ is symmetic, if whenever $N\in H$ and $\sigma:m\to m$, then $N\circ\sigma\in H$.
Hence $\A(n,r)$ embeds into the $Ra$ reduct of $C$.  

\begin{theorem} 
Assume that $3\leq m\leq n$, and let  
$$\C(m,n,r)=\Ca(H_m^{n+1}(\A(n,r),  \omega)).$$ Then the following hold:
\begin{enumarab}
\item For any $r$ and $3\leq m\leq n<\omega$, we 
have $\C(m,n,r)\in \Nr_m\PEA_n$. 
\item  For $m<n$ and $k\geq 1$, there exists $x_n\in \C(n,n+k,r)$ such that $\C(m,m+k,r)\cong \Rl_{x}C(n, n+k, r).$
\item $S\Nr_{\alpha}\CA_{\alpha+k+1}$ is not axiomatizable by a finite schema over $S\Nr_{\alpha}\CA_{\alpha+k}$
\end{enumarab}
\end{theorem}
\begin{demo}{Proof} 
\begin{enumarab}
\item $H_n^{n+1}(\A(n,r), \omega)$ is a wide $n$ dimensional $\omega$ symmetric hyperbases, so $\Ca H\in \PEA_n.$
But $H_m^{n+1}(\A(n,r),\omega)=H|_m^{n+1}$.
Thus
$$\C_r=\Ca(H_m^{n+1}(\A(n,r), \omega))=\Ca(H|_m^{n+1})\cong \Nr_m\Ca H$$
\item For $m<n$, let $$x_n=\{f\in F(n,n+k,r): m\leq j<n\to \exists i<m f(i,j)=Id\}.$$ 
Then $x_n\in \c C(n,n+k,r)$ and ${\sf c}_ix_n\cdot {\sf c}_jx_n=x_n$ for distinct $i, j<m$.
Futhermore 
\[{I_n:\c C}(m,m+k,r)\cong \Rl_{x_n}\Rd_m {\c C}(n,n+k, r).\]
via
\[ I_n(S)=\{f\in F(n, n+k, r): f\upharpoonright m\times m\in S, \forall j(m\leq j<n\to  \exists i<m\; f(i,j)=Id)\}.\]

\item Follows from theorem \ref{2.12}.
\end{enumarab}
\end{demo}

\subsection{The class of neat reducts proper}

Let $\K$ be any of cylindric algebra, polyadic algebra, with and without equality, or Pinter's substitution algebra.
We give a unified model theoretic construction, to show the following:
\begin{enumarab}
\item For $n\geq 3$ and $m\geq 3$, $\Nr_n\K_m$ is not elementary, and $S_c\Nr_n\K_{\omega}\nsubseteq \Nr_n\K_m.$
\item For any $k\geq 5$, $\Ra\CA_k$ is not elementary
and $S_c\Ra\CA_{\omega}\nsubseteq \Ra\CA_k$. 
\end{enumarab}
 $S_c$ stands for the operation of forming {\it complete} subalgebras. 
The relation algebra part formulated in the abstract reproves a result of Hirsch in \cite{r}, and answers a question of his 
posed in {\it op.cit}.
For $\CA$ and its relatives the idea is very much like that in \cite{MLQ}, the details implemented, in each separate case, 
though are significantly distinct, because we look for terms not in the clone of operations
of the algebras considered; and as much as possible, we want these to use very little spare dimensions.

The relation algebra part is more delicate. We shall construct a relation algebra $\A\in \Ra\CA_{\omega}$ with a complete subalgebra $\B$, 
such that $\B\notin \Ra\CA_k$, and $\B$ is elementary equivalent to $\A.$ (In fact, $\B$ will be an elementary subalgebra of $\A$.)

Roughly the idea is to use an uncountable cylindric algebra in $\Nr_3\CA_{\omega}$, hence representable, and a finite atom structure
of another cylindric algebra.
We construct a finite product of the the uncountable cylindric  algebra; the product will be indexed by the atoms of the atom structure; 
the $\Ra$ reduct of the former will be as desired; it will be a full $\Ra$ reduct of an $\omega$ dimensional algebra and it has a 
complete elementary equivalent subalgebra not in 
$\Ra\CA_k$. This is the same idea for $\CA$, but in this case, and the other cases of its relatives, one spare dimension suffices.
 
This subalgebra is obtained by replacing one of the components of the product with an elementary 
{\it countable} algebra. First order logic will not see this cardinality twist, but a suitably chosen term 
$\tau_k$ not term definable in the language of relation algebras will, witnessing that the twisted algebra is not in $\Ra\CA_k$. 
For $\CA$'s and its relatives, as mentioned in the previous paragraph, we are lucky enough to have $k$ just $n+1,$
proving the most powerful result.

We concentrate on relation algebras. Let $\tau_k$ be an $m$-ary term of $\CA_k$, with $k$ large enough, 
and let $m\geq 2$ be its rank. We assume that $\tau_k$ is not definable
of relation algebras (so that $k$ has to be $\geq 5$); such terms exist.
Let $\tau$ be a term expressible in the language of relation algebras, such that
$\CA_k\models \tau_k(x_1,\ldots x_m)\leq \tau(x_1,\ldots x_m).$ (This is an implication between two first order formulas using $k$-variables).
Assume further that  whenever $\A\in {\bf Cs}_k$ (a set algebra of dimension $k$) is uncountable, 
and $R_1,\ldots R_m\in A$  are such that at least one of them is uncountable, 
then $\tau_k^{\A}(R_1\ldots R_m)$ is uncountable as well. 
(For $\CA$'s and its relatives, $k=n+1$, for $\CA$'s and $\Sc$s, it is a unary term, 
for polyadic algebras it also uses one extra dimension, it is a generalized composition hence it is a binary term.)

\begin{lemma} Let $V=(\At, \equiv_i, {\sf d}_{ij})_{i,j<3}$ be a finite cylindric atom structure, 
such that $|\At|\geq {}^33.$  
Let $L$ be a signature consisting of the unary relation 
symbols $P_0,P_1,P_2$ and
uncountably many tenary predicate symbols. 
For $u\in V$, let $\chi_u$
be the formula $\bigwedge_{u\in V}  P_{u_i}(x_i)$.  
Then there exists an $L$-structure $\M$ with the following properties:
\begin{enumarab}

\item $\M$ has quantifier elimination, i.e. every $L$-formula is equivalent
in $\M$ to a boolean combination of atomic formulas.

\item The sets $P_i^{\M}$ for $i<n$ partition $M$,

\item For any permutation $\tau$ on $3,$
$\forall x_0x_1x_2[R(x_0,x_1,x_2)\longleftrightarrow R(x_{\tau(0)},x_{\tau(1)}, x_{\tau(2)}],$
\item $\M \models \forall x_0x_1(R(x_0, x_1, x_2)\longrightarrow 
\bigvee_{u\in V}\chi_u)$, 
for all $R\in L$,

\item $\M\models  \exists x_0x_1x_2 (\chi_u\land R(x_0,x_1,x_2)\land \neg S(x_0,x_1,x_2))$ 
for all distinct tenary $R,S\in L$, 
and $u\in V.$

\item For $u\in V$, $i<3,$ 
$\M\models \forall x_0x_1x_2
(\exists x_i\chi_u\longleftrightarrow \bigvee_{v\in V, v\equiv_iu}\chi_v),$

\item For $u\in V$ and any $L$-formula $\phi(x_0,x_1,x_2)$, if 
$\M\models \exists x_0x_1x_2(\chi_u\land \phi)$ then 
$\M\models 
\forall x_0x_1x_2(\exists x_i\chi_u\longleftrightarrow 
\exists x_i(\chi_u\land \phi))$ for all $i<3$

\end{enumarab}
\end{lemma} 
\begin{demo}{Proof}  
We cannot apply Frassie's theorem to our signature, because it is uncountably infinite. What we do 
instead is that we introduce a new $4$ -ary relation symbol, that will be used to code the uncountably many tenary
relation symbols. Let $\cal L$ be the relational signature containing unary relation symbols
$P_0,\ldots, P_{3}$ and a $4$-ary relation symbol $X$. 
Let $\bold K$ be the class of all finite 
$\cal L$-structures $\D$ satsfying
\begin{enumarab}
\item The $P_i$'s are disjoint : $\forall x \bigvee_{i<j<4}(P_i(x)\land \bigwedge_
{j\neq i}\neg P_j(x)).$ 
\item $\forall x_0 x_1x_2x_3(X(x_0, x_1, x_2, x_3)\longrightarrow P_{3}(x_{3})\land \bigvee_{u\in V}\chi_u).$ 
\end{enumarab}
Then $\bold K$ contains countably many 
isomorphism types. Also it is easy to check that $\bold K$ is closed under
substructures and that $\bold K$ has the the amalgamation Property
From the latter it follows that it has the Joint Embedding Property.
Then there is a  countably infinite homogeneous $\cal L $-structure
$\N$ with age $\bold K$.  $\N$ has quantifier elimination, and obviously, so does 
any elementary extension of $\N$. 
$\bold K$ contains structures 
with arbitrarily large $P_{3}$-part, so $P_{3}^{\N}$ is infinite.
Let $\N^*$ be an elementary extension of $\N$ such that $|P_3|^{\N^*}|=|L|$, and
fix a bijection $*$ from the set of binary relation symbols of $L$ to $P_{3}^{\cal \N^*}.$
Define an $L$-structure $\M$ with domain 
$P_0^{\N^*}\cup P_1^{\N^*}\cup P_2^{\N^*}\cup \ldots P_{3}^{\N^*}$, by:
$P_i^{\M}=P_i^{\N^*}$ for $i<3$ and for
binary $R\in L$, and $\tau\in S_3$ 
$$\M\models R(\tau\circ \bar{a})\text { iff }{\N^*}\models X(\bar{a},R^*).$$ 
If $\phi(\bar {x})$ is any $L$-formula, let $\phi^*(\bar {x},\bar {R})$ be the 
$\cal L$-formula with parameters
$\bar {R}$ from $\N^*$ obtained from $\phi$ by replacing each atomic subformula 
$R(\bar{x})$ by $X(\bar{x},R^*)$ and relativizing quantifiers to $\neg P_n$, 
that is replacing $(\exists x)\phi(x)$ and $(\forall x)\phi(x)$ 
by $(\exists x)(\neg P_3(x)\to \phi(x))$
and $(\forall x)(\neg P_3(x)\to \phi(x)),$ respectively.
A straightforward induction on complexity of formulas
gives that for $\bar {a}\in \M$  
$$\M\models \phi(\bar {a})\text { iff } {\N^*}\models \phi^*(\bar {a},\bar {R}).$$

We show that $\M$ is as required. For quantifier elimination, if $\phi(\bar {x})$ is an
$L$-formula , then $\phi^*(\bar {x} ,\bar {R}^*)$ is equivalent in $\N^*$ 
to a quantifier free $\cal L$-formula
$\psi(\bar {x},\bar {R^*})$. 
Then replacing $\psi$'s atomic subformulas 
$X(x,y,z,R^*)$ by $R(x,y,z)$,
replacing all $X(t_0,\cdots t_3)$ not of this form by $\bot$ , replacing subformulas $P_3(x)$
by $\bot$, and $P_i(R^*)$ by $\bot$ if $i<3$ and $\top$ if $i=3$, gives a quantifier free $L$
-formula $\psi$ equivalent in $\M$ to $\phi$. 
(2) follows from the definition of satisfiability.
\par Let $$\sigma=\forall x(\neg P_3(x)\longrightarrow \bigvee_{i<3}
(P_i(x)\land \bigwedge_{j\neq i}\neg P_j(x))).$$
Then $K\models \sigma$, so $\M\models \sigma$ and ${\N^*}\models \sigma.$
It follows from the definition that $\M$ satisfies (3); (4) is similar.
\par For (5), let $u\in V$ and let $r,s\in P_3^{\M}$ be distinct. 
Take a finite $\cal L$-structure
$D$ with points $a_i\in P_{u_i}^D(i<3)$ and distinct $r',s'\in P_3^D$ with 
$$D\models X(a_0,a_1,a_2,r')\land \neg X(a_0,a_1,a_2,s').$$ 
Then $D\in K$, so $D$ embeds into $\M$. 
By homogeneity, we can assume that the embedding
takes $r'$ to $r$ and $s'$ to $s$. Therefore 
$${\M}\models \exists \bar {x}(\chi_u\land X(\bar {x},r)\land \neg X(\bar {x},s)),$$ 
where $\bar {x}=\langle x_0,x_1,x_2\rangle.$
Since $r,s$ were arbitrary and $\N^*$ is an elementary extension
of $\M$, we get that  
$${\N*}\models \forall yz
(P_3(y)\land P_3(z)\land y\neq z\longrightarrow \exists \bar {x}(\chi_u\land 
X(\bar {x},y)\land \neg (X(\bar {x},z))).$$
The result for $\M$ now follows.
\par Note that it follows from (4,5) that $P_i^{\M}\neq \emptyset$ for each $i<3$.
So it is clear that 
$$\M\models \forall x_0x_1x_2(\exists x_i\chi_u\longleftrightarrow 
\bigvee_{v\in {}V, v\equiv_i u}\chi_v);$$ giving
(6).

\par Finally consider $(7)$. Clearly, it is enough to show that for any
$\cal L$-formula $\phi(\bar {x})$ with parameters
$\bar {r}\in P_3^{\M}, u\in S_3, i<3$, we have 
$${\M}\models \exists \bar {x}(\chi_u\land \phi)\longrightarrow \forall \bar{x}
(\exists x_i(\chi_u\longrightarrow \exists x_i(\chi_u\land \phi)).$$
For simplicity of notation assume $i=2.$
Let $\bar {a} ,\bar {b}\in \M$ with 
$${\M}\models (\chi_u\land \phi)(\bar {a})\text { and } {\M}\models \exists x_2(\chi_u(\bar {b})).$$
We require 
$${\M}\models \exists x_2(\chi_u\land \phi)(\bar {b}).$$
\par It follows from the assumptions that
$${\M}\models P_{u_0}(a_0)\land P_{u_1}(a_1)\land a_0\neq a_1,\text { and }
{\M}\models P_{u_0}(b_0)\land P_{u_1}(b_1)\land b_0\neq b_1.$$
These are the only relations on $a_0a_r\bar {r} $ and on $b_0b_1\bar {r}$
(cf. property (4) of Lemma), so 
$$\theta^{-}=\{(a_0,b_0)(a_1,b_1) (r_l,r_l): l<|\bar {r}|\}$$ 
is a partial isomorphism
of $\M$. By homogeneity, it is induced by an automorphism $\theta$ of $\M$.
Let $c=\theta(\bar {a})=(b_0,b_1,\theta(a_2))$.
Then ${\M}\models (\chi_u\land \phi)(\bar {c}).$
Since $\bar {c}\equiv_2 \bar {b}$, we have 
${\M}\models \exists x_2(\chi_u\land \phi)(\bar {b})$
as required. 
\end{demo}

\begin{lemma}\label{term}
\begin{enumarab} 

\item For $\A\in \CA_3$ or $\A\in \SC_3$, there exist
a unary term $\tau_4(x)$ in the language of $\SC_4$ and a unary term $\tau(x)$ in the language of $\CA_3$
such that $\CA_4\models \tau_4(x)\leq \tau(x),$
and for $\A$ as above, and $u\in \At={}^33$, 
$\tau^{\A}(\chi_{u})=\chi_{\tau^{\wp(^nn)}(u).}$ 

\item For $\A\in \PEA_3$ or $\A\in \PA_3$, there exist a binary
term $\tau_4(x,y)$ in the language of $\SC_4$ and another  binary term $\tau(x,y)$ in the language of $\SC_3$
such that $PEA_4\models \tau_4(x,y)\leq \tau(x,y),$
and for $\A$ as above, and $u,v\in \At={}^33$, 
$\tau^{\A}(\chi_{u}, \chi_{v})=\chi_{\tau^{\wp(^nn)}(u,v)}.$


\item Let $k\geq 5$.
Then there exist a term $\tau_k(x_1,\ldots x_m)$ in the language of $\CA_k$ and a term $\tau(x_1,\ldots, x_m)$ in the language of $\RA$,
expressible in $\CA_3$, such that
$\CA_k\models \tau_k(x_1,\ldots x_m)\leq \tau(x_1,\ldots x_m),$
and for  $\A$ as above, and $u_1,\ldots u_m\in \At$, 
$\tau^{\A}(\chi_{u_1},\ldots \chi_{u_m})=\chi_{\tau^{\Cm\At}(u_1,\ldots u_m)}.$ 
\end{enumarab}
\end{lemma}

\begin{proof} 

\begin{enumarab}

\item For all reducts of polyadic algebras, these terms are given in \cite{FM}, and \cite{MLQ}.
For cylindric algebras $\tau_4(x)={}_3 s(0,1)x$ and $\tau(x)=s_1^0c_1x.s_0^1c_0x$.
For polyadic algebras, it is a little bit more complicated because the former term above is definable.
In this case we have $\tau(x,y)=c_1(c_0x.s_1^0c_1y).c_1x.c_0y$, and $\tau_4(x,y)=c_3(s_3^1c_3x.s_3^0c_3y)$.

\item For relation algebras, we take the term corresponding to the following generalization of Johnson's $Q$'s. 
Given $1\leq n<\omega$ and $n^2$ tenary relations
we define
$$Q(R_{ij}: i,j<n)(x_0, x_2, x_3)\longleftrightarrow$$ 
$$\exists z_0\ldots z_{n-1}(z_0=x\land z_1=y\land 
z_2=z\land \bigwedge_{i,j,l<n} 
R_{ij}(z_i, z_j, z_l).$$
The term $\tau$ is not difficult to find.
\end{enumarab}
\end{proof}

\begin{theorem}
\begin{enumarab} 
\item There exists $\A\in \Nr_3\QEA_{\omega}$
with an elementary equivalent cylindric  algebra, whose $\SC$ reduct is not in $\Nr_3\SC_4$.
Furthermore, the latter is a complete subalgebra of the former.
 
\item There exists a relation algebra $\A\in \Ra\CA_{\omega}$, with an elementary equivalent relation algebra not in $\Ra\CA_k$.
Furthermore, the latter is a complete subalgebra of the former.
\end{enumarab}
\end{theorem}
\begin{proof} Let $\L$ and $\M$ as above. Let 
$\A_{\omega}=\{\phi^M: \phi\in \L\}.$ 
Clearly $\A_{\omega}$ is a locally finite $\omega$-dimensional cylindric set algebra.
For the first part, we prove the theorem for $\CA$; and its relatives.

Then $\A\cong \Nr_3\A_{\omega}$, the isomorphism is given by 
$$\phi^{\M}\mapsto \phi^{\M}.$$
Quantifier elimination in $\M$ guarantees that this map is onto, so that $\A$ is the full $\Ra$ reduct.

For $u\in {}V$, let $\A_u$ denote the relativisation of $\A$ to $\chi_u^{\M}$
i.e $$\A_u=\{x\in A: x\leq \chi_u^{\M}\}.$$ $\A_u$ is a boolean algebra.
Also  $\A_u$ is uncountable for every $u\in V$
because by property (iv) of the above lemma, 
the sets $(\chi_u\land R(x_0,x_1,x_2)^{\M})$, for $R\in L$
are distinct elements of $\A_u$.  

Define a map $f: \Bl\A\to \prod_{u\in {}V}\A_u$, by
$$f(a)=\langle a\cdot \chi_u\rangle_{u\in{}V}.$$

Here, and elsewhere, for a relation algebra $\C$, $\Bl\C$ denotes its boolean reduct.
We will expand the language of the boolean algebra $\prod_{u\in V}\A_u$ by constants in 
such a way that
the relation algebra reduct of $\A$ becomes interpretable in the expanded structure.
For this we need.

Let $\P$ denote the 
following structure for the signature of boolean algebras expanded
by constant symbols $1_{u}$ for $u\in V$ 
and ${\sf d}_{ij}$ for $i,j\in 3$: 
We now show that the relation algebra reduct of $\A$ is interpretable in $\P.$  
For this it is enough to show that 
$f$ is one to one and that $Rng(f)$ 
(Range of $f$) and the $f$-images of the graphs of the cylindric algebra functions in $\A$ 
are definable in $\P$. Since the $\chi_u^{\M}$ partition 
the unit of $\A$,  each $a\in A$ has a unique expression in the form
$\sum_{u\in {}V}(a\cdot \chi_u^{\M}),$ and it follows that 
$f$ is boolean isomorphism: $bool(\A)\to \prod_{u\in {}V}\A_u.$
So the $f$-images of the graphs of the boolean functions on
$\A$ are trivially definable. 
$f$ is bijective so $Rng(f)$ is 
definable, by $x=x$. For the diagonals, $f({\sf d}_{ij}^{\A})$ is definable by $x={\sf d}_{ij}$.

Finally we consider cylindrifications for $i<3$. Let $S\subseteq {}V$ and  $i<3$, 
let $t_S$ be the closed term
$$\sum\{1_v: v\in {}V, v\equiv_i u\text { for some } u\in S\}.$$
Let
$$\eta_i(x,y)=\bigwedge_{S\subseteq {}V}(\bigwedge_{u\in S} x.1_u\neq 0\land 
\bigwedge_{u\in {}V\smallsetminus S}x.1_u=0\longrightarrow y=t_S).$$
We claim that for all $a\in A$, $b\in P$, we have 
$$\P\models \eta_i(f(a),b)\text { iff } b=f({\sf c}_i^{\A}a).$$
To see this, let $f(a)=\langle a_u\rangle_{u\in {}V}$, say. 
So in $\A$ we have $a=\sum_ua_u.$
Let $u$ be given; $a_u$ has the form $(\chi_i\land \phi)^{\M}$ for some $\phi\in L^3$, so
${\sf c}_i^A(a_u)=(\exists x_i(\chi_u\land \phi))^{\M}.$ 
By property (vi), if  $a_u\neq 0$, this is 
$(\exists x_i\chi_u)^M$; by property $5$,
this is $(\bigvee_{v\in {}V, v\equiv_iu}\chi_v)^{\M}.$
Let $S=\{u\in {}V: a_u\neq 0\}.$
By normality and additivity of cylindrifications we have,
$${\sf c}_i^A(a)=\sum_{u\in {}V} {\sf c}_i^Aa_u=
\sum_{u\in S}{\sf c}_i^Aa_u=\sum_{u\in S}(\sum_{v\in {}V, v\equiv_i u}\chi_v^{\M})$$
$$=\sum\{\chi_v^{\M}: v\in {}V, v\equiv_i u\text { for some } u\in S\}.$$
So $\P\models f({\sf c}_i^{\A}a)=t_S$. Hence $\P\models \eta_i(f(a),f({\sf c}_i^{\A}a)).$
Conversely, if $\P\models \eta_i(f(a),b)$, we require $b=f({\sf c}_ia)$. 
Now $S$ is the unique subset of $V$ such that 
$$\P\models \bigwedge_{u\in S}f(a)\cdot 1_u\neq 0\land \bigwedge_{u\in {}V\smallsetminus S}
f(a)\cdot 1_u=0.$$  So we obtain 
$$b=t_S=f({\sf c}_i^Aa).$$
The rest is the same as in \cite{MLQ} while for other relatives, the idea implemented is also the same;
one just uses the corresponding terms as in lemma \ref{term}.

For relation algebras we proceed as follows:
Now the $\Ra$ reduct of $\A$ is a generalized reduct of $\A$, hence $\P$ is first order interpretable in $\Ra\A$, as well.
 It follows that there are closed terms $1_{u, v},d_{i,j}$ and a formula $\eta$ built out of these closed terms such that 
$$\P\models \eta(f(a), b, c)\text { iff }b= f(a\circ c),$$
where the composition is taken in $\Ra\A$.

We have proved that $\Ra\A$ is interpretable in $\P$.
Furthermore it is easy to see that the interpretation 
is two dimensional and quantifier free.

For each $u\in V$, choose any countable boolean elementary 
complete subalgebra of $\A_{u}$, $\B_{u}$ say.
Le $u_i: i<m$ be elements in $V$ and let
$$Q=(\prod_{u_i: i<m}\A_{u_i}\times \B_{\tau^{\Cm\At}(u_1,\ldots u_m)}\times \prod_{u\in {}V\smallsetminus \{u_1,\ldots u_m, \tau^{\Cm\At\A}(u_1,\ldots u_m)\}} 
\A_u), 1_{u,v},d_{ij})_{u,v\in {}V,i,j<3}\equiv$$  
$$(\prod_{u\in V} \A_u, 1_{u,v}, {\sf d}_{ij})_{u\in V, i,j<3}=P.$$

Let $\B$ be the result of applying the interpretation given above to $Q$.
Then $\B\equiv \Ra\A$ as relation  algebras, furthermore $\Bl\B$ is a complete subalgebra of $\Bl\A$. 
Assume for contradiction that $\B=\Ra\D$ with $\D\in \CA_k$. Let $u_1,\ldots u_m\in V$ be such that 
$\tau_k^{\D}(\chi_{u_1},\ldots \chi_{u_n})$, is uncountable in $\D$. 

Because $\B$ is a full $\RA$ reduct, 
this set is contained in $\B.$ 

For simplicity assume that $\tau^{\Cm\At}(u_1\ldots u_m)=Id.$
On the other hand for $x_i\in B$, with $x_i\leq \chi_{u_i}$, we have 
$$\tau_k^{\D}(x_1,\ldots x_m)\leq \tau(x_1\ldots x_m)\in \tau(\chi_{u_1},\ldots {\chi_{u_m}})=\chi_{\tau(u_1\ldots u_m)}=\chi_{Id}.$$

But this is a contradiction, since  $B_{Id}=\{x\in B: x\leq \chi_{Id}\}$ is  countable.

\end{proof}

\section{Second question}


In the previous section we showed that it can be the case that an algebra does not neatly embed in another algebra in one extra dimension 
(for finite as well as for infinite dimensional 
algebras). Not only that, but if an algebra neatly embeds into $k$ extra dimensions, 
there is no gaurantee whatsoever, that it neatly embeds into $k+1$ extra dimensions, and in fact we know that there are cases that it cannot.
There are known examples for finite dimensions, and these can be used to
prove the analgous result for infinite dimensions. 
Such results are related to {\it completeness} results for modifications of first oder logic, or rather {\it incompleteness} ones.

Here we adress a variation, or rather a natural extension of this question, namely, suppose that an algebra does neatly embed into extra dimensions, 
is the big algebra, or the dilation, uniquely defined over the small algebra, if the latter as a set generates the former using all the extra dimensions? 
The more spare dimensions we have, the more likely that the original algebra has more control on the dilation, because 
it codes more hidden dimensions, and so in a sense, it 'defines' more, it has a larger expressive power. 

The strongest case would be the neat embedding in $\omega$ extra dimensions (this  often 
implies representability of the algebra in question) 
and  is equivalent, implementing a standard ultraproduct construction,  to embedding into an algebra having any transfinite larger 
ordinal as its dimension. (In cylindric algebra, 
for any two infinite ordinals $\alpha<\beta$ and $n\in \omega$, $\Nr_n\CA_{\alpha}=\Nr_n\CA_{\beta}$, so that getting to $\omega$ is getting over the 
hurdle). 

We will see that such uniqueness is actually equivalent to that the neat reduct operator  formulated 
in an appropriate way as a functor has a right adjoint, and that this in turn is strongly related to various amalgamation properties of the class 
in question. So while in the first case we deal with completeness theorems, or rather the lack thereof, 
in the second we deal with the interpolation property for the corresponding algebraisable logic. 

This algebraic approach via neat embeddings suggests that the two notions are not unrelated.
The mere fact that a Henkin construction can prove both completenes and interpolation, not only in the context of first order logic, 
also emphasizes this strong tie.

Another algebraic manifestation of this link, is the recurrent phenomena in algebraic logic, 
that for several algebraisatons of variants of first order logic, completeness and interpolation come hand in hand.  
This happens for example in Keislers logic \cite{DM}, \cite{super} and their countable reducts studied by Sain \cite{Sain}.

To prove a  completeness theorem using the methdology of algebraic logic, one needs to show that a certain 
abstract class of algebras defined syntactically (via a simple set of first order axioms preferably equations), consists solely
of representable algebras providing a complete semantics. This often appeals to the neat embedding theorem of Henkin.
First one proves that the algebra in question neatly embeds into an $\omega$ dilation, 
and in such an abundance of spare dimensions one can construct so-called {\it Henkin ultrafilters} that eliminate cylindrifiers,
and then the representations becomes realily definable by reducing it basically to the propositional part.

In such a context interpolation is provided by showing that this neat embedding is {\it faithfull}, 
the small algebra essentially determines the structure of the bigger one.
Expressed otherwise, for the interpolation property to hold a necessary (and in some cases also sufficient) 
condition is that algebraic terms 
definable using the added spare dimensions, to eliminate cylindrifiers, are already term definable.
The typical situation (even for cylindric algebras with no restriction whatsoever, like local finitenes) 
an interpolant  can {\it always} be found, but the problem is that  it might, and indeed there are situations where it must,
resort to extra dimensions (variables). This happens in the case for instance of the so called finitary logics of infinitary relations \cite{typeless}. 
This interpolant then becomes  term definable in higher dimensions, but when these terms 
are actually coded by ones using only the original amount of dimensions, 
we get the desired interpolant.  One way of getting around the un warranted spare dimensions in the interpolant, 
is to introduce new connectives that code extra dimensions, in which case in the original language any implication can be interpolated
by a formula using the same umber of varibles ( and common symbols) but  possibly uses the new connectives. 
In \cite{typeless} this is done to prove an interpolation theorem for the severely incomplete typless
logics studied in \cite{HMT2}.

Categorially, and indeed intuitively, this means that the dilation functor is invertible, the interpolant is found in a dilation, 
but the inverse, the neat reduct functor gets us back to our base, to our original algebra.
To formulate our results, we need some preparations.
\begin{definition}
\begin{enumarab}
\item $K$ has the \emph{Amalgamation Property } if for all $\A_1, \A_2\in K$ and monomorphisms
$i_1:\A_0\to \A_1,$ $i_2:\A_0\to \A_2$
there exist $\D\in K$
and monomorphisms $m_1:\A_1\to \D$ and $m_2:\A_2\to \D$ such that $m_1\circ i_1=m_2\circ i_2$.
\item If in addition, $(\forall x\in A_j)(\forall y\in A_k)
(m_j(x)\leq m_k(y)\implies (\exists z\in A_0)(x\leq i_j(z)\land i_k(z) \leq y))$
where $\{j,k\}=\{1,2\}$, then we say that $K$ has the superamalgamation property $(SUPAP)$.
\end{enumarab}
\end{definition}

\begin{definition} An algebra $\A$ has the {\it strong interpolation theorem}, $SIP$ for short, if for all $X_1, X_2\subseteq A$, $a\in \Sg^{\A}X_1$,
$c\in \Sg^{\A}X_2$ with $a\leq c$, there exist $b\in \Sg^{\A}(X_1\cap X_2)$ such that $a\leq b\leq c$.
\end{definition}

For an algebra $\A$, $Co\A$ denotes the set of congruences on $\A$.
\begin{definition}

An algebra $\A$ has the {\it congruence extension property}, or $CP$ for short,
 if for any $X_1, X_2\subset A$
if $R\in Co \Sg^{\A}X_1$ and $S\in Co \Sg^{A}X_2$ and
$$R\cap {}^2\Sg^{A}(X_1\cap X_2)=S\cap {}^2\Sg^{\A}{(X_1\cap X_2)},$$
then there exists a congruence $T$ on $\A$ such that
$$T\cap {}^2 \Sg^{\A}X_1=R \text { and } T\cap {}^2\Sg^{\A}{(X_2)}=S.$$

\end{definition}


Maksimova and Mad\'arasz \cite{Mad}, \cite{Mak}, proved that if interpolation holds in free algebras of a variety, then the variety has the
superamalgamation property.
Using a similar argument, we prove this implication in a slightly more
general setting. But first an easy  lemma:

\begin{lemma} Let $K$ be a class of $BAO$'s. Let $\A, \B\in K$ with $\B\subseteq \A$. Let $M$ be an ideal of $\B$. We then have:
\begin{enumarab}
\item $\Ig^{\A}M=\{x\in A: x\leq b \text { for some $b\in M$}\}$
\item $M=\Ig^{\A}M\cap \B$
\item if $\C\subseteq \A$ and $N$ is an ideal of $\C$, then
$\Ig^{\A}(M\cup N)=\{x\in A: x\leq b+c\ \text { for some $b\in M$ and $c\in N$}\}$
\item For every ideal $N$ of $\A$ such that $N\cap B\subseteq M$, there is an ideal $N'$ in $\A$
such that $N\subseteq N'$ and $N'\cap B=M$. Furthermore, if $M$ is a maximal ideal of $\B$, then $N'$ can be taken to be a maximal ideal of $\A$.

\end{enumarab}
\end{lemma}
\begin{demo}{Proof} Only (iv) deserves attention. The special case when $n=\{0\}$ is straightforward.
The general case follows from this one, by considering
$\A/N$, $\B/(N\cap \B)$ and $M/(N\cap \B)$, in place of $\A$, $\B$ and $M$ respectively.
\end{demo}
The previous lemma will be frequently used without being explicitly mentioned.

\begin{theorem}\label{supapgeneral} Let $K$ be a class of $BAO$'s such that $\mathbf{H}K=\mathbf{S}K=K$.
Assume that for all $\A, \B, \C\in K$, inclusions
$m:\C\to \A$, $n:\C\to \B$, there exist $\D$ with $SIP$ and $h:\D\to \C$, $h_1:\D\to \A$, $h_2:\D\to \B$
such that for $x\in h^{-1}(\C)$,
$$h_1(x)=m\circ h(x)=n\circ h(x)=h_2(x).$$
Then $K$ has $SUPAP$. In particular, if $K$ is a variety and the free algebras have $SIP$ then $V$ has $SUPAP$.
\bigskip
\bigskip
\bigskip
\bigskip
\bigskip
\bigskip
\bigskip
\bigskip
\bigskip
\bigskip
\bigskip
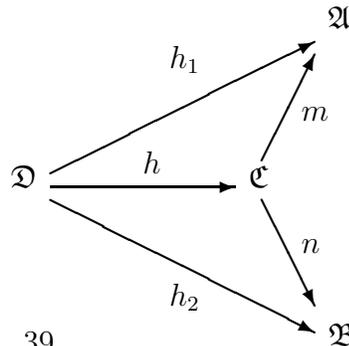
\begin{picture}(10,0)(-30,70)
\thicklines
\put (-10,0){$\D$}
\put(5,0){\vector(1,0){70}}\put(80,0){$\C$}
\put(5,5){\vector(2,1){100}}\put(110,60){$\A$}
\put(5,-5){\vector(2,-1){100}}\put(110,-60){$\B$}
\put(85,10){\vector(1,2){20}}
\put(85,-5){\vector(1,-2){20}}
\put(40,5){$h$}
\put(100,25){$m$}
\put(100,-25){$n$}
\put(50,45){$h_1$}
\put(50,-45){$h_2$}
\end{picture}
\end{theorem}
\bigskip
\bigskip
\begin{proof} Let $\D_1=h_1^{-1}(\A)$ and $\D_2=h_2^{-1}(\B)$. Then $h_1:\D_1\to \A$, and $h_2:\D_2\to \B$.

Let $M=ker h_1$ and $N=ker h_2$, and let
$\bar{h_1}:\D_1/M\to \A, \bar{h_2}:\D_2/N\to \B$ be the induced isomorphisms.

Let $l_1:h^{-1}(\C)/h^{-1}(\C)\cap M\to \C$ be defined via $\bar{x}\to h(x)$, and
$l_2:h^{-1}(\C)/h^{-1}(\C)\cap N$ to $\C$ be defined via $\bar{x}\to h(x)$.
Then those are well defined, and hence
$k^{-1}(\C)\cap M=h^{-1}(\C)\cap N$.
Then we show that $\P=\Ig(M\cup N)$ is a proper ideal and $\D/\P$ is the desired algebra.
Now let $x\in \mathfrak{Ig}(M\cup N)\cap \D_1$.
Then there exist $b\in M$ and $c\in N$ such that $x\leq b+c$. Thus $x-b\leq c$.
But $x-b\in \D_1$ and $c\in \D_2$, it follows that there exists an interpolant
$d\in \D_1\cap \D_2$  such that $x-b\leq d\leq c$. We have $d\in N$
therefore $d\in M$, and since $x\leq d+b$, therefore $x\in M$.
It follows that
$\mathfrak{Ig}(M\cup N)\cap \D_1=M$
and similarly
$\mathfrak{Ig}(M\cup N)\cap \D_2=N$.
In particular $P=\mathfrak{Ig}(M\cup N)$ is a proper ideal.

Let $k:\D_1/M\to \D/P$ be defined by $k(a/M)=a/P$
and $h:\D_2/N\to \D/P$ by $h(a/N)=a/P$. Then
$k\circ m$ and $h\circ n$ are one to one and
$k\circ m \circ f=h\circ n\circ g$.
We now prove that $\D/P$ is actually a
superamalgam. i.e we prove that $K$ has the superamalgamation
property. Assume that $k\circ m(a)\leq h\circ n(b)$. There exists
$x\in \D_1$ such that $x/P=k(m(a))$ and $m(a)=x/M$. Also there
exists $z\in \D_2$ such that $z/P=h(n(b))$ and $n(b)=z/N$. Now
$x/P\leq z/P$ hence $x-z\in P$. Therefore  there is an $r\in M$ and
an $s\in N$ such that $x-r\leq z+s$. Now $x-r\in \D_1$ and $z+s\in\D_2,$ it follows that there is an interpolant
$u\in \D_1\cap \D_2$ such that $x-r\leq u\leq z+s$. Let $t\in \C$ such that $m\circ
f(t)=u/M$ and $n\circ g(t)=u/N.$ We have  $x/P\leq u/P\leq z/P$. Now
$m(f(t))=u/M\geq x/M=m(a).$ Thus $f(t)\geq a$. Similarly
$n(g(t))=u/N\leq z/N=n(b)$, hence $g(t)\leq b$. By total symmetry,
we are done.
\end{proof}

For a cardinal  $\beta>0$, $L\subseteq K_{\alpha}$  and $\rho:\beta\to \wp(\alpha),$ $\Fr_{\beta}^{\rho}L$
stands for the dimension restricted $L$ free algebra on $\beta$ generators. 
The sequence $\langle \eta/Cr_{\beta}^{\rho}L: \eta<\beta\rangle$
$L$-freely generates $\Fr_{\beta}^{\rho}L$, cf. \cite{HMT1} Theorem 2.5.35. 
$\Fr_{\beta}^{\rho}\CA_{\alpha}$ is treated in \cite{P} 
under the name of free algebras over $L$ subject to certain defining relations, cf. \cite{P} Definition 1.1.5.
The super amalgamation property, due to Maksimova, is rarely applied to algebraisations of first order logic; yet in this direction we have:

\begin{theorem} Let $\kappa$ be any ordinal $>1$. 
Let $\bold M=\{\A\in K_{\kappa+\omega}: \A=\Sg^{\A}\Nr_{\kappa}\A\}$.
Then $\bold M$ has $SUPAP$. 
\end{theorem}
\begin{proof} First one proves that dimension restricted free algebras have the strong interpolation property, and then shows
that they satisfy the conditions in the previous theorem. The first part is proved in \cite{MStwo}.
For the second part we proceed as follows.
Let $\kappa$ be an arbitrary ordinal $>0.$
Let $\A,\B$ and $\C$ be in $K$ and $f:\C\to \A$ and $g:\C\to \B$ be monomorphisms. We want to find an amalgam.
Let $\langle a_i: i\in I\rangle$ be an enumeration of $\A$ and $\langle b_i:i\in J\rangle$ 
be an enumeration of $\B$ such that
$\langle c_i: i\in I\cap J\rangle$ is an enumeration of $\C$ with $f(c_i)=a_i$ and $g(c_i)=b_i$ for all $i\in I\cap J$. 
Then $\Delta a_i=\Delta b_i$ for all $i\in I\cap J$. Let $k=I\cup J$. 
Let $\xi$ be a bijection from $k$ onto a cardinal $\mu$.
Let $\rho\in {}^{\mu}\wp(\kappa+\omega)$ be defined by
$\rho \xi i=\Delta a_i$ for $i\in I$ and $\rho \xi j=\Delta b_j$ for $j\in J$. Then $\rho$ is well defined.
Let $\beta=\kappa+\omega$ and let $\Fr=\Fr_{\mu}^{\rho}K_{\beta}$.
Let $\Fr^{I}$ be the subalgebra of $\Fr$ generated by $\{\xi i/Cr_{\mu}^{\rho}K_{\beta} :i \in I\}$ and let 
$\Fr^J$ be the subalgebra generated by $\{\xi i/Cr_{\mu}^{\rho}K_{\beta} :i \in J\}$.
To avoid cumbersome notation we write $\xi i$ instead of $\xi i/Cr_{\mu}^{\rho}K_{\beta}$ and similarly for $\xi j$.
No confusion is likely to ensue. Then there exists 
a homomorphism from $\Fr^{I}$ onto $\A$
such that
$\xi i\mapsto a_i$ $(i\in I)$
and similarly 
a homomorphism from $\Fr^{J}$ into $\B$ such that
$\xi j\mapsto b_j$ $(j\in I).$

\end{proof}

The cylindric algebra anlagoue of the following theorem was proved for cylindric algebras using a very strong result, 
namely that the class of representable algebras does not have $AP$. 
This result is proved for quasi polyadic algebras by the present author \cite{amal} 
so the same thread of argument reported in \cite{Sayed}  
together with the latter result gives the required. 

However, here we follow a different route that is admittedly slighly long and winding.
But it has many advantages. First it avoids such a strong result.
Second, the proof shows explicity the connections between the unique neat embedding property,
existence of universal maps, interpolation properties and congruence extension properties in free algebras 
and ultimately the amalgamation property.

$\QEA$ denotes quasi-polyadic algebras and $\RQEA$ denotes the representable $\QEA$s.

\begin{theorem} For $\alpha\geq \omega$, the following hold:
\begin{enumroman}
\item There exists $\A\in \RQEA_{\alpha}$, $\B\in \QEA_{\alpha+\omega}$ such that $\A\subseteq \Nr_{\alpha}\B$ $A$ generates $\B$ 
but $\A\neq \Nr_{\alpha}\B$.
\item There exist $\A\in \RQEA_{\alpha}$, a $\B\in \QEA_{\alpha+\omega}$ and  an ideal $J\subseteq \B$, such that $\A\subseteq \Nr_{\alpha}\B$, 
$A$ generates $\B$, but $\Ig^{\B}(J\cap A)\neq \B$.
\item There exist $\A, \A' \in \RQEA_{\alpha}$, $\B, \B' \in \QEA_{\alpha+\omega}$ with 
embeddings $e_A:\A \to \Nr_\alpha \B$ and $e_{A'}:\A' \to \Nr_\alpha \B'$ 
such that $ \Sg^\B e_A(A) = \B$ and $ \Sg^{\B'} e_{A'}(A) = \B'$, and an isomorphism $ i : \A \longrightarrow\A'$ 
for which there exists no isomorphism $\bar{i} : \B \longrightarrow \B'$ such that $\bar{i} 
\circ e_A =e_{A'} \circ i$.
\end{enumroman}
\end{theorem}

\begin{demo}{Proof} Assume that $(i)$ is false. Then we can prove (*) (the negation of (ii)).
 
(*)  For all $\A\in \RQEA_{\alpha}$, $\B\in \QEA_{\alpha+\omega}$ and  ideal $J\subseteq \B$, if $\A\subseteq \Nr_{\alpha}\B$, 
and $A$ generates $\B$, then $\Ig^{\B}(J\cap A)=J$.

From (*) we prove that (iii) is false, from which we reach a contradiction.
This will prove (i) and (ii) and (iii). 

Since $\A$ generates $\B$ we have $\A=\Nr_{\alpha}\B$. Clearly $\Ig^{\B}(J\cap A)\subseteq J$. 
Conversely, let $x\in J$. Then $\Delta x\sim \alpha$ is finite,
call it $\Gamma$. So ${\sf c}_{(\Gamma)}x\in \Nr_{\alpha}\B$, so, by assumption, it is in $\A$. 
Hence ${\sf c}_{(\Gamma)}x\in \A\cap J$. But $x\leq {\sf c}_{(\Gamma)}x$ we get the 
first required.  Now we can assume (*) . Let $\beta=\alpha+\omega$.


Let  $\A, \A' \in \RQEA_{\alpha}$, $\B, \B' \in \QEA_{\beta}$ and assume that $e_A, e_{A'}$ are embeddings from $\A, \A'$ into   $\Nr_\alpha \B,
\Nr_\alpha \B'$, respectively, such that
$ \Sg^\B (e_A(A)) = \B$
and $ \Sg^{\B'} (e_{A'}(A')) = \B',$
and let $ i : \A \longrightarrow \A'$ be an isomorphism. 
Let $\mu=|A|$. Let $x$ be a  bijection  from $\mu$ onto $A$.   
Let $y$ be a bijection from $\mu$ onto $A'$,
such that $ i(x_j) = y_j$ for all $j < \mu$.
Let $\rho = \langle \Delta ^{(\A)} x_j : j < \mu\rangle$, $\D = \Fr^{(\rho)}_{\mu} \QEA_{\beta}$, $g_\xi =
\xi/Cr^{(\rho)}_{\mu} \QEA_{\beta}$ for all  $\xi< \mu$
and $\C = \Sg^{\Rd_\alpha \D} \{ g_\xi : \xi < \mu \}.$
Then $\C \subseteq \Nr_\alpha \D,\  C \textrm{ generates } \D
~~\textrm{and }~~ \C \in \RQEA_{\alpha}.$
There exist $ f \in Hom (\D, \B)$ and $f' \in Hom (\D,
\B')$ such that
$f (g_\xi) = e_A(x_\xi)$ and $f' (g_\xi) = e_{A'}(y_\xi)$ for all $\xi < \mu.$
Note that $f$ and $f'$ are both onto. We now have 
$e_A \circ i^{-1} \circ e_{A'}^{-1} \circ ( f'\upharpoonleft \C) = f \upharpoonleft \C.$
Therefore $ Ker f' \cap \C = Ker f \cap \C.$
Hence $\Ig(Ker f' \cap \C) = \Ig(Ker f \cap \C).$
So, by assumption, $Ker f'  = Ker f.$
Let $y \in B$, then there exists $x \in D$ such that $y = f(x)$. Define $ \bar{i} (y) = f' (x).$
The map is well defined and is as required.
That is, $\bar{i} \circ e_A =e_{A'} \circ i$.
From now on, we assume that (ii) is false. Hence we assume the following (**):

For all $\A, \A' \in \RQEA_{\alpha}$, $\B, \B' \in \QEA_{\beta}$ $e_A, e_{A'}$ are embeddings from $\A, \A'$ into   $\Nr_\alpha \B,
\Nr_\alpha \B'$, respectively, such that
$ \Sg^\B (e_A(A)) = \B$
and $ \Sg^{\B'} (e_{A'}(A')) = \B',$ and isomorphism $ i : \A \longrightarrow \A'$, there exist $\bar{i}: \B\to \B$' such that $\bar{i}\circ e_A=e_A'\circ i$. 
We will arrive at a contradiction.

We start with the following claim that free algebras satisfy. Recall that $S\in Co\A$ means that $S$ is a congruence relation on $\A$. 
\begin{athm}{Claim 1}
Let $\mu$ be a cardinal $>0$. Let $\A=\Fr_{\mu}\RQEA_{\alpha}$. 
For any $X_1, X_2\subseteq \mu$   
if $R\in Co \A^{(X_1)}$ and $S\in Co \A^{(X_2)}$ and 
$$R\cap {}^2A^{(X_1\cap X_2)}=S\cap {}^2A^{(X_1\cap X_2)},$$
then there exists a congruence $T$ on $\A$ such that
$$T\cap {}^2 A^{(X_1)}=R \text { and } T\cap {}^2A^{(X_2)}=S.$$ 
\end{athm}

\begin{demo}{Proof} For $R\in Co\A$ and $X\subseteq A$, by $(\A/R)^{(X)}$ we understand the subalgebra of 
$\A/R$ generated by $\{x/R: x\in X\}.$ Let $\A$, $X_1$, $X_2$, $R$ and $S$ be as specified in the claim.
Define $$\theta: \A^{(X_1\cap X_2)}\to \A^{(X_1)}/R$$
by $$a\mapsto a/R.$$
Then $ker\theta=R\cap {}^2A^{(X_1\cap X_2)}$ and $Im\theta=(\A^{(X_1)}/R)^{(X_1\cap X_2)}$.
It follows that $$\bar{\theta}:\A^{(X_1\cap X_2)}/R\cap {}^2A^{(X_1\cap X_2)}\to (\A^{(X_1)}/R)^{(X_1\cap X_2)}$$
defined by
$$a/R\cap {}^{2}A^{X_1\cap X_2)}\mapsto a/R$$
is a well defined isomorphism.
Similarly
$$\bar{\psi}:\A^{(X_1\cap X_2)}/S\cap {}^2A^{(X_1\cap X_2)}\to (\A^{(X_2)}/S)^{(X_1\cap X_2)}$$
defined by
$$a/S\cap {}^{2}A^{X_1\cap X_2)}\mapsto a/S$$
is also a well defined isomorphism.
But $$R\cap {}^2A^{(X_1\cap X_2)}=S\cap {}^2A^{(X_1\cap X_2)},$$
Hence
$$\phi: (\A^{(X_1)}/R)^{(X_1\cap X_2)}\to (\A^{(X_2)}/S)^{(X_1\cap X_2)}$$
defined by
$$a/R\mapsto a/S$$
is a well defined isomorphism. 
Now
$(\A^{(X_1)}/R)^{(X_1\cap X_2)}$ embeds into ${\A}^{(X_1)}/R$ via the inclusion map; it also embeds in $\A^{(X_2)}/S$ via $i\circ \phi$ where $i$ 
is also the inclusion map.
For brevity let $\A_0=(\A^{(X_1)}/R)^{(X_1\cap X_2)}$, $\A_1={\A}^{(X_1)}/R$ and $\A_2={\A}^{(X_2)}/S$ and $j=i\circ \phi$.
Then $\A_0$ embeds in $\A_1$ and $\A_2$ via $i$ and $j$ respectively.
We now use (**) to show that there exists $\B\in \RQEA_{\alpha}$ and monomorphisms $f$ and $g$ from $\A_1$ and $\A_2$ respectively to 
$\B$ such that
$f\circ i=g\circ j$.
By the Neat embedding Theorem, there exist $\A_0^+, \A_1^+, \A_2^+\in \CA_{\alpha+\omega}$, $e_1:\A_1\to \Nr_{\alpha}\A_1^+,$ 
$e_2:\A_2\to  \Nr_{\alpha}\A_2^+$ and $e_0:\A_0\to \Nr_{\alpha}\A_0^+$.
We can assume that $\Sg^{\A_1^+}e_1(A_1)=\A_1^+$ and similarly for $\A_2^+$ and $\A_0^+$.
Let $i(A_0)^+=\Sg^{\A_1^+}e_1(i(A_0))$ and $j(A_0)^+=\Sg^{\A_2^+}e_2(j(A_0)$, then by (**) there exist 
$\bar{i}:\A_0^+\to i(A_0)^+$ and $\bar{j}:\A_0^+\to j(A_0)^+$ such that 
$(e_1\upharpoonright i(A_0))\circ i=\bar{i}\circ e_0$ and $(e_2\upharpoonright j(A_0))\circ j=\bar{j}\circ e_0$.
Now $K$ as in theorem 5 has $SUPAP$, hence there is a $\D^+$ in $\K$ and $k:\A^+\to \D^+$ and $h:\B^+\to \D^+$ such that
$k\circ \bar{i}=h\circ \bar{j}$. Let $\D=\Nr_{\alpha}\D^+$. Then $f=k\circ e_1:\A_1\to \Nr_{\alpha}\D$ and
$g=h\circ e_2:\A_2\to \Nr_{\alpha}\D$ are one to one and
$k\circ e_1 \circ i=h\circ e_2\circ j$.
\end{demo}
\newpage

\begin{picture}(20,0)(-70,70)

\thicklines \put(10,10){\vector(1,1){40}}\put (-5,0){$A_1$} \put
(-5,30){$k\circ e_1$}

\put (60,50){$Nr_\alpha D^+$}\put (80,70){\vector(0,1){80}}\put
(60,105){$Id$}\put (75,160){$D^+$}

\put(140,10){\vector(-1,1){40}} \put (130,30){$h\circ e_2$} \put (150,0){$A_2$}

 \put(170,0){\vector(1,0){80}} \put (200,5){$e_{2}$} \put (260,0){$ A_2^+$}


\put (145,-95){$j(A_0)$} \put(170,-90){\vector(1,0){80}} \put
(190,-100){$e_{2}\upharpoonright j(A_0)$}
\put(150,-75){\vector(0,1){60}} \put(135,-60){$Id$}

 \put (260,-100){$ Sg^{A_2^+}
(e_{2}j(A_0))$}
\put(265,-75){\vector(0,1){60}} \put(245,-60){$Id$}

\put (80,-150){$A_0$} \put(80,-160){\vector(0,-1){70}} \put
(65,-190){$e_0$} \put (80,-260){$ A_0^+$}

\put(100,-250){\vector(1,1){140}}\put (170,-190){$\bar{j}$}

\put (120,-130){$j$} \put(100,-140){\vector(1,1){40}}

\put(60,-250){\vector(-1,1){140}}\put (-15,-200){$\bar{i}$}

 \put(35,-130){$i$} \put(60,-140){\vector(-1,1){40}}

 \put(-10,0){\vector(-1,0){80}} \put (-50,5){$e_{1}$}
 \put (-120,0){$ A_1^+$}


 \put(-10,-90){\vector(-1,0){80}}\put (-5,-95){$i(A_0)$}
 \put(-60,-100){$e_{1}\upharpoonright i(A_0)$}

\put(0,-75){\vector(0,1){60}} \put(-20,-60){$Id$}


  \put (-160,-100){$Sg^{A_1^+} (e_{1}i(A_0))$}

\put(-110,-75){\vector(0,1){60}} \put(-135,-60){$Id$}

\put(-100,15){\vector(1,1){140}}\put (-50,80){$k$}

\put(270,15){\vector(-1,1){140}}\put (210,80){$h$}

\put (80,-280) {\makebox (0,0){{%
\large{  Figure 1}}}}

\end{picture}

\newpage
Let $$\bar{f}:\A^{(X_1)}\to \B$$ be defined by $$a\mapsto f(a/R)$$ and $$\bar{g}:\A^{(X_2)}\to \B$$ 
be defined by $$a\mapsto g(a/R).$$
Let $\B'$ be the algebra generated by $Imf\cup Im g$.
Then $\bar{f}\cup \bar{g}\upharpoonright X_1\cup X_2\to \B'$ is a function since $\bar{f}$ and $\bar{g}$ coincide on $X_1\cap X_2$.
By freeness of $\A$, there exists $h:\A\to \B'$ such that $h\upharpoonright_{X_1\cup X_2}=\bar{f}\cup \bar{g}$.
Let $T=kerh $. Then it is not hard to check that 
$$T\cap {}^2 A^{(X_1)}=R \text { and } T\cap {}^2A^{(X_2)}=S.$$ 
Note that $T$ is the congruence generated by $R\cup S$.
\end{demo}
The above claim is an algebraic version of Robinson's joint consistency property. The next one, is, on the other hand, 
an algebraic version of the Craig 
interpolation property.
\begin{athm}{Claim 2} 
Let $\mu$ be a cardinal $>0$. Let $\A=\Fr_{\mu}\RQEA_{\alpha}$. Then for any $X_1$, $X_2 \subseteq \mu$, if $x\in \A^{(X_1)}$ and $z\in \A^{(X_2)}$
and $x\leq z$ then there is a $y\in \A^{(X_1\cap X_2)},$ a finite $\Gamma\subseteq
\alpha$ such that
\begin{equation}\label{t1}
\begin{split}
x\leq y\leq {\sf c}_{(\Gamma)}z.
\end{split}
\end{equation}
\end{athm}

\begin{demo}{Proof} We call $y$ in the claim an interpolant of $x$ and $z$ and we say that $x\leq z$ can be interpolated inside $\A^{(X_1\cap X_2)}$.
Now let $x\in A^{(X_1)}$, $z\in A^{(X_2)}$ and assume that $x\leq z$.
Then $$x\in (\Ig^{\A}\{z\})\cap \A^{(X_1)}.$$
Let $$M=\Ig^{\A^{(X_1)}}\{z\}\text { and } N=\Ig^{\A^{(X_2)}}(M\cap A^{(X_1\cap X_2)}).$$
Then $$M\cap A^{(X_1\cap X_2)}=N\cap A^{(X_1\cap X_2)}.$$
By identifying ideals with congruences, and using the congruence extension property,
there is a an ideal $P$ of $\A$ 
such that $$P\cap A^{(X_1)}=N\text { and }P\cap A^{(X_2)}=M.$$
It follows that 
$$\Ig^{\A}(N\cup M)\cap A^{(X_1)}\subseteq P\cap A^{(X_1)}=N.$$
Hence
$$(\Ig^{(\A)}\{z\})\cap A^{(X_1)}\subseteq N.$$
and we have
$$x\in \Ig^{\A^{(X_1)}}[\Ig^{\A^{(X_2)}}\{z\}\cap A^{(X_1\cap X_2)}.]$$
This implies that there is an element $y$ such that
$$x\leq y\in A^{(X_1\cap X_2)}$$
and $y\in \Ig^{\A^{(X_2)}}\{z\}$.
But $\Ig^{\A^{(X_2)}}\{z\}=\{x\in A: x\leq {\sf c}_{(\Gamma)}z: 
\text{ for some finite } \Gamma\subseteq \alpha\}.$ 
Indeed, let $H$ denote the set of elements on the right hand side.
It is easy to check $H\subseteq \Ig^{\A^{(X_2)}}\{z\}$.
Conversely, assume that $y\in H,$ $\Gamma\subseteq_{\omega} \alpha.$ 
It is clear that ${\sf c}_{(\Gamma)}y\in H$.
Now let $x, y\in H$. Assume that $x\leq {\sf c}_{(\Gamma)}z$
and $y\leq {\sf c}_{(\Delta)}z,$
then $$x+y\leq {\sf c}_{(\Gamma\cup \Delta)}z.$$

Therefore there exists 
$\Gamma\subseteq_{\omega} \alpha$ 
such that  
$$x\leq y\leq {\sf c}_{(\Gamma)}z.$$
\end{demo}

\begin{athm}{Claim 3}
\begin{enumarab}
\item Let $\A=\Fr_4\RCA_{\alpha}$. Let $r, s$ and $t$ be defined as follows:
$$ r = {\sf c}_0(x\cdot {\sf c}_1y)\cdot {\sf c}_0(x\cdot -{\sf c}_1y),$$
$$ s = {\sf c}_0{\sf c}_1({\sf c}_1z\cdot {\sf s}^0_1{\sf c}_1z\cdot -{\sf d}_{01}) + {\sf c}_0(x\cdot -{\sf c}_1z),$$
$$ t = {\sf c}_0{\sf c}_1({\sf c}_1w\cdot {\sf s}^0_1{\sf c}_1w\cdot -{\sf d}_{01}) + {\sf c}_0(x\cdot -{\sf c}_1w),$$
where $ x, y, z, \text { and } w$ are the first four free generators
of $\A$.  Then $r\leq s\cdot t$
\item Let $\A=\Fr_5\RQA_{\alpha}$on $5$ generators. 
Let $r, s$ and $t$ be defined as follows:
$$ r = {\sf c}_0(x\cdot {\sf c}_1y)\cdot {\sf c}_0(x\cdot -{\sf c}_1y),$$
$$ s = {\sf c}_0{\sf c}_1({\sf c}_1z\cdot {\sf s}^0_1{\sf c}_1z\cdot -m) + {\sf c}_0(x\cdot -{\sf c}_1z),$$
$$ t = {\sf c}_0{\sf c}_1({\sf c}_1w\cdot {\sf s}^0_1{\sf c}_1w\cdot -m) + {\sf c}_0(x\cdot -{\sf c}_1w),$$
where $ x, y, z, w$ and $m$ are the five generators
of $\A$.
Then $r\leq s\cdot t$.  
\end{enumarab}
\end{athm}
\begin{demo}{Proof} 
Put $$a = x\cdot {\sf c}_1 y\cdot  -{\sf c}_0(x\cdot - {\sf c}_1z),$$
$$b = x\cdot  - {\sf c}_1 y\cdot -{\sf c}_0(x\cdot - {\sf c}_1z).$$
Then we have
\begin{equation*}
\begin{split}
 {\sf c}_1 a\cdot  {\sf c}_1 b & \leq {\sf c}_1 (x \cdot {\sf c}_1y)\cdot  {\sf c}_1(x\cdot -{\sf c}_1y)
  \,\,\ \textrm{by \cite{HMT1}1.2.7}\\
&  = {\sf c}_1 x\cdot {\sf c}_1 y\cdot  {\sf c}_1 x\cdot -{\sf c}_1y \,\,\,\, \textrm{by \cite{HMT1} 1.2.11}
\end{split}
\end{equation*}
and so
\begin{equation}\label{p1}
\begin{split}
{\sf c}_1 a\cdot  {\sf c}_1 b = 0.
\end{split}
\end{equation}
From the inclusion $x\cdot -{\sf c}_1 z \leq {\sf c}_0 (x\cdot -{\sf c}_1z)$ we get
$$x\cdot -{\sf c}_0 (x\cdot -{\sf c}_1z) \leq {\sf c}_1 z.$$
Thus $a, b \leq {\sf c}_1z$ and hence, by \cite{HMT1} 1.2.9,
\begin{equation}\label{p2}
\begin{split}
{\sf c}_1 a, {\sf c}_1 b \leq {\sf c}_1z.
\end{split}
\end{equation}
 We now compute:
\begin{equation*}
\begin{split}
{\sf c}_0 a\cdot  {\sf c}_0 b & \leq {\sf c}_0 {\sf c}_1 a \cdot 
{\sf c}_0 {\sf c}_1 b \,\,\,\,\,\  \textrm{by \cite{HMT1}
 1.2.7} \\
& = {\sf c}_0 {\sf c}_1 a \cdot  {\sf c}_1 {\sf s}^0_1 {\sf c}_1 b \,\,\,\,\,\  \textrm{by \cite{HMT1} 1.5.8
(i),\,\,\,\,\,\ \cite{HMT1} 1.5.9 (i)}\\
& = {\sf c}_1({\sf c}_0{\sf c}_1 a\cdot  {\sf s}^0_1 {\sf c}_1 b)\\
& = {\sf c}_0{\sf c}_1({\sf c}_1a\cdot {\sf s}^0_1{\sf c}_1b)\\
& = {\sf c}_0{\sf c}_1[{\sf c}_1a\cdot {\sf s}^0_1{\sf c}_1b\cdot (-{\sf d}_{01} + {\sf d}_{01})\\
& = {\sf c}_0{\sf c}_1[({\sf c}_1a\cdot {\sf s}^0_1{\sf c}_1b\cdot -{\sf d}_{01}) + ({\sf c}_1a\cdot {\sf s}^0_1{\sf c}_1b\cdot {\sf d}_{01})]\\
& = {\sf c}_0{\sf c}_1[({\sf c}_1a\cdot {\sf s}^0_1{\sf c}_1b\cdot -{\sf d}_{01}) + ({\sf c}_1a\cdot {\sf c}_1b\cdot {\sf d}_{01})] \,\,\,\,\
\textrm{by \cite{HMT1}
1.5.5}\\
& = {\sf c}_0{\sf c}_1 ({\sf c}_1a\cdot {\sf s}^0_1{\sf c}_1b\cdot -{\sf d}_{01}) \,\,\,\,\  \textrm{by (\ref{p1})}\\
& \leq {\sf c}_0{\sf c}_1 ({\sf c}_1z\cdot {\sf s}^0_1{\sf c}_1z\cdot -{\sf d}_{01}) \,\,\,\,\  \textrm{by (\ref{p2}), \cite{HMT1} 1.2.7}\\
\end{split}
\end{equation*}
We have proved that
$$ {\sf c}_0[ x\cdot  {\sf c}_1y\cdot -{\sf c}_0(x\cdot -{\sf c}_1z)]\cdot {\sf c}_0[x\cdot -{\sf c}_1y\cdot -{\sf c}_0(x\cdot -{\sf c}_1z)] \leq
{\sf c}_0{\sf c}_1({\sf c}_1z\cdot {\sf s}^0_1{\sf c}_1z\cdot -{\sf d}_{01}).$$
In view of \cite{HMT1} 1.2.11 this gives
$$ {\sf c}_0( x\cdot  {\sf c}_1y)\cdot {\sf c}_0(x\cdot -{\sf c}_1y)\cdot -{\sf c}_0(x\cdot -{\sf c}_1z) \leq
{\sf c}_0{\sf c}_1({\sf c}_1z\cdot {\sf s}^0_1{\sf c}_1z\cdot -{\sf d}_{01}).$$
The conclusion of the claim now follows.

By proving the next claim, we reach a contradiction with Claim 2, and thus our theorem will be proved.
\end{demo}
\begin{athm}{Claim 4}
Let everything be as in the previous claim. Then the inequality $r \leq s\cdot t$ cannot be interpolated by an element of $\A^{(\{x\})}$.
\end{athm}
\begin{demo}{Proof}This part is taken from \cite{amal}. We include it for the sake of completeness.
Let $$\B = ( \wp (^{\alpha}{\alpha}), \cup, \cap, \sim, \emptyset,
{^{\alpha}{\alpha}}, {\sf C}_\kappa, {\sf D}_{\kappa \lambda}, {\sf P}_{ij} )_{\kappa, \lambda
< \alpha},$$ 
that is $\B$ is the  full set algebra in the
space $ {^{\alpha}{\alpha}}$. 
Let $E$ be the set of all equivalence
relations on $\alpha$, and for each $ R \in E $ set
$$ X_R = \{ \varphi : \varphi  \in {^{\alpha}{\alpha}} \,\,\,\
\textrm{and, for all} \,\,\,\  \xi, \eta < \alpha , \varphi_\xi =
\varphi_\eta \,\,\,\ \textrm{iff} \,\,\,\  \xi R \eta \}.$$
Let
$$ C = \{ \bigcup_{R \in L} X_R : L \subseteq E \}. $$
$C$ is clearly closed under the formation of arbitrary unions, and
since
$$ \sim \bigcup_{R \in L} X_L =  \bigcup_{R \in E \sim L} X_R$$
for every $ L \subseteq E$, we see that $C$ is closed under the
formation of complements with respect to ${^{\alpha}{\alpha}}$. Thus
$ C $ is a Boolean subuniverse ( indeed, a complete Boolean
subuniverse) of $\B$; moreover, it is obvious that
\begin{equation}\label{p4}
\begin{split}
X_R \,\,\,\ \textrm{is an atom of} \,\,\,\ ( C, \cup, \cap, \sim,
0, {^{\alpha}{\alpha}}) \,\,\,\,\ \textrm{for each} \,\,\,\ R \in E.
\end{split}
\end{equation}
For all $ \kappa, \lambda < \alpha$ we have ${\sf  D}_{\kappa \lambda } =
\bigcup \{ X_R : ( \kappa, \lambda ) \in R \in E \} $ and hence $
{\sf D}_{\kappa \lambda } \in B$. Also,
$$ {\sf C}_\kappa X_R = \bigcup \{ X_S : S \in E, {^{2}{(\alpha \sim
\{\kappa\})}} \cap S = {^{2}{(\alpha \sim \{\kappa\})}} \cap R \}$$
for any $ \kappa < \alpha$ and $ R \in E$. Thus, because ${\sf C}_\kappa$
is completely additive ( cf.\cite{HMT1} 1.2.6(i)) and the remark preceding
it), we see that $C$ is closed under the operation $ {\sf C}_\kappa$ 
for every $ \kappa < \alpha$.  Also it is straightforward to see that $C$ is closed under substitutions.
For any $\tau=[i,j]\in {}^{\alpha}\alpha$,
$${\sf S}_{\tau}X_R=\bigcup\{X_S: S\in E, \forall i, j<\omega(iRj \longleftrightarrow \tau(i)S\tau(j)\}.$$
Therefore, we have shown
that
\begin{equation}\label{p5}
\begin{split}
C \,\,\,\ \textrm{is a subuniverse of} \,\,\,\ \B .
\end{split}
\end{equation}
To prove that $r\leq s\cdot t$ can't be interpolated by an element of $\A^{(\{x\})},$ 
it suffices to show that there is a subset $Y$ of $
{^{\alpha}{\alpha}}$ such that
\begin{equation}\label{p6}
\begin{split}
&X_{Id} \cap f(r) \neq 0 \,\,\ \textrm{for
every} \,\,\ f \in Hom (\A, \B) \\
& \textrm{such that} \,\,\ f(x) = X_{Id} \,\,\
\textrm{and} \,\,\ f(y) = Y,
\end{split}
\end{equation}
and also that for every finite $ \Gamma \subseteq \alpha$, there are subsets $Z, W$ of $ {^{\alpha}{\alpha}}$ such that
\begin{equation}\label{p7}
\begin{split}
& X_{Id} \sim {\sf C}_{(\Gamma)}g(s\cdot t) \neq 0 \,\,\
\textrm{for every} \,\,\ g\in Hom (\A, \B)\\
 & \textrm{such that}
\,\,\ g(x) = X_{Id}, g(z) = Z \,\,\ \textrm{ and }
\,\,\ g(w) = W.
\end{split}
\end{equation}
For suppose, on the contrary, that these conditions are not
sufficient. Then there exists a finite $ \Gamma \subseteq \alpha$, 
and an interpolant $u \in \A^{(\{x\})}$  and there also exist $ Y, Z, W \subseteq
{^{\alpha}{\alpha}}$ such that (\ref{p6}) and (\ref{p7}) hold. Take
any $h \in Hom ( \A, \B)$ such that $ h(x) = X_{Id}$, 
$ h(y) = Y$, $h(z) = Z$, and $h(w) = W.$
This is possible by the freeness of $\A.$ 
Then using the fact that $X_{Id} \cap h(r)$ is non-empty by (\ref{p6}) we get
$$ X_{Id} \cap h(u) = h(x\cdot u) \supseteq h(x\cdot r)
\neq 0.$$
And using the fact that $X_{Id} \sim
{\sf C}_{(\Gamma)}h(s\cdot t)$ is non-empty by (\ref{p7}) we get
$$ X_{Id} \sim h(u) = h( x\cdot -u) \supseteq
h(x\cdot -{\sf c}_{(\Gamma)}(s\cdot t)) \neq 0.$$
However, in view of (\ref{p4}), it is impossible for $X_{Id}$ 
 to intersect both $h(u)$ and its complement since
$h(u) \in C$ and $X_{Id}$ is an atom; 
to see that $h(u)$ is indeed contained in $C$ recall that
$ u \in \A^{(\{x\})}$, and then observe that because of
(\ref{p5}) and the fact that $X_{Id} \in C$ we
must have
\begin{equation}\label{p8}
\begin{split}
h [\A^{(\{x\})}] \subseteq C
\end{split}
\end{equation}
Therefore, (\ref{p6}) and (\ref{p7}) are sufficient conditions for
$r\leq s\cdot t$ not to be interpolated by an element of $\A^{(\{x\})}$.
The next part of the prof is taken verbatim from \cite{P} p. 340-341.
Let $\sigma\in {}^{\alpha}{\alpha}$ be such that 
$\sigma_0 = 0$, and $\sigma_\kappa = \kappa + 1$ for every non-zero
$\kappa < \alpha$ and otherwise $\sigma(k)=k$.
Let $ \tau = \sigma\upharpoonright (\alpha \sim \{0\}) \cup \{(0, 1)\}$. 
Then $
\sigma, \tau \in X_{Id}$. Take
$$ Y = \{\sigma\}.$$
Then

$$ \sigma \in X_{Id} \cap {\sf C}_1 Y \,\,\ \textrm{and}
\,\,\ \tau \in X_{Id} \sim {\sf C}_1 Y$$ and hence
\begin{equation}\label{p9}
\begin{split}
\sigma \in {\sf C}_0 (X_{Id} \cap {\sf C}_1 Y) \cap {\sf C}_0
(X_{Id} \sim {\sf C}_1 Y).
\end{split}
\end{equation}
Therefore, we have $ \sigma \in f(r)$ for every $ f \in Hom(\A, \B)$
such that $ f(x) = X_{Id}$ and $f(y) =Y$, and that
(\ref{p6}) holds.
We now want to show that for any given finite $ \Gamma \subseteq
\alpha$, there exist sets $ Z, W \subseteq {^{\alpha}{\alpha}}$ such
that (\ref{p7}) holds; it is clear that no generality is lost if we
assume that $0, 1 \in \Gamma$, so we make this assumption. Take
$$ Z = \{ \varphi : \varphi \in X_{Id},
\varphi_0 < \varphi_1 \} \cap {\sf C}_{(\Gamma)} \{Id\}$$ and
$$ W = \{ \varphi : \varphi \in X_{Id},
\varphi_0 > \varphi_1 \} \cap {\sf C}_{(\Gamma)} \{Id\}.$$
We show that
\begin{equation}\label{p10}
\begin{split}
Id \in X_{Id} \sim {\sf C}_{(\Gamma)} g(s\cdot t)
\end{split}
\end{equation}
for any $ g \in Hom(\A, \B)$ such that $ g(x) = X_{Id}$, $g(z) = Z$, 
and $g(w) = W$; to do this we simply
compute the value of ${\sf C}_{(\Gamma)} g(s\cdot t)$. For the purpose of this
computation we make use of the following property of ordinals: if
$\Delta$ is any non-empty set of ordinals, then $ \bigcap \Delta$ is the
smallest ordinal in $ \Delta$, and if, in addition, $\Delta $ is
finite, then $\bigcup \Delta$ is the largest element ordinal in 
$\Delta$. Also, in this computation we shall assume that $\varphi$
always represents an arbitrary sequence in ${^{\alpha}{\alpha}}$.
Then, setting
$$ \Delta \varphi = \Gamma \sim \varphi[\Gamma \sim \{0, 1 \}]$$
for every $\varphi$, we successively compute:
$$ {\sf C}_1Z = \{ \varphi : |\Delta \varphi | = 2, \varphi_0 = \bigcap
\Delta \varphi \} \cap {\sf C}_{(\Gamma)} \{Id\},$$
\begin{equation*}
\begin{split}
& (X_{Id} \sim {\sf C}_1Z)
\cap {\sf C}_{(\Gamma)} \{ Id\} = \\
 & \{ \varphi : |\Delta \varphi | = 2, \varphi_0 = \bigcup \Delta
\varphi, \varphi_1 = \bigcap \Delta \varphi \}\cap {\sf C}_{(\Gamma)} \{Id\},
\end{split}
\end{equation*}
and, finally,
\begin{equation}\label{p11}
\begin{split}
{\sf C}_0(X_{Id}\sim {\sf C}_1Z)
\cap {\sf C}_{(\Gamma)} \{Id\} = \\
 & \{ \varphi : |\Delta \varphi | = 2, \varphi_1 =
 \bigcap \Delta \varphi \}\cap {\sf C}_{(\Gamma)} \{
Id\}.
\end{split}
\end{equation}
Similarly, we obtain
\begin{equation*}
\begin{split}
{\sf C}_0(X_{Id}\sim {\sf C}_1W)
\cap {\sf C}_{(\Gamma)} \{Id\} = \\
 & \{ \varphi : |\Delta \varphi | = 2, \varphi_1 =
 \bigcup \Delta \varphi \}\cap {\sf C}_{(\Gamma)} \{Id\}.
\end{split}
\end{equation*}
The last two formulas together give
\begin{equation}\label{p12}
\begin{split}
{\sf C}_0(X_{Id}\sim {\sf C}_1Z) \cap {\sf C}_0 (X_{Id} \sim {\sf C}_1W) \cap {\sf C}_{(\Gamma)} \{Id\} = 0.
\end{split}
\end{equation}
Continuing the computation we successively obtain:
$$ {\sf C}_1Z \cap {\sf D}_{01} = \{ \varphi : |\Delta \varphi | = 2,
\varphi_0 = \varphi_1 =
 \bigcap \Delta \varphi \}\cap {\sf C}_{(\Gamma)}\{Id\},$$
$$ {\sf S}^0_1{\sf C}_1Z = \{ \varphi : |\Delta \varphi | = 2,
\varphi_1 = \bigcap \Delta \varphi \}\cap {\sf C}_{(\Gamma)} \{ Id\},$$
$$ {\sf C_1}Z \cap {\sf S}^0_1{\sf C}_1Z = \{ \varphi : |\Delta \varphi | = 2,
\varphi_0 = \varphi_1 =
 \bigcap \Delta \varphi \}\cap {\sf C}_{(\Gamma)} \{Id\};$$
hence we finally get
\begin{equation}\label{p13}
\begin{split}
{\sf C}_0{\sf C}_1({\sf C}_1Z \cap {\sf S}^0_1{\sf C}_1Z \cap \sim {\sf D}_{01}) =  
{\sf C}_0{\sf C}_1 0 = 0,
\end{split}
\end{equation}
and similarly we get
\begin{equation}\label{p14}
\begin{split}
{\sf C}_0{\sf C}_1({\sf C}_1W \cap {\sf S}^0_1{\sf C}_1W \cap \sim {\sf D}_{01}) = 0
\end{split}
\end{equation}
Now take $g$ to be any homomorphism from $\A$ into $\B$ such that
$g(x) = X_{Id}$, $g(z) = Z $ and $ g(w) = W$. Let $a=g(s\cdot t)$. 
Then $a={\sf C}_0(X_{Id}\sim {\sf C}_1Z) \cap {\sf C}_0 (X_{Id} \sim {\sf C}_1W)$.
Then from (\ref{p12}), we have
$$ a\cap {\sf C}_{(\Gamma)} \{Id\} = \emptyset .$$
Then applying ${\sf C}_{(\Gamma)}$ to both sides of this equation we get
$${\sf C}_{(\Gamma)}a \cap {\sf C}_{(\Gamma)} \{Id\} = \emptyset . $$
Thus (\ref{p10}) holds and we are done.
\end{demo}

Using the techniques above, one can prove the following new theorem

\begin{theorem} Let $\alpha$ be an infinite ordinal. Then for any $n\in \omega$ the variety 
$V=\SNr_{\alpha}\QEA_{\alpha+n}$ does not have $AP$.
\end{theorem}
\begin{proof} In what follows by $\A^{(X)}$ we denote the subalgebra of $\A$ generated by $X$, and we write
$\A^{(x)}$ for $\A^{(\{x\})}$. Seeking a contradiction, assume that $V$ has $AP$ with respect
Let $\A=\Fr_4V$, the free $V$ algebra on $4$ generators. 
Let $r, s$ and $t$ be defined as follows:
$$ r = {\sf c}_0(x\cdot {\sf c}_1y)\cdot {\sf c}_0(x\cdot -{\sf c}_1y),$$
$$ s = {\sf c}_0{\sf c}_1({\sf c}_1z\cdot {\sf s}^0_1{\sf c}_1z\cdot -d_{01}m) + {\sf c}_0(x\cdot -{\sf c}_1z),$$
$$ t = {\sf c}_0{\sf c}_1({\sf c}_1w\cdot {\sf s}^0_1{\sf c}_1w\cdot -d_{01}) + {\sf c}_0(x\cdot -{\sf c}_1w),$$
where $ x, y, z, w$ and $m$ are the five generators
of $\A$.Then $r\leq s\cdot t$. 
Let $X_1=\{x, y\}$
and $X_2=\{x,z, w\}.$ Then  
\begin{equation}\label{p15}
\begin{split}
\A^{(X_1\cap X_2)}=\Sg^{\A}\{x\}.
\end{split}
\end{equation}
We have 
\begin{equation}
\begin{split}
r\in A^{(X_1)} \text { and }s,t\in A^{(X_2)}.
\end{split}
\end{equation}
Let $\{x', y', z', w'\}$ be the first four generators of $\D=\Fr_{4}\RQEA_{\alpha}$. Let $h$ be the homomorphism from $\A$ to $\D$ such that
$h(i)=i'$ for $i\in \{x,y,w,z\}$.
Let $J$ be the kernel of $h$. Then
\begin{equation}\label{p17}
\begin{split}
\A/J\cong \D
\end{split}
\end{equation}
We work inside the algebra $\A$. Since $r\leq s\cdot t$ we have  
\begin{equation}\label{p18}
\begin{split}
r\in \Ig^{\A}\{s\cdot t\}\cap A^{(X_1)}.
\end{split}
\end{equation}
Let
\begin{equation}\label{p19}
\begin{split}
M=\Ig^{\A^{(X_2)}}[\{s\cdot t\}\cup (J\cap A^{(X_2)})];
\end{split}
\end{equation}
\begin{equation}\label{p20}
\begin{split}
N=\Ig^{\A^{(X_1)}}[(M\cap A^{(X_1\cap X_2)})\cup (J\cap A^{(X_1)})].
\end{split}
\end{equation}
Then we have
\begin{equation}\label{p21}
\begin{split}
J\cap A^{(X_2)}\subseteq M\text { and }J\cap A^{(X_1)}\subseteq N.
\end{split}
\end{equation}
From the first of these inclusions we get
$$M\cap A^{(X_1\cap X_2)}\supseteq (J\cap A^{(X_2)})\cap A^{(X_1\cap X_2)}=(J\cap A^{(X_1)})
\cap A^{(X_1\cap X_2)}.$$
Then 
$$N\cap A^{(X_1\cap X_2)}=M\cap A^{(X_1\cap X_2)}.$$
For $R$ an ideal of $\A$ and $X\subseteq A$, by $(\A/R)^{(X)}$ we understand the subalgebra of 
$\A/R$ generated by $\{x/R: x\in X\}.$  
Replacing $S$ by $N$ and $R$ by $M$ in the first part of the proof and using that $V$ has $AP$,
let $P$ be the ideal corresponding to $kerh$ as defined above. Then, as before: 
\begin{equation}\label{p23}
\begin{split}
P\cap A^{(X_1)}=N,
\end{split}
\end{equation}
and 
\begin{equation}\label{p24}
\begin{split}
P\cap A^{(X_2)}=M.
\end{split}
\end{equation}
Now $s\cdot t\in P$ and so $r\in P$. 
Consequently we get $r\in N$, and so  
there exist elements
\begin{equation}\label{p25}
\begin{split}
u\in M\cap A^{(X_1\cap X_2)}
\end{split}
\end{equation}
and $b\in J$ such that
\begin{equation}\label{p26}
\begin{split}
r\leq u+b.
\end{split}
\end{equation}
Since $u\in M$ by $(7)$ there is a $\Gamma\subseteq_{\omega} \alpha$ and 
$c\in J$ such that
$$u\leq {\sf c}_{(\Gamma)}(s\cdot t)+c.$$
Recall that $h$ is 
the homomorphism from $\A$ to $\D$ such that
$h(i)=i'$ for $i\in \{x,y,w,z\}$.
and that $ker h=J$. Then $h(b)=h(c)=0$.
It follows that $$h(r)\leq h(u) \leq {\sf c}_{(\Gamma)}(h(s)\cdot h(t)).$$
And this is impossible.
\end{proof}

\section{Adjointness of the neat reduct functor}

Now it is  high time to find the spirit! 

For this purpose we put category theory in use. 
The main advantage of category theory is that it allows one not to miss the forest for the trees.

In our investigations, we have a certain dichotomy; we have two trees, the polyadic one and the cylindric one; or perhaps even two forests (indeed each paradigm is huge enough).
One way of formulating our investigations in this paper, in a nut shell, is {\it where is}, or rather {\it what is} 
 the forest, or, perhaps,  the forest encompassing the two forests?

Between the lines, one can see that we  are basically proving that the superamalgamation property for a subclass of representable algebras 
is equivalent to invertibility of the dilation functor,  while only the existence of a right adjoint is equivalent to 
amalgamation. This is the general picture. But the details are intricate. 

Such results will be also proved in a 
much more general setting in the final section, when we apply category theory to one way  (most probably not the only one) 
of locating the {\it universal} forest (the common spirit of the two spirits), a generalized systems of varieties 
covering also the polyadic paradigm.

In category theory what really counts are the {\it formulation of the definitions}. The proofs come later in priority of importance.
The most important versatile concept in category theory is that of adjoint situations which abound in all branches of pure mathematics.

A special case of adjoint situations is equivalence of two categories; this is most interesting and intriguing 
when this equivalence can be implemented by the contravariant Hom functor using a co-separator in the target category. 
Examples include Boolean algebras and Stone spaces, cylindric algebras and Sheaves,
locally compact abelian groups and abelian groups and C Star algebras and compact Hausdorf space.

It is definitely most inspiring and exciting to discover that two seeminly unrelated areas are nothing more than two sides of the same coin.
Here, our adjoint situation proved in the polyadic case,  shows that the category of algebras in $\omega$ dimensions 
is actually equivalent to that  in $\omega+\omega$ dimensions. 
This also holds for the countable reducts studied by Sain as a solution to the so called finitizability problem.

The aim of this problems that casts its shadow over the entire field, 
is to capture infinitely many dimensions in a finitary way, which seems paradoxal at first sight. 
But this can be done, with some ingenuity in usual set theory for infinite dimensions, and by changing the ontology 
to  non-well founded set theories for finite dimensions. (Here the extra dimensions are generated 'downwards').

Categorially this is expressed by the fact that the hitherto established equivalence says that the {\it infinite} gap can be finitized, 
but alas, even more, it actually says that
it does not exist at all. The apparent gap happens to be there, 
because it is either a historical accident or an unintended  repercussion of 
the original 
formulation of such algebras.

In our categorial notation we follow \cite{cat}

\begin{definition} Let $L$ and $K$ be two categories.
Let $G:K\to L$ be a functor and let $\B\in Ob(L)$. A pair $(u_B, \A_B)$ wth $\A_B\in Ob(K)$ and $u_B:\B\to G(\A_B)$ is called a universal map
with respect to $G$
(or a $G$ universal map) provided that for each $\A'\in Ob(K)$ and each $f:\B\to G(\A')$ there exists a unique $K$ morphism
$\bar{f}: \A_B\to \A'$ such that
$$G(\bar{f})\circ u_B=f.$$
\end{definition}

\begin{displaymath}
    \xymatrix{
        \mathfrak{B} \ar[r]^{u_B} \ar[dr]_f & G(\mathfrak{A}_\mathfrak{B}) \ar[d]^{G(f)}  &\mathfrak{A}_\mathfrak{B} \ar[d]^{\hat{f}} \\
             & G(\mathfrak{A}')  & \mathfrak{A}'}
\end{displaymath}

The above definition is strongly related to the existence of adjoints of functors.
For undefined notions in the coming definition, the reader is referred to \cite{cat}
Theorem 27.3 p. 196.
\begin{theorem} Let $G:K\to L$.
\begin{enumarab}
\item If each $\B\in Ob(K)$ has a $G$ universal map $(\mu_B, \A_B)$, then there exists a unique adjoint situation $(\mu, \epsilon):F\to G$
such that $\mu=(\mu_B)$ and for each $\B\in Ob(L),$
$F(\B)=\A_B$.
\item Conversely, if we have an adjoint situation $(\mu,\epsilon):F\to G$ then for each $\B\in Ob(K)$ $(\mu_B, F(\B))$ have a $G$ universal map.
\end{enumarab}
\end{theorem}
Now we apply this definition to the `neat reduct functor' from a
certain subcategory of $\CA_{\alpha+\omega}$
to $\RCA_{\alpha}$. More precisely, let 
$$\L=\{\A\in \CA_{\alpha+\omega}: \A=\Sg^{\A}\Nr_{\alpha}\A\}.$$
Note that $\L\subseteq \RCA_{\alpha+\omega}$. The reason is that any $\A\in \L$ is generated by $\alpha$ -dimensional elements,
so is dimension complemented (that is $\Delta x\neq \alpha$ for all $x$), and such algebras are representable.
Consider $\Nr_{\alpha}$ as a functor from $\bold L$ to $\CA_{\alpha}$, but we restrict morphisms to one to one homomorphisms; that is we take only
embeddings.
By the neat embedding theorem $\Nr_{\alpha}$ is a functor from $\L$ to $\RCA_{\alpha}$.
(For when $\A\in \CA_{\alpha+\omega},$ then $\Nr_{\alpha}\A\in \RCA_{\alpha}$).
The question we adress is: Can this functor be ``inverted".
This functor is not dense since there are representable algebras not in $\Nr_{\alpha}\CA_{\alpha+\omega}$, 
as the following example, which is a straightforward  adaptation  
of a result in \cite{SL} shows:
\begin{example}\label{ex} 
\begin{enumarab}
\item  Let $\F$ be a field of characteristic $0$. Let 
$$V=\{s\in {}^{\alpha}\F: |\{i\in \alpha: s_i\neq 0\}|<\omega\},$$
Let
$${\C}=(\wp(V),
\cup,\cap,\sim, \emptyset , V, {\sf c}_i,{\sf d}_{ij})_{i,j\in \alpha},$$
with cylindrifiers and diagonal elements restricted to $V$.  
Let $y$ denote the following $\alpha$-ary relation:
$$y=\{s\in V: s_0+1=\sum_{i>0} s_i\}.$$
Note that the sum on the right hand side is a finite one, since only 
finitely many of the $s_i$'s involved 
are non-zero. 
For each $s\in y$, we let 
$y_s$ be the singleton containing $s$, i.e. $y_s=\{s\}.$ 
Define 
${\A}\in \CA_{\alpha}$ 
as follows:
$${\A}=\Sg^{\C}\{y,y_s:s\in y\}.$$
Then it is proved in \cite{SL} that  
$$\A\notin \Nr_{\alpha}\CA_{\alpha+1}.$$ 
That is for no $\mathfrak{P}\in \CA_{\alpha+1}$, it is the case that $\Sg^{\C}\{y,y_s:s\in y\}$ 
exhausts the set of all $\alpha$ dimensional elements 
of $\mathfrak{P}$.  
\item Let $\A$ be as in above. Then since $\A$ is a weak set algebra, it is representable.
Hence $\A\in S\Nr_{\alpha}\CA_{\alpha+\omega}$. Let $\B\in \CA_{\alpha+\omega}$ be an algebra such that $\A\subseteq \Nr_{\alpha}\B$. 
Let $\B'$ be the subalgebra of $\B$ 
generated by $\A$. Then $\A$ generates $\B$ but $\A$ is not isomorphic to $\Nr_{\alpha}\B$.
\end{enumarab}
\end{example}

Item (2) in the above example says that there are two non isomorphic algebras, namely $\A$ and $\Nr_{\alpha}\B'$ 
that generate the same algebra $\B'$ using extra dimensions \cite{neat}.
If $\A\subseteq \Nr_{\alpha}\B$ then $\B$ is called a dilation of $\A$. $\B$ is a minimal dilation if $\A$
generates $\B$, in which case $\A$ is called a generating subreduct of $\B$. In the previous example $\A$ is a generating subreduct of $\B$.
One would expect that the ``inverse" of the Functor $\Nr$
would be the functor that takes $\A$ to a minimal dilation, and lifting morphisms.
But this functor is not even a right adjoint.

\begin{corollary} Let $\bold L=\{\A\in \RCA_{\alpha+\omega}: \A=\Sg^{\A}\Nr_{\alpha}\A\}$. Then the neat reduct functor $\Nr_{\alpha}$ 
from $\bold L$ to $\RCA_{\alpha}$ with morphisms restricted to injective homomorphisms
does not have a right adjoint.
\end{corollary}
\begin{proof} $UNEP$ is equivalent to existence of universal maps; the former does not hold for the representable algebras.
\end{proof}
\begin{corollary} If $\A$ has a universal map with respect to the above functor, then $\A$ belongs to the amalgamation base of $\RK_{\alpha}$
\end{corollary}
For $\A\in \PA_{\alpha}$ a polyadic algebra and $\beta>\alpha$, a $\beta$ dilation of $\A$ is an algebra $\B\in \PA_{\beta}$
such that $\A\subseteq \Nr_{\alpha}\B.$ $\B$ is a minimal dilation of $\A$ if $A$ generates $\B.$
Let $\bold L=\{\A\in \PA_{\beta}: \Sg\Nr_{\alpha}A=A\}$. Then $\Nr_{\alpha}:\bold L\to \PA_{\alpha}$ is an equivalence.
To prove this we first note that polyadic algebras do not satisfy $(i)$ of \ref{amal}. 
But before that we need a lemma. For $X\subseteq A$, 
$\Ig^{A}X$ denotes the ideal generated by $A$.:
\begin{lemma}\label{polyadic} Let $\alpha<\beta$ be infinite ordinals. Let $\B\in \PA_{\beta}$ and $\A\subseteq \Nr_{\alpha}\B$.
\begin{enumarab}
\item if $A$ generates $\B$ then $\A=\Nr_{\alpha}\B$
\item If $A$ generates $\B$, and $I$ is an ideal of $\B$, then $\Ig^{\B}(I\cap A)=I$
\end{enumarab}
\end{lemma}\ref{net}
\begin{demo}{Proof}
\begin{enumarab}
\item Let $\A\subseteq \Nr_{\alpha}\B$ and $A$ generates $\B$ then $\B$ consists of all elements ${\sf s}_{\sigma}^{\B}x$ such that 
$x\in A$ and $\sigma$ is a transformation on $\beta$ such that
$\sigma\upharpoonright \alpha$ is one to one \cite{DM} theorem 3.3 and 4.3.
Now suppose $x\in \Nr_{\alpha}\Sg^{\B}X$ and $\Delta x\subseteq
\alpha$. There exists $y\in \Sg^{\A}X$ and a transformation $\sigma$
of $\beta$ such that $\sigma\upharpoonright \alpha$ is one to one
and $x={\sf s}_{\sigma}^{\B}.$  
Let $\tau$ be a 
transformation of $\beta$ such that $\tau\upharpoonright  \alpha=Id
\text { and } (\tau\circ \sigma) \alpha\subseteq \alpha.$ Then
$x={\sf s}_{\tau}^{\B}x={\sf s}_{\tau}^{\B}{\sf s}_{\sigma}y=
{\sf s}_{\tau\circ \sigma}^{\B}y={\sf s}_{\tau\circ
\sigma\upharpoonright \alpha}^{\A'}y.$
Abusing notation we write $\A$ for $\Sg^{\A}X$ and $\B$ for $\Sg^{\B}X$. Then $\B$ is a minimal dilation of $\A$.
Each element of $\B$ has the form ${\sf s}_{\sigma}^{\B}a$ for some 
$a\in A$, and $\sigma$ a transformation on $\beta$ such that $\sigma\upharpoonright \alpha$ is one to one. 
We claim that $\Nr_{\alpha}\B\subseteq \A$. 
Indeed let $x\in Nr_{\alpha}B$. 
Then by the above we have 
$x={\sf s}_{\sigma}^{\B}y$, for some $y\in A$ 
and $\sigma\in {}^{\beta}\beta$. Let $\tau\in {}^{\beta}\beta$ such that
\begin{equation}\label{tarek1}
\begin{split}
\tau\upharpoonright \alpha\subseteq Id \text 
{ and }(\tau\circ \sigma)\alpha\subseteq \alpha. 
\end{split}
\end{equation}
Such a $\tau$ clearly exists.
Since $x\in \Nr_{\alpha}\B$, it follows by definition that 
${\sf c}_{(\beta\sim  \alpha)}x=x$.
From 
$$\tau\upharpoonright \beta\sim (\beta\sim  \alpha)=
\tau\upharpoonright \alpha=Id\upharpoonright \alpha=
Id\upharpoonright \beta\sim (\beta\sim \alpha),$$ we get
from the polyadic axioms that
$${\sf s}_{\tau}^{\B}x={\sf s}_{\tau}^{\B}{\sf c}_{(\beta\sim  \alpha)}x=
{\sf s}_{Id}^{\B}{\sf c}_{(\beta\sim \alpha)}x={\sf s}_{Id}^{\B}x=x.$$
Therefore
\begin{equation}\label{tarek2}
\begin{split}
x={\sf s}_{\tau}^{\B}x={\sf s}_{\tau}^{\B}{\sf s}_{\sigma}^{\B}x={\sf s}_{\tau\circ \sigma}^{\B}x.
\end{split}
\end{equation}
Let $$\mu= \tau\circ \sigma\upharpoonright \alpha\text { and } 
\bar{\mu}=\mu\cup Id\upharpoonright (\beta\sim \alpha).$$
Since 
$$\bar{\mu}\upharpoonright \beta\sim (\beta\sim \alpha)=
\bar{\mu}\upharpoonright \alpha=\mu=\tau\circ\sigma\upharpoonright \beta\sim 
(\beta\sim \alpha),$$
we have
$${\sf s}_{\bar{\mu}}^{\B}{\sf c}_{(\beta\sim \alpha)}y=
{\sf s}_{\tau\circ \sigma}^{\B}{\sf c}_{(\beta\sim \alpha)}y.$$
Since $\A\subseteq \Nr_{\alpha}\B$ and $y\in A,$ we have
${\sf s}_{\mu}^{\A}y={\sf s}_{\bar{\mu}}{}^{\B}y$ and ${\sf c}_{(\beta\sim\alpha)}^{\B}y=y.$
Therefore
\begin{equation}\label{tarek3}
\begin{split}
{\sf s}_{\mu}^{\A}y={\sf s}_{\bar{\mu}}^{\B}y={\sf s}_{\bar{\mu}}^{\B}{\sf c}_{(\beta\sim \alpha)}^{\B}y=
{\sf s}_{\tau\circ \sigma}^{\B}{\sf c}_{(\beta\sim \alpha)}^{\B}y={\sf s}_{\tau\circ \sigma}^{\B}y.
\end{split}
\end{equation}
From (\ref{tarek2}) and (\ref{tarek3}) we get $x={\sf s}_{\mu}^{\A}y\in \A$. By this the proof is complete
since $x$ was arbitrary.
\item Let $x\in \Ig^{\B}(I\cap A)$. Then ${\sf c}_{(\Delta x\sim \alpha)}x\in \Nr_{\alpha}\B=\A$, 
hence in $I\cap A$. But $x\leq {\sf c}_{(\Delta x\sim \alpha)}x$, and we are done.
\end{enumarab}
\end{demo}
The previous lemma fails for cylindric algebras in general \cite{neat}, but it does hold for $\Dc_{\alpha}$'s, see theorem 2.6.67, and 
2.6.71 in \cite{HMT1}.

\begin{theorem} Let $\alpha<\beta$ be infinite ordinals. Assume that $\A,\A'\in \PA_{\alpha}$ and $\B,\B'\in \PA_{\beta}.$
If $\A\subseteq \Nr_{\alpha}\B$ and $\A\subseteq \Nr_{\alpha}\B'$ and $A$ generates both then $\B$ and $\B'$ are isomorphic, 
then $\B$ and $\B'$ are isomorphic with an isomorphism that fixes $\A$ pointwise.
\end{theorem}

\begin{demo}{Proof} \cite{HMT2} theorem 2.6.72. 
We prove something stronger, we assume that $\A$ embeds into $\Nr_{\alpha}\B$ and similarly for $\A'$.
So let  $\A, \A' \in \PA_{\alpha}$ and $\beta>\alpha$. Let 
$\B, \B' \in \PA_{\beta}$ and assume that $e_A, e_{A'}$ are embeddings from $\A, \A'$ into   $\Nr_\alpha \B,
\Nr_\alpha \B'$, respectively, such that
$ \Sg^\B (e_A(A)) = \B$
and $ \Sg^{\B'} (e_{A'}(A')) = \B',$
and let $ i : \A \longrightarrow \A'$ be an isomorphism.
We need to ``lift" $i$ to $\beta$ dimensions.
Let $\mu=|A|$. Let $x$ be a  bijection  from $\mu$ onto $A.$ 
Let $y$ be a bijection from $\mu$ onto $A'$,
such that $ i(x_j) = y_j$ for all $j < \mu$.
Let $\D = \Fr_{\mu} \PA_{\beta}$ with generators $(\xi_i: i<\mu)$. Let $\C = \Sg^{\Rd_\alpha \D} \{ \xi_i : i < \mu \}.$
Then $\C \subseteq \Nr_\alpha \D,\  C \textrm{ generates } \D
~~\textrm{and so by the previous lemma }~~ \C =\Nr_{\alpha}\D.$
There exist $ f \in Hom (\D, \B)$ and $f' \in Hom (\D,
\B')$ such that
$f (g_\xi) = e_A(x_\xi)$ and $f' (g_\xi) = e_{A'}(y_\xi)$ for all $\xi < \mu.$
Note that $f$ and $f'$ are both onto. We now have
$e_A \circ i^{-1} \circ e_{A'}^{-1} \circ ( f'\upharpoonleft \C) = f \upharpoonleft \C.$
Therefore $ Ker f' \cap \C = Ker f \cap \C.$
Hence by $\Ig(Ker f' \cap \C) = \Ig(Ker f \cap \C).$
So, again by the the previous lemma, $Ker f'  = Ker f.$
Let $y \in B$, then there exists $x \in D$ such that $y = f(x)$. Define $ \hat{i} (y) = f' (x).$
The map is well defined and is as required.

\bigskip

\begin{displaymath}
    \xymatrix{
       D \ar[r]^f \ar[dr]_{f'} & B \ar[d]^{\hat{i}}  &\ar[l]_{e_{A}} A \ar[d]^{i} \\
             & B'  & \ar[l]^{e_{A'}}\mathcal{A}'}
\end{displaymath}

\end{demo}
\begin{corollary}\label{net} Let $\A, \A', i, e_A, e_{A'},$ $\B$ and $\B'$ be as in the previous proof. 
Then if $i$ is a monomorphism form $\A$ to $\A'$, then it lifts to a monomorphism $\bar{i}$ from $\B$ to $\B'$.
\end{corollary}
\begin{displaymath}
\xymatrix{
B \ar[d]^{\hat{i}}  &\ar[l]_{e_{A}} A \ar[d]^{i} \\
              B'  & \ar[l]^{e_{A'}}\mathcal{A}'}
\end{displaymath}

\begin{demo}{Proof} Consider $i:\A\to i(\A)$. Take $\C=\Sg^{\B'}(e_{A'}i(A))$. Then $i$ lifts to an isomorphism $\bar{i}\to \C\subseteq \B$.
\end{demo}
\begin{theorem} Let $\beta>\alpha$. Let $\bold L=\{\A\in \PA_{\beta}: \A=\Sg^{\A}\Nr_{\alpha}\A$\}. Let $\Nr:\bold L\to \PA_{\alpha}$ be the neat reduct functor. 
Then $\Nr$ is invertible. That is, there is a functor $G:\PA_{\alpha}\to \bold L$ and natural isomorphisms
$\mu:1_{\bold L}\to G\circ \Nr$ and $\epsilon: \Nr\circ G\to 1_{\PA_{\alpha}}$.
\end{theorem}
\begin{demo}{Proof} The idea is that a full, faithful, dense functor is invertible, \cite{cat} theorem 1.4.11. 
Let $L$ be a system of representatives for isomorphism on $Ob(\bold L)$.
For each $\B\in Ob(\PA_{\alpha})$ there is a unique $\G(B)$ in $L$ such that $\Nr(G(\B))\cong \B$.
$G(\B)$ is a minmal dilation of $\B$. Then $G:Ob(\PA_{\alpha})\to Ob(\bold L)$ is well defined. 
Choose one isomorphism $\epsilon_B: \Nr(G(B))\to \B$. If $g:\B\to \B'$ is a $\PA_{\alpha}$  morphism, then the square

\begin{displaymath}
    \xymatrix{ \Nr(G(B)) \ar[r]^{\epsilon_B}\ar[d]_{\epsilon_B^{-1}\circ g\circ \epsilon_{B'}} & B \ar[d]^g \\
               \Nr(G(B'))\ar[r]_{\epsilon_{B'}} & B' }
\end{displaymath}
commutes. By corollary \ref{net}, there is a unique morphism $f:G(\B)\to G(\B')$ such that $\Nr(f)=\epsilon_{\B}^{-1}\circ g\circ \epsilon$.
We let $G(g)=f$. Then it is easy to see that $G$ defines a functor. Also, by definition $\epsilon=(\epsilon_{\B})$ 
is a natural isomorphism from $\Nr\circ G$ to $1_{\PA_{\alpha}}$.
To find a natural isomorphism from $1_{\bold L}$ to $G\circ \Nr,$ observe that $e_{FA}:\Nr\circ G\circ \Nr(\A)\to \Nr(\A)$ is an isomorphism.
Then there is a unique $\mu_A:\A\to G\circ \Nr(\A)$ such that $\Nr(\mu_{\A})=e_{FA}^{-1}.$
Since $\epsilon^{-1}$ is natural for any $f:\A\to \A'$ the square
\bigskip
\bigskip
\begin{displaymath}
    \xymatrix{ \Nr(A) \ar[r]^{\epsilon_{\Nr(A)}^{-1}=\Nr(\mu_A)}\ar[d]_{\Nr(f)} & \Nr\circ G\circ \Nr(A) \ar[d]^{\Nr\circ G\circ \Nr(f)} \\
               \Nr(A')\ar[r]_{\epsilon_{FA}^{-1}=\Nr(\mu_{A'})} & \Nr\circ G\circ \Nr(A') }
\end{displaymath}

commutes, hence the square

\bigskip
\begin{displaymath}
    \xymatrix{ A \ar[r]^{\mu_A}\ar[d]_f & G\circ \Nr(A) \ar[d]^{G\circ \Nr(f)} \\
               A'\ar[r]_{\mu_{A'}} & G\circ \Nr(A') }
\end{displaymath}

commutes, too. Therefore $\mu=(\mu_A)$ is as required.
\end{demo}
To summarize we have theorems  2.6.67 (ii), 2.6.71-72 of \cite{HMT1} 
formulated for $\Dc$'s do not hold for $\K\in \{\SC, \CA, \QA, \QEA\}$; in fact they do not hold for
$\RK_{\alpha}$ 
but they hold for $\PA_{\alpha}$'s. Here $\alpha$ is an infinite ordinal.
This establishes yet another dichotomy between the $\CA$ paradigm and the $\PA$ paradigm.

Let $C$ be the reflective subcatogory of $\RCA_{\alpha}$ that has universal maps. Then $\Dc_{\alpha}\subseteq \bold L$.
And indeed we have:
\begin{theorem}\label{SUP} Let $\alpha\geq \omega$ . Let $\A_0\in \Dc_{\alpha}$, $\A_1,\A_2\in \RCA_{\alpha}$
and $f:\A_0\to \A_1$ and $g:\A_0\to \A_2$ be monomorphisms. Then there exists $\D\in \Nr_{\alpha}\CA_{\alpha+\omega}$
and $m:\A_1\to D$ and $n:\A_2\to \D$ such that $m\circ f=n\circ g$. Furthermore $\D$ is a super amalgam.
\end{theorem}
\begin{demo}{Proof}
Looking at figure 1, assuming that the base algebra $\A_0$ is in $\Dc_{\alpha}$, we obtain 
$\D\in \Nr_{\alpha}\CA_{\alpha+\omega}$
$m:\A_1\to \D$, and $n:\A_2\to \D$
such that $m\circ i=n\circ j$.
Here $m=k\circ e_1$ and $n=h\circ e_2$.
Denote $k$ by $m^+$ and $h$ by $n^+$.
Now we further want to show that if $m(a) \leq n(b)$,
for $a\in A_1$ and $b\in A_2$, then there exists $t \in A_0$
such that $ a \leq i(t)$ and $j(t) \leq b$.
So let $a$ and $b$ be as indicated . We have  $m^+ \circ e_1(a) \leq n^+ \circ e_2(b),$ so
$m^+ ( e_1(a)) \leq n^+ ( e_2(b)).$
Since $\bold L$ has $SUPAP$, there exist $ z \in A_0^+$ such that $e_1(a) \leq \bar{i}(z)$ and
$\bar{j}(z) \leq e_2b)$.
Let $\Gamma = \Delta z \smallsetminus \alpha$ and $z' =
c_{(\Gamma)}z$. (Note that $\Gamma$ is finite.) So, we obtain that
$e_1(c_{(\Gamma)}a) \leq \bar{i}(c_{(\Gamma)}z)~~ \textrm{and} ~~ \bar{j}(c_{(\Gamma)}z) \leq
e_2(c_{(\Gamma)}b).$ It follows that $e_A(a) \leq \bar{i}(z')~~
\textrm{and} ~~ \bar{j}(z') \leq e_B(b).$ Now $z' \in \Nr_\alpha \A_0^+
= \Sg^{\Nr_\alpha \A_0^+} (e_{A_0}(A_0)) = A_0.$ Here we use \cite{HMT1} 2.6.67.
So, there exists $t \in C$ with $ z' = e_C(t)$. Then we get
$e_1(a) \leq \bar{i}(e_0(t))$ and $\bar{j}(e_1(t)) \leq e_2(b).$ It follows that $e_1(a) \leq e_A \circ i(t)$ and
$e_2 \circ j(t) \leq
e_2(b).$ Hence, $ a \leq i(t)$ and $j(t) \leq b.$
We are done.
\end{demo}

Now we have senn that polyadic algebras have a nice representation theorem. But is the standard modelling of 
polyadic algebras appropriate for certain phenomena of reasoning? 
The recieved conclusions about complexity of the axiomatizations of such algebras, 
are not warranted, they are an artefact of these modellings, which are not mandatory in anyway.
The complexity of axiomatizations of polyadic algebras is highly complex fro the recursion theory point of view \cite{Sagi}.
This is one negative aspect of polyadic algebras, that is highly undesirable.

\section{Cylindric-polyadic algebras}

Now what? 

Let us reflect on our earlier investigations. We have formulated an adjoint situation that cylindric algebras 
do not satisfy while polyadic algebras do.

Cylindric algebras are nice in many respects. They are definable by a finite simple schema, 
and the there exists recursive axiomatizations of the representable algebras, which is the least that can be said about polyadic algebras. 
On the other hand polyadic algebra has a strong
Stone like representability result, which cylindric algebras lack, in a very resilient way.

Can we amalgamate the positive properties of both paradigms?
Can we tame non-finite axiomatizability results of cylindric algebras, and at the same time obtain positive results of polyadic algebras 
like interpolation? 

The question is certainly fair, and worhtwhile pondering about, even though it does not really have a mathematical exact formulation, and hence 
can lend itself to
different interpretations. Indeed, it is more of a philosophical question, but in history, it often happened  that 
what was a philosophical question at one point of time became a mathematical one, with a rigorous mathematical anwser at  a later time.
(For example Greeks talked abouy atoms).
Furthermore, vagueness could be an acet not a liability. 
On the other hand, insisting on rigour can occasionally be counterproductive. 

We can even go further than philosophy, and use the title of this article. Metaphysically, what is the "spirit" of cylindric algebras? 
Why, is there this feeling in the air, that quasi-polyadic algebras belong to the cylindric paradigm, while polyadic algebras do not.
Can we get this feeling down to earth, can we pinn it down.
This will be the aim of our later investigations.

One natural way to approach our general problem is to experiment with signatures and see what happens.
This is not a novel approach, it was already implemented by Sain studying countable reducts of polyadic algebras.
Only finite cylindrifiers are available, but the algebras intersect the poyadic paradigm because of the presence of infinitary substitutions
(that is substitutions moving infinitely many points) though a finite number of them only, and it can be proved that two is enough.

Now  we study a very natural amalgam of cylindric and polyadic algebras. We allow {\it all} substitutions, and restrict cylindrifies only to finite ones.
We do not alter the notion of representability; representable algebras are those that can be 
represented  on disjoint unions  of cartesian squares, so we keep the Tarskian (semantical) spirit.
We prove both a completeness and an interpolation result.  And we also show that the neat reduct functor is strongly invertible, which is utterly
unsurprising, because we have so many substitutions (these can be used to code extra dimensions).

We prove only the interpolation property (completeness is discernible below the surface of the proof) 
but in the presence of full fledged commutativity of cylindrifiers. 
The proof is basically a Henkin construction; 
algebraically a typical neat embedding theorem.


\begin{lemma}\label{net} Let $\A$ be a polyadic algebra of dimension $\alpha$. 
Then for every $\beta>\alpha$ there exists a polyadic algebra of dimension $\beta$ 
such that
$\A\subseteq \Nr_{\alpha}\B$, and furthermore, for all $X\subseteq \A$ we have
$$\Sg^{\A}X=\Sg^{\Nr_{\alpha}\A}X=\Nr_{\alpha}\Sg^{\B}X.$$ In particular, $\A=\Nr_{\alpha}\B$.
$\Sg^{\B}\A$ is called the minimal dilation of $\A.$ 
\end{lemma}
\begin{demo}{Proof} The proof depends essentially on the abundance of substitutions; we have all of them, which makes stretching dimensions 
possible. We provide a proof for cylindric polyadic algebras; the rest of the cases are like the corresponding prof in \cite{DM} for Boolean polyadic 
algebras. 

We extensively use the techniques in \cite{DM}, but we have to watch out, for we only have finite cylindrifications.
Let $(\A, \alpha,S)$ be a transformation system. 
That is to say, $\A$ is a Heyting algebra and $S:{}^\alpha\alpha\to End(\A)$ is a homomorphism. For any set $X$, let $F(^{\alpha}X,\A)$ 
be the set of all functions from $^{\alpha}X$ to $\A$ endowed with Heyting operations defined pointwise and for 
$\tau\in {}^\alpha\alpha$ and $f\in F(^{\alpha}X, \A)$, ${\sf s}_{\tau}f(x)=f(x\circ \tau)$. 
This turns $F(^{\alpha}X,\A)$ to a transformation system as well. 
The map $H:\A\to F(^{\alpha}\alpha, \A)$ defined by $H(p)(x)={\sf s}_xp$ is
easily checked to be an isomorphism. Assume that $\beta\supseteq \alpha$. Then $K:F(^{\alpha}\alpha, \A)\to F(^{\beta}\alpha, \A)$ 
defined by $K(f)x=f(x\upharpoonright \alpha)$ is an isomorphism. These facts are straighforward to establish, cf. theorem 3.1, 3.2 
in \cite{DM}. 
$F(^{\beta}\alpha, \A)$ is called a minimal dilation of $F(^{\alpha}\alpha, \A)$. Elements of the big algebra, or the cylindrifier free 
dilation, are of form ${\sf s}_{\sigma}p$,
$p\in F(^{\beta}\alpha, \A)$ where $\sigma$ is one to one on $\alpha$, cf. \cite{DM} theorem 4.3-4.4.

We say that $J\subseteq I$ supports an element $p\in A,$ if whenever $\sigma_1$ and  $\sigma_2$ are 
transformations that agree on $J,$ then  ${\sf s}_{\sigma_1}p={\sf s}_{\sigma_2}p$.
$\Nr_JA$, consisting of the elements that $J$ supports, is just the  neat $J$ reduct of $\A$; 
with the operations defined the obvious way as indicated above. 
If $\A$ is an $\B$ valued $I$ transformaton system with domain $X$, 
then the $J$ compression of $\A$ is isomorphic to a $\B$ valued $J$ transformation system
via $H: \Nr_J\A\to F(^JX, \A)$ by setting for $f\in\Nr_J\A$ and $x\in {}^JX$, $H(f)x=f(y)$ where $y\in X^I$ and $y\upharpoonright J=x$, 
cf. \cite{DM} theorem 3.10.

Now let $\alpha\subseteq \beta.$ If $|\alpha|=|\beta|$ then the the required algebra is defined as follows. 
Let $\mu$ be a bijection from $\beta$ onto $\alpha$. For $\tau\in {}^{\beta}\beta,$ let ${\sf s}_{\tau}={\sf s}_{\mu\tau\mu^{-1}}$ 
and for each $i\in \beta,$ let 
${\sf c}_i={\sf c}_{\mu(i)}$. Then this defined $\B\in GPHA_{\beta}$ in which $\A$ neatly embeds via ${\sf s}_{\mu\upharpoonright\alpha},$
cf. \cite{DM} p.168.  Now assume that $|\alpha|<|\beta|$.
Let $\A$ be a  given polyadic algebra of dimension $\alpha$; discard its cylindrifications and then take its minimal dilation $\B$, 
which exists by the above.
We need to define cylindrifications on the big algebra, so that they agree with their values in $\A$ and to have $\A\cong \Nr_{\alpha}\B$. We let (*):
$${\sf c}_k{\sf s}_{\sigma}^{\B}p={\sf s}_{\rho^{-1}}^{\B} {\sf c}_{\rho(\{k\}\cap \sigma \alpha)}{\sf s}_{(\rho\sigma\upharpoonright \alpha)}^{\A}p.$$
Here $\rho$ is a any permutation such that $\rho\circ \sigma(\alpha)\subseteq \sigma(\alpha.)$
Then we claim that the definition is sound, that is, it is independent of $\rho, \sigma, p$. 
Towards this end, let $q={\sf s}_{\sigma}^{\B}p={\sf s}_{\sigma_1}^{\B}p_1$ and 
$(\rho_1\circ \sigma_1)(\alpha)\subseteq \alpha.$

We need to show that (**)
$${\sf s}_{\rho^{-1}}^{\B}{\sf c}_{[\rho(\{k\}\cap \sigma(\alpha)]}^{\A}{\sf s}_{(\rho\circ \sigma\upharpoonright \alpha)}^{\A}p=
{\sf s}_{\rho_1{^{-1}}}^{\B}{\sf c}_{[\rho_1(\{k\}\cap \sigma(\alpha)]}^{\A}{\sf s}_{(\rho_1\circ \sigma\upharpoonright \alpha)}^{\A}p.$$
Let $\mu$ be a permutation of $\beta$ such that
$\mu(\sigma(\alpha)\cup \sigma_1(\alpha))\subseteq \alpha$.
Now applying ${\sf s}_{\mu}$ to the left hand side of (**), we get that 
$${\sf s}_{\mu}^{\B}{\sf s}_{\rho^{-1}}^{\B}{\sf c}_{[\rho(\{k\})\cap \sigma(\alpha)]}^{\A}{\sf s}_{(\rho\circ \sigma|\alpha)}^{\A}p
={\sf s}_{\mu\circ \rho^{-1}}^{\B}{\sf c}_{[\rho(\{k\})\cap \sigma(\alpha)]}^{\A}{\sf s}_{(\rho\circ \sigma|\alpha)}^{\A}p.$$
The latter is equal to
${\sf c}_{(\mu(\{k\})\cap \sigma(\alpha))}{\sf s}_{\sigma}^{\B}q.$
Now since $\mu(\sigma(\alpha)\cap \sigma_1(\alpha))\subseteq \alpha$, we have
${\sf s}_{\mu}^{\B}p={\sf s}_{(\mu\circ \sigma\upharpoonright \alpha)}^{\A}p={\sf s}_{(\mu\circ \sigma_1)\upharpoonright \alpha)}^{\A}p_1\in A$.
It thus follows that 
$${\sf s}_{\rho^{-1}}^{\B}{\sf c}_{[\rho(\{k\})\cap \sigma(\alpha)]}^{\A}{\sf s}_{(\rho\circ \sigma\upharpoonright \alpha)}^{\A}p=
{\sf c}_{[\mu(\{k\})\cap \mu\circ \sigma(\alpha)\cap \mu\circ \sigma_1(\alpha))}{\sf s}_{\sigma}^{\B}q.$$ 
By exactly the same method, it can be shown that 
$${\sf s}_{\rho_1{^{-1}}}^{\B}{\sf c}_{[\rho_1(\{k\})\cap \sigma(\alpha)]}^{\A}{\sf s}_{(\rho_1\circ \sigma\upharpoonright \alpha)}^{\A}p
={\sf c}_{[\mu(\{k\})\cap \mu\circ \sigma(\alpha)\cap \mu\circ \sigma_1(\alpha))}{\sf s}_{\sigma}^{\B}q.$$ 
By this we have proved (**).

Furthermore, it defines the required algebra $\B$. Let us check this.
Since our definition is slightly different than that in \cite{DM}, by restricting cylindrifications to be olny finite, 
we need to check the polyadic axioms which is tedious but routine. The idea is that every axiom can be pulled back to 
its corresponding axiom holding in the small algebra 
$\A$.
We check only the axiom $${\sf c}_k(q_1\land {\sf c}_kq_2)={\sf c}_kq_1\land {\sf c}_kq_2.$$
We follow closely \cite{DM} p. 166. 
Assume that $q_1={\sf s}_{\sigma}^{\B}p_1$ and $q_2={\sf s}_{\sigma}^{\B}p_2$. 
Let $\rho$ be a permutation of $I$ such that $\rho(\sigma_1I\cup \sigma_2I)\subseteq I$ and let 
$$p={\sf s}_{\rho}^{\B}[q_1\land {\sf c}_kq_2].$$
Then $$p={\sf s}_{\rho}^{\B}q_1\land {\sf s}_{\rho}^{\B}{\sf c}_kq_2
={\sf s}_{\rho}^{\B}{\sf s}_{\sigma_1}^{\B}p_1\land {\sf s}_{\rho}^{\B}{\sf c}_k {\sf s}_{\sigma_2}^{\B}p_2.$$
Now we calculate ${\sf c}_k{\sf s}_{\sigma_2}^{\B}p_2.$
We have by (*)
$${\sf c}_k{\sf s}_{\sigma_2}^{\B}p_2= {\sf s}^{\B}_{\sigma_2^{-1}}{\sf c}_{\rho(\{k\}\cap \sigma_2I)} {\sf s}^{\A}_{(\rho\sigma_2\upharpoonright I)}p_2.$$
Hence $$p={\sf s}_{\rho}^{\B}{\sf s}_{\sigma_1}^{\B}p_1\land {\sf s}_{\rho}^{\B}{\sf s}^{\B}_{\sigma^{-1}}{\sf c}_{\rho(\{k\}\cap \sigma_2I)} 
{\sf s}^{\A}_{(\rho\sigma_2\upharpoonright I)}p_2.$$
\begin{equation*}
\begin{split}
&={\sf s}^{\A}_{\rho\sigma_1\upharpoonright I}p_1\land {\sf s}_{\rho}^{\B}{\sf s}^{\A}_{\sigma^{-1}}{\sf c}_{\rho(\{k\}\cap \sigma_2I)} 
{\sf s}^{\A}_{(\rho\sigma_2\upharpoonright I)}p_2,\\
&={\sf s}^{\A}_{\rho\sigma_1\upharpoonright I}p_1\land {\sf s}_{\rho\sigma^{-1}}^{\A}
{\sf c}_{\rho(\{k\}\cap \sigma_2I)} {\sf s}^{\A}_{(\rho\sigma_2\upharpoonright I)}p_2,\\
&={\sf s}^{\A}_{\rho\sigma_1\upharpoonright I}p_1\land {\sf c}_{\rho(\{k\}\cap \sigma_2I)} {\sf s}^{\A}_{(\rho\sigma_2\upharpoonright I)}p_2.\\
\end{split}
\end{equation*} 
Now $${\sf c}_k{\sf s}_{\rho^{-1}}^{\B}p={\sf c}_k{\sf s}_{\rho^{-1}}^{\B}{\sf s}_{\rho}^{\B}(q_1\land {\sf c}_k q_2)={\sf c}_k(q_1\land {\sf c}_kq_2)$$
We next calculate ${\sf c}_k{\sf s}_{\rho^{-1}}p$.
Let $\mu$ be a permutation of $I$ such that $\mu\rho^{-1}I\subseteq I$. Let $j=\mu(\{k\}\cap \rho^{-1}I)$.
Then applying (*), we have:
\begin{equation*}
\begin{split}
&{\sf c}_k{\sf s}_{\rho^{-1}}p={\sf s}^{\B}_{\mu^{-1}}{\sf c}_{j}{\sf s}_{(\mu\rho^{-1}|I)}^{\A}p,\\
&={\sf s}^{\B}_{\mu^{-1}}{\sf c}_{j}{\sf s}_{(\mu\rho^{-1}|I)}^{\A}
{\sf s}^{\A}_{\rho\sigma_1\upharpoonright I}p_1\land {\sf c}_{(\rho\{k\}\cap \sigma_2I)} {\sf s}^{\B}_{(\rho\sigma_2\upharpoonright I)}p_2,\\
 &={\sf s}^{\B}_{\mu^{-1}}{\sf c}_{j}[{\sf s}_{\mu \sigma_1\upharpoonright I}p_1\land r].\\
\end{split}
\end{equation*}

where 
$$r={\sf s}_{\mu\rho^{-1}}^{\B}{\sf c}_j {\sf s}_{\rho \sigma_2\upharpoonright I}^{\A}p_2.$$
Now ${\sf c}_kr=r$. Hence, applying the axiom in the small algebra, we get: 
$${\sf s}^{\B}_{\mu^{-1}}{\sf c}_{j}[{\sf s}_{\mu \sigma_1\upharpoonright I}^{\A}p_1]\land {\sf c}_k q_2
={\sf s}^{\B}_{\mu^{-1}}{\sf c}_{j}[{\sf s}_{\mu \sigma_1\upharpoonright I}^{\A}p_1\land r].$$
But
$${\sf c}_{\mu(\{k\}\cap \rho^{-1}I)}{\sf s}_{(\mu\sigma_1|I)}^{\A}p_1=
{\sf c}_{\mu(\{k\}\cap \sigma_1I)}{\sf s}_{(\mu\sigma_1|I)}^{\A}p_1.$$
So 
$${\sf s}^{\B}_{\mu^{-1}}{\sf c}_{k}[{\sf s}_{\mu \sigma_1\upharpoonright I}^{\A}p_1]={\sf c}_kq_1,$$ and 
we are done. The second part, is exactly like theorem \ref{polyadic}.
\end{demo}

\begin{theorem} Let $\beta$ be a cardinal, and $\A=\Fr_{\beta}\FPA_{\alpha}$ be the free algebra on $\beta$ generators.
Let $X_1, X_2\subseteq
\beta$, $a\in \Sg^{\A}X_1$ and $c\in \Sg^{\A}X_2$ be such that $a\leq c$.
Then there exists $b\in \Sg^{\A}(X_1\cap X_2)$ such that $a\leq b\leq c.$
\end{theorem}

\begin{demo}{Proof}  Let $a\in \Sg^{\A} X_1$ and $c\in \Sg^{\A}X_2$ be such that $a\leq c$. 
We want to find an interpolant in 
$\Sg^{\A}(X_1\cap X_2)$. Assume that $\kappa$ is a regular cardinal $>max(|\alpha|,|A|)$.
Let $\B\in \FPA_{\kappa}$  such that $\A=\Nr_{\alpha}\B$, 
and $A$ generates $\B$.
 Let $H_{\kappa}=\{\rho\in {}^{\kappa}\kappa: |\rho(\alpha)\cap (\kappa\sim \alpha)|<\omega\}$.
Let $S$ be the semigroup generated by $H_{\kappa}.$ 
Let $\B'\in \FPA_{\kappa}$ be an ordinary  dilation of $\A$ where all transformations in $^{\kappa}\kappa$ are used.  
(This can be easily defined like in the case of ordinary polyadic algebras). 
Then $\A=\Nr_{\alpha}\B'$. We take a suitable reduct of $\B'$. Let $\B$ be the subalgebra of 
$\B'$ generated from $A$ be all operations except for substitutions indexed by transformations not in $S$.
Then, of course $A\subseteq \B$; in fact, $\A=\Nr_{\alpha}\B$, since for each $\tau\in {}^{\alpha}\alpha$, $\tau\cup Id\in S.$ 
It can be checked inductively that for $b\in B$, if 
$|\Delta b\sim \alpha|<\omega$, and $\rho\in S$, then $|\rho(\Delta b)\sim \alpha|<\omega$. 
Then there exists a finite  
$\Gamma\subseteq \kappa\sim \alpha$ such that $a\leq {\sf c}_{(\Gamma)}b\leq c$ and 
$${\sf c}_{(\Gamma)}b\in \Nr_{\alpha}\Sg^{\B}(X_1\cap X_2)=\Sg^{\Nr_{\alpha}\B}(X_1\cap X_2)=\Sg^{\A}(X_1\cap X_2).$$

The rest of the proof is similar to that in \cite{AU}, except that the latter  reference deals with countable algebras, 
and here our algebras could be 
uncountable, hence the condition of regularity on the cardinal $\kappa$.
Arrange $\kappa\times \Sg^{\C}(X_1)$ 
and $\kappa\times \Sg^{\C}(X_2)$ 
into $\kappa$-termed sequences:
$$\langle (k_i,x_i): i\in \kappa\rangle\text {  and  }\langle (l_i,y_i):i\in \kappa\rangle
\text {  respectively.}$$ Since $\kappa$ is regular, we can define by recursion 
$\kappa$-termed sequences: 
$$\langle u_i:i\in \kappa\rangle \text { and }\langle v_i:i\in \kappa\rangle$$ 
such that for all $i\in \kappa$ we have:
$$u_i\in \kappa\smallsetminus
(\Delta a\cup \Delta c)\cup \cup_{j\leq i}(\Delta x_j\cup \Delta y_j)\cup \{u_j:j<i\}\cup \{v_j:j<i\}$$
and
$$v_i\in \kappa\smallsetminus(\Delta a\cup \Delta c)\cup 
\cup_{j\leq i}(\Delta x_j\cup \Delta y_j)\cup \{u_j:j\leq i\}\cup \{v_j:j<i\}.$$
For a boolean algebra $\C$  and $Y\subseteq \C$, we write 
$fl^{\C}Y$ to denote the boolean filter generated by $Y$ in $\C.$  Now let 
$$Y_1= \{a\}\cup \{-{\sf  c}_{k_i}x_i+{\sf s}_{u_i}^{k_i}x_i: i\in \omega\},$$
$$Y_2=\{-c\}\cup \{-{\sf  c}_{l_i}y_i+{\sf s}_{v_i}^{l_i}y_i:i\in \omega\},$$
$$H_1= fl^{Bl\Sg^{\B}(X_1)}Y_1,\  H_2=fl^{Bl\Sg^{\B}(X_2)}Y_2,$$ and 
$$H=fl^{Bl\Sg^{\B}(X_1\cap X_2)}[(H_1\cap \Sg^{\B}(X_1\cap X_2)
\cup (H_2\cap \Sg^{\B}(X_1\cap X_2)].$$
Then $H$ is a proper filter of $\Sg^{\B}(X_1\cap X_2).$
This is proved by induction with the base of the induction bieng no interpolant exists in $\Sg^{\B}(X_1\cap X_2)$, cf. \cite{AU} Claim 2.18 
p.339.
Let $H^*$ be a (proper boolean) ultrafilter of $\Sg^{\B}(X_1\cap X_2)$ 
containing $H.$ 
We obtain  ultrafilters $F_1$ and $F_2$ of $\Sg^{\B}X_1$ and $\Sg^{\B}X_2$, 
respectively, such that 
$$H^*\subseteq F_1,\ \  H^*\subseteq F_2$$
and (**)
$$F_1\cap \Sg^{\B}(X_1\cap X_2)= H^*= F_2\cap \Sg^{\B}(X_1\cap X_2).$$
Now for all $x\in \Sg^{\cal B}(X_1\cap X_2)$ we have 
$$x\in F_1\text { if and only if } x\in F_2.$$ 
Also from how we defined our ultrafilters, $F_i$ for $i\in \{1,2\}$ satisfy the following
condition:
(*) For all $k<\mu$, for all $x\in \Sg^{\B}X_i$ 
if ${\sf  c}_kx\in F_i$ then ${\sf s}_l^kx$ is in $F_i$ for some $l\notin \Delta x.$  

Let $\D_i=\Sg^{\A}X_i$. For a transformation $\tau\in {}^{\alpha}\kappa$ let $\bar{\tau}=\tau\cup Id_{\kappa\sim \alpha}$.
Define $f_i$ from ${\D}_i$
to the full set algebra $\C$ with unit $^{\alpha}\kappa$ as follows:
$$f_i(x)=\{ \tau\in  {}^{\alpha}\kappa: {\sf s}_{\bar{\tau}}x\in F_i\}, \text { for } x\in {\cal D}_i$$ 
Then $f_i$ is a homomorphism by (*), \cite{AU} p.343.
Without loss of generality, we can assume that $X_1\cup X_2=X.$
By (**) we have $f_1$ and $f_2$ agree on $X_1\cap X_2$. 
So that $f_1\cup f_2$ defines a function on $X_1\cup X_2$, by freeness
it follows that there is a homomorphism  $f$ from $\B$ to $\C$ such that
$f_1\cup f_2\subseteq f$.   Then $q\in f(a)\cap f(-c) = f(a-c).$ This is so because $s_{Id}a=a\in F_1$
$s_{Id}(-c)=-c\in F_2.$
But this contradicts the premise that $a\leq c.$
\end{demo}

\subsection{Ferenzci's algebras}

Now we get rid of commutativity of cylindrifes and adopt weaker axioms adding also 
the so callled merry-go-round identies. In set algebras based on cartesian squares cylindrifiers commute, 
so we have no choice but to alter the notion of representability as well.
In modal logic this is termed as relativization. This approach pays, very much so. 

So called relativization started as a technique for generalizing representations of cylindric algebras, while also, in some cases, 
`defusing' undesirable properties, like undecidability or lack of definability (like Beth definability).
These ideas have counterparts in logic, and they have been influential in several ways. 
Relativization in cylindric-like algebras lends itself to a modal perspective where transitions are viewed as 
objects in their own right, in addition to states, while algebraic terms 
now correspond to modal formulas defining the essential properties 
of transitions.

Indeed, why insist on standard models? This is a voluntary commitment to only one mathematical implementation, 
whose undesirable complexities can pollute the laws of logics needed to describe the core phenomena.
Set theoretic cartesian squares modelling as the intended  vehicle may not be an orthogonal concern, 
it can be detremental, repeating hereditory sins of old paradigms. 

Indeed in \cite{HMT1} square units got all the attention and relativization  was treated as a side issue.
Extending original classes of models for logics to manipulate their properties is common. 
This is no mere tactical opportunism, general models just do the right thing.

The famous move from standard models to generalized models is 
Henkin's turning round  second  order logic into an axiomatizable two sorted first
order logic. Such moves are most attractive 
when they get an independent motivation. 

The idea is that we want to find a semantics that gives just the bare bones of action, 
while additional effects of square set theoretic modelling are separated out as negotiable decisions of formulation 
that threatens completeness, decidability,  and interpolation.

And indeed by using relativized representations Ferenzci, proved that if we weaken commutativity of cylindrifiers
and allow  relativized representations, then we get a finitely axiomatizable variety of representable 
quasi-polyadic equality algebras (analogous to the Resek Thompson $CA$ version); 
even more this can be done without the merry go round identities.
This is in sharp view with our complexity results proved above for quasi poyadic equlaity algebras.

Now we use two techniques to get positive results.
The first is yet again a Henkin construction (carefully implemented because we have changed the semantics, 
so that Henkin ultrafilters constructed are more involved), the other is inspired by the well-developed 
duality theory in modal logic between Kripke frames and complex algebras.

This technique was first implemented by N\'emeti in the context of relativized cylindric set algebras 
which are complex algebras of 
weak atom structures.

\begin{theorem} Let $\beta$ be a cardinal, and $\A=\Fr_{\beta}\CPA_{\alpha}$ be the free algebra on $\beta$ generators.
Let $X_1, X_2\subseteq
\beta$, $a\in \Sg^{\A}X_1$ and $c\in \Sg^{\A}X_2$ be such that $a\leq c$.
Then there exists $b\in \Sg^{\A}(X_1\cap X_2)$ such that $a\leq b\leq c.$
\end{theorem}

\begin{demo}{Proof}  Let $a\in \Sg^{\A} X_1$ and $c\in \Sg^{\A}X_2$ be such that $a\leq c$. 
We want to find an interpolant in 
$\Sg^{\A}(X_1\cap X_2)$. Assume that $\kappa$ is a regular cardinal $>max(|\alpha|,|A|)$.
Let $\B\in \CPA_{\kappa}$, be as in the previous lemma,  such that $\A=\Nr_{\alpha}\B'$, 
and $A$ generates $\B$. Like before we can asume that no interpolant exists in $\B$.
One defines filters in $\Sg^{\A}X_1$ and in $\Sg^{\A}X_2$
like in Ferenczi \cite{Fer}.
Let
$$Y_i=\{s_{\tau}c_jx: \tau\in adm, j\in \alpha, x\in A\}.$$
$$H_1= fl^{Bl\Sg^{\B}(X_1)}Y_1,\  H_2=fl^{Bl\Sg^{\B}(X_2)}Y_2,$$ and 
$$H=fl^{Bl\Sg^{\B}(X_1\cap X_2)}[(H_1\cap \Sg^{\B}(X_1\cap X_2)
\cup (H_2\cap \Sg^{\B}(X_1\cap X_2)].$$

Then $H$ is a proper filter of $\Sg^{\B}(X_1\cap X_2).$
This can be  proved by induction with the base of the induction bieng no interpolant exists in $\Sg^{\B}(X_1\cap X_2)$. 
Let $H^*$ be a (proper boolean) ultrafilter of $\Sg^{\B}(X_1\cap X_2)$ 
containing $H.$ 
We obtain  ultrafilters $F_1$ and $F_2$ of $\Sg^{\B}X_1$ and $\Sg^{\B}X_2$, 
respectively, such that 
$$H^*\subseteq F_1,\ \  H^*\subseteq F_2$$
and (**)
$$F_1\cap \Sg^{\B}(X_1\cap X_2)= H^*= F_2\cap \Sg^{\B}(X_1\cap X_2).$$
Now for all $x\in \Sg^{\cal B}(X_1\cap X_2)$ we have 
$$x\in F_1\text { if and only if } x\in F_2.$$ 
Also from how we defined our ultrafilters, $F_i$ for $i\in \{1,2\}$ are perfect.  

Then define the homomorphisms, one on each subalgebra, like in \cite{Sayed} p. 128-129, using the perfect ultrafilters, 
then freeness will enable use to end these homomophisms to the set of free generators, and it will satisfy 
$h(a.-c)\neq 0$ which is a contradiction.

\end{demo}

We show using techniques of Marx, that weak polyadic algebras have $SUPAP$ as well. This works 
for all varieties of relativized cylindric polyadic algebras studied by Ferenzci and reported in \cite{Fer}. 

This follows from the simple observation that such varieties can be axiomatized with positive, hence Sahlqvist equations, 
and therefore they are canonical; and also we do not have a Rosser condition on cylindrifiers; cylindrifiers do not commute, 
this allows that the first order correspondants of such equations are {\it clausifiable}, see \cite{Marx} for the 
definition of this. This proof is inspired by the modal perspective of cylindric-like algebras 
that suggests a whole landscape below standard predicate logic, 
with a minimal modal logic at the base ascending to standard semantics via frame constraints. In particular, 
this landscape contains nice sublogics of the full predicate logic, sharing its desirable meta properties and at the same time avoiding its 
negative accidents due to its Tarskian 'square frames' modelling. Such mutant logics are currently a very rich area of research.

The technique used here can be traced back to N\'emeti, when he proved that relativized cylindric set algebras have $SUPAP$; 
using (classical) duality
between atom structures and cylindric algebras. Marx 'modalized' the proof, and slightly strenghtened N\'emeti results, using instead 
the well-established duality between modal frames and 
complex algebras.  

We consider the non commutive cylindric polyadic algebras introduced by Ferenzci; but we use the notation $\WFPA_{\alpha}$.
A frame is a first order structure $\F=(V,  T_{i}, S_{\tau})_{i\in \alpha, \tau\in {}^{\alpha}\alpha}$ where $V$ is an arbitrary set and
and  both $T_{i}$ and $S_{\tau}$  are binary relations on $V$  for all $i\in \alpha$; and $\tau\in {}^{\alpha}\alpha$.

Given a frame $\F$, its complex algebra will be denotet by $\F^+$; $\F^+$ is the algebra $(\wp(\F), c_i, s_{\tau})_{i\in \alpha, \tau\in {}^{\alpha}\alpha}$  
where for $X\subseteq  V$,
$c_i(X)=\{s\in V: \exists t\in X, (t, s)\in T_i \}$, and similarly for $s_{\tau}$.

For $K\subseteq \WFPA_{\alpha}$, we let $\Str K=\{\F: \F^+\in K\}.$ 

For a variety $V$, it is always the case that 
$\Str V\subseteq \At V$ and equality holds if the variety is atom-canonical. 
If $V$ is canonical, then $\Str V$ generates $V$ in the strong sense, that is 
$V= {\bf S}\Cm \Str V$. For Sahlqvist varieties, as is our case, $\Str V$ is elementary.

\begin{definition}
\item Given a family $(\F_i)_{i\in I}$ of frames, a {\it zigzag  product} of these frames is a substructure of $\prod_{i\in I}\F_i$ such that the
projection maps restricted to $S$ are
onto.
\end{definition}

\begin{definition} Let $\F, \G, \H$ be frames, and $f:\G\to \F$ and $h:\F\to \H$. 
Then $INSEP=\{(x,y)\in \G\times \H: f(x)=h(y)\}$. 
\end{definition}

\begin{lemma} The frame $INSEP \upharpoonright G\times H$ is a zigzag product
of $G$ and $H$, such that $\pi\circ \pi_0=h\circ \pi_1$, where $\pi_0$ and $\pi_1$ are the projection maps.
\end{lemma}
\begin{proof} \cite{Marx} 5.2.4
\end{proof}
For an algebra $\A$, $\A_+$ denotes its ultrafilter atom structure. 
For $h:\A\to \B$, $h_+$ denotes the function from $\B_+\to \A_+$ defined by $h_+(u)=h^{-1}[u]$ 
where the latter is $\{x\in a: h(x)\in u\}.$

\begin{theorem}(\cite{Marx} lemma 5.2.6)
Assume that $K$ is a canonical variety and $\Str K$ is closed under finite zigzag products. Then $K$ has the superamalgamation
property.
\end{theorem}
\begin{demo}{Sketch of proof} Let $\A, \B, \C\in K$ and $f:\A\to \B$ and $h:\A\to \C$ be given monomorphisms. 
Then $f_+:\B_+\to \A_+$ and $h_+:\C_+\to \A_+$. We have $INSEP=\{(x,y): f_+(x)=h_+(y)\}$ is a zigzag connection. Let $\F$ be the zigzag product 
of $INSEP\upharpoonright \A_+\times \B_+$. 
Then $\F^+$ is a superamalgam.
\end{demo}

\begin{theorem}
The variety $\WFPA_{\alpha}$ has $SUPAP$.
\end{theorem}
\begin{proof} $\WFPA_{\alpha}$ can be easily  defined by positive equations then it is canonical.
The first order correspondents of the positive equations translated to the class of frames will be Horn formulas, hence clausifiable \cite{Marx} theorem 
5.3.5, and so $\Str K$ is closed under finite zigzag products. Marx's theorem finishes the proof.
\end{proof}

Let us make philosophy and metaphysics mathematics:

\begin{definition}\label{t} 
A class of algebras is in the cylindric paradigm if the the neat reduct functor $\Nr$ does not have a right adjoint; it is in the polyadic one, if the neat reduct functor is strongly 
invertible.
\end{definition}

More crudely, the polyadic paradigm versus the cylindric one, establishes a dichotomy in algebraic logic; the separating point is the presence of 
substitutions that move infinitely 
many points. 

Another dichotomy, existing in algebraic logic,  that can be traced back to the 
famous Andr\'eka-Resek-Thompson theorem, is that between square representations and relativized one. Here non-comutativity of cylindrifies is the 
separating point. Ferenczi's recent work on cylindric polyadic sheds a 
lot of light on this intriuging phenomena; and the connection of neat embeddings to relativized representations
is indeed quite a remarkable achievement.  

If one views relativized models as the natural semantics for predicate logic rather than some tinkering 
devise which is the approach adopted in \cite{HMT1}, then 
many well -established taboos of the field must be challenged. 

In standard textbooks one learns that predicate logical validity is one unique notion specified once and for all by the usual 
Tarskian (square) semantics and canonized by G\"odel's completeness theorem. Moreover, it is essentially complex, 
being undecidable by Church's theorem.

On the present view, however standard predicate logic has arisen historically by making several ad-hoc semantic decisions 
that could have gone differently. Its not all about 'one completeness theorem' but rather about several completeness theorems
obtained by varying both the semantic and syntactical parameters.

But on the other  hand, careful scrutiny of the situation reveals that things are not so clear cut, and the borderlines are hazy. 
Within the polyadic cylindric dichotomy there is the square relativisation dichotomy, and also 
vice versa.

An important border line class of algebras are those studied by Sain. 
Are they in the polyadic paradigm?
According to the last definition, they are. 
But this is not the end of the story. 
Sain's algebras introduced in \cite{Sain} have a unique status. They share the positive properties of both paradigms, 
the cylindric one, and the 
polyadic one. First thing they extend first order logic without equality, 
so that, in particular, cylindrifiers commute (this property is precarious in other contexts, 
it can kill decidability 
and amalgamation.)

They are representable as genuine fields of sets, they are finitely axiomatizable over finitely presented semigroups, 
they have a recursive equational axiomatization,
they admit dilations (neat embedding in $\omega$ extra dimensions), they have the super amalgamation property, and they also have 
{\it infinitary substitutions, at least two of them}.
The main discrepancy between Sain's algebras and polyadic algebras, say, 
is that the equational theory of the last has very high complexity in the recursion theoretic sense, a result of N\'emeti and Sagi \cite{Sagi}.
It is hard to place these algebras in either paradigm alone, 
but we believe that it is fair to say that they belong to the positive part of {\it both}.

\section{Cylindric and Polyadic algebras rolled into one}

Avoiding Platonic complacency, in this section we construct a forest containing the two trees, as opposed to {\it the} forest. 
Other forests are concievable. We do not claim that we have built a complete philosophical house, but we may have opened a lot of windows.

We give a general definition of a system of varieties definable by a uniform schema, that covers, or rather unifies, Monk's definition of 
systems of varieties definable by a schema and the N\'emeti-S\'agi definition of Halmos' schemes. 
In this very general context, we define 
the operation of forming neat reducts, and we furthermore view the neat reduct operator as a functor that lessens dimensions. 
Then we prove a general theorem, extending our previous, that relates adjointness of this functor to 
various forms of the amalgamation property for the system of varieties in question.

The difficulty with the polyadic like algebras, is that there is an algebraic structure on (a part of) 
the indexing set, the set of all maps from an infinite ordinal to itself, 
namely, the operation of composition of maps.
From a universal algebraic perspective this structure does not manifest itself explicitly; instead it is somewhere up there in the meta language. 
This is a situation similar to modules over rings. One approach to deal with such structures in first order logic is to allow two 
sorts, one for the scalars, and the other for the vectors.
We adopt, following N\'emeti and Sagi, the same philosophy; however, we need three sorts, one for the substitutions, one for sets of 
ordinals, and one for the first order situation.

$Ord$ is the class of all ordinals. For ordinals $\alpha<\beta$, $[\alpha,\beta]$ denotes the set of ordinals $\mu $ such that 
$\alpha\leq \mu \leq \beta$. $I_{\alpha}$ is the class $\{\beta\in Ord:\beta\geq \alpha\}$
By an interval of ordinals, or simply an interval, we either mean $[\alpha, \beta]$ or $I_{\alpha}$.

\begin{definition}
\begin{enumroman}
\item A type schema is a quantuple $t=(T, \delta, \rho,c, s)$ such that
$T$ is a set, $\delta$ maps $T$ into $\omega$, $c,s\in T$, and $\delta c=\rho c=\delta s=\rho s=1$.
\item A type schema as in (i) defines a similarity type $t_{\alpha}$ for each $\alpha$ as follows. Sets 
$C_{\alpha}\subseteq \wp(\alpha)$, $G_{\alpha}\subseteq {} ^{\alpha}{\alpha}$ are fixed,
and the domain $T_{\alpha}$ of $t_{\alpha}$ is
$$T_{\alpha}=\{(f, k_0,\ldots k_{\delta f-1}): f\in T\sim\{c,s\}, k\in {}^{\delta f}\alpha\}$$
$$\cup \{(c, r): r\in C_{\alpha}\}\cup \{(q,r): r\in C_{\alpha}\}\cup \{(s,\tau): \tau\in G_{\alpha}\}.$$
For each $(f, k_0,\ldots k_{\delta f-1})\in T_{\alpha}$ we set $t_{\alpha}(f, k_0\ldots k_{\delta f-1})=\rho f$
and we set $\rho(c,r)=\rho(q,r)=\rho (s,\tau)=1$ 
\item Let $\mu$ be an interval of ordinals. A system $(\K_{\alpha}: \alpha\in \mu)$ of classes 
of algebras is of type schema $t$ if for each $\alpha\in \mu$, the class $\K_{\alpha}$ is a class of algebras of type 
$t_{\alpha}$.
\end{enumroman}
\end{definition}
\begin{definition} Let $L_T$ be the first order language that consists of countably many unary relational symbols $(Rel)$, 
countably many function symbols $(Func)$ and countably many 
constants $(Cons)$, which are $r_1, r_2\ldots$ and $f_1, f_2\ldots$ and $n_1, n_2\ldots$, respectively. We let $L_T=Rel\cup Cons\cup Func$.
\end{definition}
\begin{definition}
\begin{enumarab}
\item A schema is a pair $(s, e)$ where $s$ is a first order formula of $L_T$ and $e$ is an equation in the language $L_{\omega}$ 
of $K_{\omega}$. We denote a schema $(s, e)$ by $s\to e$. We define $Ind(L_{\omega})=\omega\cup C_{\omega}\cup G_{\omega}$.
A function $h: L_T\to L_{\omega}$ is admissable if $h$ is an injection and 
$h\upharpoonright Const\subseteq \omega, h\upharpoonright Rel\subseteq C_{\omega}$ and
$h\upharpoonright Func\subseteq G_{\omega}.$
\item Let $g$ be an equation in the languse of $\K_{\alpha}$. Then $g$ is an $\alpha$ instance of a schema 
$s\to e$ if there exist an admissable function $h$, sets, functions and constants
$$r_1^{M},r_2^{M}\ldots f_1^{M}, f_2^{M}\ldots \in {}G_{\alpha}, n_1^{M},n_2^{M}\ldots \in \alpha$$ 
such 
$$M=(\alpha, r_1^M, r_2^M,\ldots f_1^M, f_2^M, n_1^{M}, n_2^{M},\ldots )\models s$$
and $g$ is obtained from $e$ by replacing $h(r_i), h(f_i)$ and $h(n_i)$ by $r_i^M$, $f_i^M$ and $n_i^M$, respectively.
\end{enumarab}
\end{definition}

\begin{definition}
A system of varieties is a {\it generalized system of varieties definable by a schema}, if there exists a strictly finite set of schemes, 
such that for every $\alpha$, 
$\K_{\alpha}$ is axiomatized  by  the $\alpha$ dimensional instances of such schemes.
\end{definition}

Given such a system of varieties, we denote algebras in $\K_{\alpha}$ by 
$$\A=(\B, c_{(r)}, q_{(r)}, s_{\tau})_{r\in C_r, \tau\in G_{\alpha}},$$ that is, we highlight the operations of cylindrifiers and substitutions,
and the operations in $T\sim \{c,s\}$ (of the Monk's schema part, so to speak), with indices from $\alpha$, are encoded in $\B$.

Indeed, Monk's definition is the special case, when we forget the sort of substitutions.
That is a system of varieties is definable by Monk's schemes if $G_{\alpha}=\emptyset$ for all $\alpha$ and each schema the form 
by $True\to e$; see definition below. 

\begin{example}
\begin{enumarab}

\item Tarski's Cylindric algebras, Pinter's substitution algebras, Halmos' quasi-polyadic algebras and Halmos' quasi-polyadic algebras with equality, 
all of infinite dimension. Here $G_{\alpha}=\emptyset$ for all $\alpha\geq \omega$, and $C_{\alpha}=\{\{r\}: r\in \alpha\}$.

\item Less obvious are Halmos' {\it polyadic algebras}, of infinite dimension, as defined in \cite{HMT2}.
Such algebras are  axiomatized by Halmos schemes; hence the form a generalized system of varieties 
definable by a schema of equations. For example, the $\omega$ instance of $(P_{11})$ is:

$[(\forall y) (r_2(y)\longleftrightarrow \exists z(r_1(z)\land y=f_1(z)\land (\forall y,z)(r_2(y)\land r_2(z)\land y\neq z\implies\\
f_1(y)\neq f_1(z), c_{r_1}s_{f_1}(x)=s_{f_1}c_{r_2}(x)].$ 
\end{enumarab}

\end{example}

\begin{example} Cylindric-polyadic algebras \cite{Fer}. 
These are reducts of polyadic algebras of infinite dimension, where we have all substitutions, but cylindrification is allowed only on finitely many 
indices. Such algebras have become fashonable lately, with the important recent the work of Ferenczi. However, 
Ferenczi deals with (non-classical) versions of such algebras, where commutativity of cylindrifiers is weakened, 
substantially, and he proves strong representation theorems on generalized set algebras.
Here $C_{\alpha}$ is again the set of singletons, manifesting the cylindric spirit of the algebras, while 
$G_{\alpha}= {}^{\alpha}\alpha$, manifesting, in turn, its polyadic reduct.

\end{example}

\begin{example} 

(a) Sain's algebras \cite{Sain}: Such algebras povide a solution to one of the most central problems in algebraic logic, 
namely, the so referred to in the literature as the fintizability problem.
Those  are  countable reducts of polyadic algebras, and indeed of cylindric-polyadic algebras.
Cylindrifies are finite, that is they are defined only on finitely many indices, but at least  two infinitary substitutions are there. 

Like polyadic algebras, and for that matter cylindric polyadic algebras, such classes algebras, which happen to be varieties, can 
be easily formulated as a generalized system of varieties definable by a schema on the
interval $[\alpha, \alpha+\omega]$, $\alpha$ a countable ordinal.  

Here we only have substitutions coming from a {\it countable} semigroup $G_{\alpha},$ and $G_{\alpha+n}$, $n\leq \omega$, is the 
sub-semigroup of $^{\alpha+n}\alpha+n$ generated by $\bar{\tau}=\tau\cup Id_{(\alpha+n)\sim \alpha}$, $\tau \in G_{\alpha}$. 
Such algebras, were introduced by Sain, can be modified, in case the semigroups determining their similarity types
are finitely presented, providing first order logic without equality a {\it strictly finitely based} algebraisation, see also \cite{AU}. 

(b) Sain's algebras with diagonal elements \cite{Sain2}. These are investigated by Sain and Gyuris, in the context of 
finitizing first order logic {\it with equality}. This problem turns out to be harder, and so 
the results obtained are weaker, because the class of representable algebras $V$ is not elementary; it is not closed under ultraproducts.
The authors manage to provide, in this case a generalized finite schema, for the class of  ${\bf H}V$; this only implies weak completeness 
for the corresponding infinitary
logics; that is $\models \phi$ implies $\vdash \phi$, relative to a finitary Hilbert style axiomatization, involving only type free valid schemes.
However, there are non-empty sets of formulas $\Gamma$, such that $\Gamma\models \phi$, but there is no proof of $\phi$ from
$\Gamma$. 
\end{example}

It is timely to highlight the novelties in the above definition when compared to Monk's definition of a system definable by schemes. 
\begin{enumarab}

\item First, the most striking addition, is that it allows dealing with infinitary substitutions coming from a set
$G_{\alpha}$, which is usually a semigroup. Also infinitary cylindrifiers are permitted.  
This, as indicated in the above examples, covers polyadic algebras,  
Heyting polyadic algebras, $MV$ polyadic algebras and Ferenzci's 
cylindric-polyadic algebras, together with their important reducts studied by Sain.

\item Second thing, cylindrifiers are not mandatory;
this covers many algebraisations of multi dimensional modal logics, like for example modal logics of substitutions. 
(This will be elaborated upon below, in the new context of complete system of varieties definable 
by a schema, which integrates  finite dimensions).  

\item We also have another (universal) quantifier $q$, intended to be the dual of cylindrifiers in the case of presence of negation;
representing  univeral quantification. This is appropriate for logics where we do not have negation in the classical sense,
like intuitionistic logic,  expressed algebrically by Heyting polyadic algebras.

\item Finally the system could be definable only on an interval of ordinals of the form $[\alpha,\beta]$, while the usual definition of Monk's schemes
defines systems of varieties on $I_{\omega}$; without this more general condition, 
we would have not been able to approach Sain's algebras.

\end{enumarab}

An operator that features prominently in systems of varieties defined earelier is the neat reduct operator, 
which can be viewed as functor from algebras to algebras of a lesser dimension. 
The definition of neat reducts for Monk's schemes is fairly straightforward. 
By allowing infinitary substitutions possibly moving infinitely many points, 
the definition becomes more intricate, and it needs 
caution.
\begin{definition}
\begin{enumarab}
\item Let $(\K_{\alpha}: \alpha\geq \mu)$ be a system of varieties. 
For $\alpha<\beta$, both in $\mu$, and $\tau\in G_{\alpha}$ we assume that $\bar{\tau}=\tau\cup Id\in G_{\beta}$. We also assume that if 
$r\in C_{\alpha}$, then  $r\in C_{\beta}$ for every $\beta>\alpha$ in $\mu$
Given $\A\in \K_{\beta},$ $\A=(\B, q_{(r)}, c_{(r)}, s_{\bar{\tau}})_{r\in C_{\alpha}, \tau\in G_{\beta}},$ say,
we define $$\Rd_{\alpha}\A=(\Rd_{\alpha}\B, q_{(r)}, c_{r}, s_{\bar{\tau}})_{r\in C_{\beta}\alpha, \tau\in G_{\alpha}},$$
where $\Rd_{\alpha}\B$ is the reduct obtained from $\B$ by allowing only operations whose indices come form $\alpha$, and discard the rest.

\item Given $\A\in \K_{\beta}$ and $x\in A$,  $Nr_{\alpha}\A=\{x\in A: \forall r\subseteq (\beta\sim \alpha), r\in C_{\alpha};  c_{(r)}x=x\}.$

\item We assume that for all $f\in T\sim \{c,s\}$, $\A\in \K_{\beta}$ and $\alpha<\beta\in \mu$, 
if $r\in \wp(\beta\sim \alpha)\cap C_{\alpha}$ and $c_{(r)}x=x$, then $c_{(r)}f_{i_0,\ldots i_{n-1}}(\bar{x})=f_{i_0,\ldots i_{n-1}}(\bar{x})$; here
$n=\delta(f)$ and  $|\bar{x}|=\rho(f)$. Same for $q$.
Furthermore if $\tau\in G_{\beta}$ and $\tau\upharpoonright \beta\sim \alpha=Id_{\beta\sim \alpha}$, then for all $r\in \wp(\beta\sim \alpha)\cap C_{\alpha}$,
we have $c_{(r)}s_{\tau}x=s_{\tau}x$ (this is a very reasonable condition, because the indicses moved by the substitution lie outside
the scope of the generalized cylinrifier) . 

Then $\Nr_{\alpha}\B$ is the subalgebra of $\Rd_{\alpha}\B$ with universe $Nr_{\alpha}\B$.
This is well defined. 

\item For $\bold L\subseteq \K_{\beta}$, and $\alpha<\beta$, $\Nr_{\alpha}\bold L=\{\Nr_{\alpha}\A: \A\in \K_{\beta}\}.$ 

\item If $\alpha, \alpha+\omega\in \mu$, we set $\Kn_{\alpha}=S\Nr_{\alpha}\K_{\alpha+\omega}$.

\end{enumarab}
\end{definition}

Since in generalized systems, $\K_{\omega}$ specifies higher dimensions uniquely, it is reasonable to formulate our results 
for only  $\omega$ dimensional algebras.
This is no real restriction; what can be proved for $\omega$ can be proved for any larger ordinal in the interval defining the system.
From now on, we assume that $\omega$ and $\omega+\omega$ are in the interval defining systems adressed.

Call a sytem of varieties {\it nice} if $\Kn_{\omega}$ has the amalgamation property, and call it {\it very nice} if $\Kn_{\omega}$ 
has the superamalgamation property. 
Our next theorem shows that given that $\bold M$ enjoys a strong form of amalgamation (which happens often, like incylindric algebras and quasipolyadic algebras with and without equality and Pinter's 
substitution algebras, 
the amalgamation property is actually equivalent to the adjointness of the neat reduct functor, 
while the superamalgamation property
is equivalent to its strong invertibility. 

Using metaphysical jagon, yet again, 
the next theorem is the heart and soul of this paper, 
formulated rigorously in the dialect of category theory and algebraic 
logic:

\begin{theorem}\label{main} Let $\K=(\K_{\alpha}: \alpha\in \mu)$ be a system of varieties, such that $\omega$ and $\omega+\omega\in \mu$. 
Assume that $\bold M =\{\A\in \K_{\omega+\omega}: \A=\Sg^{\A}\Nr_{\omega}\A\}$
has $SUPAP$. Assume further that for any injective homomorphism $f:\Nr_{\alpha}\B\to \Nr_{\alpha}\B'$, there exists an injective
homomorphism $g:\B\to \B'$ such that $f\subseteq g$. Then the following two conditions are equivalent.
\begin{enumarab}
\item $\Kn_{\omega}$ is (very) nice.
\item $\Nr_{\omega}$ is (strongly) invertible
\end{enumarab}
\end{theorem} 
\begin{demo}{Proof}

\begin{enumarab}
\item Assume that $\Kn_{\omega}$ has the amalgamation property. We first show that 
$\Kn_{\omega}$ has the following unique neat embedding property:
If $i_1:\A\to \Nr_{\alpha}\B_1$, $i_2:\A\to \Nr_{\alpha}\B_1$ 
are such that $i_1(A)$ generates $\B_1$ and $i_2(A)$ generates
$\B_2$, then there is an isomorphism $f:\B_1\to \B_2$ auch that $f\circ i_1=i_2$.

By assumption, there is an amalgam, that is there is  $\D\in \Kn_{\omega}$, $m_1:\Nr_{\alpha}\B_1\to \D$, $m_2:\Nr_{\alpha}\B_2\to \D$ such that
$m_1\circ i_1=m_2\circ i_2$. We can assume that $m_1:\Nr_{\alpha}\B\to \Nr_{\alpha}\D^+$ for some $\D^+\in \bold M$, and similarly for $m_2$.
By hypothesis, let $\bar{m_1}:\B_1\to \D^+$ and $\bar{m_2}:\B_2\to \D^+$ be isomorphisms extending $m_1$ and $m_2$.  
Then since $i_1A$ generates $\B_1$ and $i_2A$ generates $\B_2$, then $\bar{m_1}\B_1=\bar{m_2}\B_2$. It follows that 
$f=\bar{m}_2^{-1}\circ \bar{m_1}$ is as desired.
From this it easily follows that $\Nr$ has universal maps and we are done. 

In  fact, the uniqueness property established above, call it $UNEP$,
is equivalent to existense of unversal maps; this is quite easy to show, hence to prove the converse, we assume $UNEP$, and we set out 
to prove that 
$\Kn_{\omega}$ has $AP$.

Let $\A,\B\in \Kn_{\omega}$. Let $f:\C\to \A$ and $g:\C\to \B$ be injective homomorphisms.
Then there exist $\A^+, \B^+, \C^+\in \K_{\alpha+\omega}$, $e_A:\A\to \Nr_{\alpha}\A^+$ 
$e_B:\B\to  \Nr_{\alpha}\B^+$ and $e_C:\C\to \Nr_{\alpha}\C^+$.
We can assume that $\Sg^{\A^+}e_A(A)=\A^+$ and similarly for $\B^+$ and $\C^+$.
Let $f(C)^+=\Sg^{A^+}e_A(f(C))$ and $g(C)^+=\Sg^{B^+}e_B(g(C)).$
Since $\C$ has $UNEP$, there exist $\bar{f}:\C^+\to f(C)^+$ and $\bar{g}:\C^+\to g(C)^+$ such that 
$(e_A\upharpoonright f(C))\circ f=\bar{f}\circ e_C$ and $(e_B\upharpoonright g(C))\circ g=\bar{g}\circ e_C$.
Now $\bold M$ as $SUPAP$, hence there is a $\D^+$ in $\bold M$ and $k:\A^+\to \D^+$ and $h:\B^+\to \D^+$ such that
$k\circ \bar{f}=h\circ \bar{g}$. Then $k\circ e_A:\A\to \Nr_{\alpha}\D^+$ and
$h\circ e_B:\B\to \Nr_{\alpha}\D^+$ are one to one and
$k\circ e_A \circ f=h\circ e_B\circ g$.

\item  Now for the second equivalence. Assume that $\Kn_{\omega}$ has $SUPAP$. Then, {\it a fortiori}, it has $AP$ 
hence, by the above argument, it has $UNEP.$ 
We first show that if $\A\subseteq \Nr_{\alpha}\B$ and $\A$ generates $\B$ then equality holds, we call this property $NS$, short for neat reducts commuting with forming 
subalgebras.

If not, then $\A\subseteq \Nr_{\alpha}\B$, $\B\in K,$ $A$ generates $\B$ and $\A\neq \Nr_{\alpha}\B$.
Then $\A$ embeds into $\Nr_{\alpha}\B$ via the inclusion map $i$ . Let $\C=\Nr_{\alpha}\B$.
By $SUPAP$, there exists $\D\in \Kn_{\omega}$ and $m_1$, $m_2$ monomorphisms 
from $\C$ to $\D$ such that $m_1(\C)\cap m_2(\C)=m_1\circ i(\A)$. Let $y\in \C\sim A$. 
Then $m_1(y)\neq m_2(y)$ for else $d=m_1(y)=m_2(y)$ will be in
$m_1(\C)\cap m_2(\C)$ but not in $m_1\circ i(\A)$. Assume that 
$\D\subseteq \Nr_{\alpha}\D^+$ with $\D^+\in K$. By hypothesis, there exist injections
$\bar{m_1}:\B\to \D^+$ and $\bar{m_2}:\B\to \D^+$ extending $m_1$ and $m_2$. 
But $A$ generates $\B$ and so $\bar{m_1}=\bar{m_2}$.
Thus $m_1y=m_2y$ which is a contradiction.   

Now let  $\beta=\alpha+\omega$. Let $\bold M=\{\A\in \K_{\beta}: \A=\Sg^{\A}\Nr_{\omega}\A$\}. 
Let $\Nr:\bold M\to \Kn_{\omega}$ be the neat reduct functor. 
We show that  $\Nr$ is strongly invertible, namely there is a functor $G:\Kn_{\omega}\to \bold M$ and natural isomorphisms
$\mu:1_{\bold M}\to G\circ \Nr$ and $\epsilon: \Nr\circ G\to 1_{\Kn_{\omega}}$.
Let $L$ be a system of representatives for isomorphism on $Ob(\bold M)$.
For each $\B\in Ob(\Kn_{\omega})$ there is a unique $G(\B)$ in $\bold M$ such that $\Nr(G(\B))\cong \B$.
Then $G:Ob(\Kn_{\omega})\to Ob(\bold M)$ is well defined. 
Choose one isomorphism $\epsilon_B: \Nr(G(\B))\to \B$. If $g:\B\to \B'$ is a $\Kn_{\omega}$  morphism, then the square

\begin{displaymath}
    \xymatrix{ \Nr(G(B)) \ar[r]^{\epsilon_B}\ar[d]_{\epsilon_B^{-1}\circ g\circ \epsilon_{B'}} & B \ar[d]^g \\
               \Nr(G(B'))\ar[r]_{\epsilon_{B'}} & B' }
\end{displaymath}
commutes. There is a unique morphism $f:G(\B)\to G(\B')$ such that $\Nr(f)=\epsilon_{\B}^{-1}\circ g\circ \epsilon$.
We let $G(g)=f$. Then it is easy to see that $G$ defines a functor. Also, by definition $\epsilon=(\epsilon_{\B})$ 
is a natural isomorphism from $\Nr\circ G$ to $1_{\Kn_{\omega}}$.
To find a natural isomorphism from $1_{\bold M}$ to $G\circ \Nr,$ observe that that for each $\A\in Ob(\bold M)$,
$\epsilon_{\Nr\A}:\Nr\circ G\circ \Nr(\A)\to \Nr(\A)$ is an isomorphism.
Then there is a unique $\mu_A:\A\to G\circ \Nr(\A)$ such that $\Nr(\mu_{\A})=\epsilon_{\Nr\A}^{-1}.$
Since $\epsilon^{-1}$ is natural for any $f:\A\to \A'$ the square
\bigskip
\bigskip
\begin{displaymath}
    \xymatrix{ \Nr(A) \ar[r]^{\epsilon_{\Nr(A)}^{-1}=\Nr(\mu_A)}\ar[d]_{\Nr(f)} & \Nr\circ G\circ \Nr(A) \ar[d]^{\Nr\circ G\circ \Nr(f)} \\
               \Nr(A')\ar[r]_{\epsilon_{\Nr\A}^{-1}=\Nr(\mu_{A'})} & \Nr\circ G\circ \Nr(A') }
\end{displaymath}

commutes, hence the square

\bigskip
\begin{displaymath}
    \xymatrix{ A \ar[r]^{\mu_A}\ar[d]_f & G\circ \Nr(A) \ar[d]^{G\circ \Nr(f)} \\
               A'\ar[r]_{\mu_{A'}} & G\circ \Nr(A') }
\end{displaymath}

commutes, too. Therefore $\mu=(\mu_A)$ is as required.

Conversely, assume that the functor $\Nr$ is invertible. Then we have the $UNEP$ and the $NS$. The  $UNEP$ 
follows from the fact that the functor has 
a right adjoint, and so it has universal maps. To prove that it has $NS$ assume for contradiction 
that there exists $\A$ generating subreduct of $\B$ and 
$\A$ is not isomorphic to $\Nr_{\alpha}\B$. This means that $Nr$ is not invertible, because 
had it been invertible, with inverse $Dl$, then  $Dl(\A)=Dl(\Nr_{\alpha}\B)$ and this cannot happen.

Now we prove that $\Kn_{\omega}$ has $SUPAP$.
We obtain (using the notation in the first part)
$\D\in \Nr_{\alpha}\K_{\alpha+\omega}$ 
and $m:\A\to \D$ $n:\B\to \D$
such that $m\circ f=n\circ g$.
Here $m=k\circ e_A$ and $n=h\circ e_B$.  Denote $k$ by $m^+$ and $h$ by $n^+$.
Suppose that $\C$ has $SNEP$. We further want to show that if $m(a) \leq n(b)$, 
for $a\in A$ and $b\in B$, then there exists $t \in C$ 
such that $ a \leq f(t)$ and $g(t) \leq b$.
So let $a$ and $b$ be as indicated. 
We have  $m^+ \circ e_A(a) \leq n^+ \circ e_B(b),$ so
$m^+ ( e_A(a)) \leq n^+ ( e_B(b)).$
Since $\bold M$ has $SUPAP$, there exist $ z \in C^+$ such that $e_A(a) \leq \bar{f}(z)$ and
$\bar{g}(z) \leq e_B(b)$.
Let $\Gamma = \Delta z \sim \alpha$ and $z' =
{\sf c}_{(\Gamma)}z$. So, we obtain that 
$e_A({\sf c}_{(\Gamma)}a) \leq \bar{f}({\sf c}_{(\Gamma)}z)~~ \textrm{and} ~~ \bar{g}({\sf c}_{(\Gamma)}z) \leq
e_B({\sf c}_{(\Gamma)}b).$ It follows that $e_A(a) \leq \bar{f}(z')~~\textrm{and} ~~ \bar{g}(z') \leq e_B(b).$ Now by hypothesis
$$z' \in \Nr_\alpha \C^+ = \Sg^{\Nr_\alpha \C^+} (e_C(C)) = e_C(C).$$ 
So, there exists $t \in C$ with $ z' = e_C(t)$. Then we get
$e_A(a) \leq \bar{f}(e_C(t))$ and $\bar{g}(e_C(t)) \leq e_B(b).$ It follows that $e_A(a) \leq e_A \circ f(t)$ and 
$e_B \circ g(t) \leq
e_B(b).$ Hence, $ a \leq f(t)$ and $g(t) \leq b.$

\end{enumarab}
\end{demo}

\section{Another adjoint situation for finite dimensions}

\begin{definition} Let ${C}\in \CA_{\alpha}$ and $I\subseteq \alpha$, and let $\beta$ be the order type of $I$. Then
$$Nr_IC=\{x\in C: c_ix=x \textrm{ for all } i\in \alpha\sim I\}.$$
$$\Nr_{I}{\C}=(Nr_IC, +, \cdot ,-, 0,1, c_{\rho_i}, d_{\rho_i,\rho_j})_{i,j<\beta},$$
where $\beta$ is the unique order preserving one-to-one map from $\beta$ onto $I$, and all the operations 
are the restrictions of the corresponding operations on $C$. When $I=\{i_0,\ldots i_{k-1}\}$ 
we write $\Nr_{i_0,\ldots i_{k-1}}\C$. If $I$ is an initial segment of $\alpha$, $\beta$ say, we write $\Nr_{\beta}\C$.
\end{definition}
Similar to taking the $n$ neat reduct of a $\CA$, $\A$ in a higher dimension, is taking its $\Ra$ reduct, its relation algebra reduct.
This has unverse consisting of the $2$ dimensional elements of $\A$, and composition and converse are defined using one spare dimension.
A slight generalization, modulo a reshufflig of the indicies: 

\begin{definition}\label{RA} For $n\geq 3$, the relation algebra reduct of $\C\in \CA_n$ is the algebra
$$\Ra\C=(Nr_{n-2, n-1}C, +, \cdot,  1, ;, \breve{}, 1').$$ 
where $1'=d_{n-2,n-1}$, $\breve{x}=s_{n-1}^0s_{n-1}^{n-2}s_0^{n-1}x$ and $x;y=c_0(s_0^{n-1}x. s_0^{n-2}y)$. 
Here $s_i^j(x)=c_i(x\cdot d_{ij})$ when $i\neq q$ and $s_i^i(x)=x.$
\end{definition}
But what is not obvious at all is that an $\RA$ has a $\CA_n$ reduct for $n\geq 3$. 
But Simon showed that certain relations algebras do; namely the $\QRA$s.

\begin{definition} A relation algebra $\B$ is a $\QRA$ 
if there are elements $p,q$ in $\B$ satisfying the following equations:
\begin{enumarab}
\item $\breve{p};p\leq 1', q; q\leq 1;$
\item  $\breve{p};q=1.$
\end{enumarab}
\end{definition}
In this case we say that $\B$ is a $\QRA$  with quasi-projections $p$ and $q$. 
To construct cylindric algebras of higher dimensions 'sitting' in a $\QRA$, 
we need to define certain terms. seemingly rather complicated, their intuitive meaning 
is not so hard to grasp.
\begin{definition} Let $x\in\B\in \RA$, then $dom(x)=1';(x;\breve{x})$ and $ran(x)=1';(\breve{x}; x)$, $x^0=1'$, $x^{n+1}=x^n;x$. $x$ 
is a functional element if $x;\breve{x}\leq 1'$.
\end{definition}
Given a $\QRA$, which we denote by $\bold Q$, we have quasi-projections $p$ and $q$ as mentioned above. 
Next we define certain terms in ${\bf Q}$, cf. \cite{Andras}:


$$\epsilon^{n}=dom q^{n-1},$$
$$\pi_i^n=\epsilon^{n};q^i;p,  i<n-1, \pi_{n-1}^{(n)}=q^{n-1},$$
$$ \xi^{(n)}=\pi_i^{(n)}; \pi_i^{(n)},$$
$$ t_i^{(n)}=\prod_{i\neq j<n}\xi_j^{(n)}, t^{(n)}=\prod_{j<n}\xi_j^{(n)},$$
$$ c_i^{(n)}x=x;t_i^{(n)},$$
$$ d_{ij}^{(n)}=1;(\pi_i^{(n)}.\pi_j^{(n)}),$$
$$ 1^{(n)}=1;\epsilon^{(n)}.$$
and let
$$\B_n=(B_n, +, \cdot, -, 0,1^{(n)}, c_i^{(n)}, d_{ij}^{(n)})_{i,j<n},$$
where $B_n=\{x\in B: x=1;x; t^{(n)}\}.$
The intuitive meaning of those terms is explained in \cite{Andras}, right after their definition on p. 271.

\begin{theorem} Let $n>1$
\begin{enumerate} 
\item Then ${\B}_n$ is closed under the operations.
\item ${\B}_n$ is a $\CA_n$.
\end{enumerate}
\end{theorem}
\begin{proof} (1) is  proved in \cite{Andras} lemma 3.4 p.273-275 where the terms are definable in a $\QRA$. 
That it is a $\CA_n$ can be proved as \cite{Andras}  theorem 3.9.
\end{proof} 

\begin{definition} Consider the following terms.
$$suc (x)=1; (\breve{p}; x; \breve{q})$$
and
$$pred(x)=\breve{p}; ranx; q.$$ 
\end{definition}
It is proved in \cite{Andras} that $\B_n$ neatly embeds into $\B_{n+1}$ via $succ$. The successor function thus codes 
extra dimensions. The thing to observe here is that  we will see that $pred$; its inverse; 
guarantees a condition of commutativity of two operations: forming neat reducts and forming subalgebras;
it does not make a difference which operation we implement first, as long as we implement both one after the other.
So the function $succ$ {\it captures the extra dimensions added.}. From the point of view of {\it definability} it says 
that terms definable in extra dimensions add nothing, they are already term definable.
And this indeed is a definability condition, that will eventually lead to stong interpolation property we wnat.

\begin{theorem}\label{neat} Let $n\geq 3$. Then $succ: {\B}_n\to \{a\in {\B}_{n+1}: c_0a=a\}$ 
is an isomorphism into a generalized neat reduct of ${\B}_{n+1}$.
Strengthening the condition of surjectivity,  for all $X\subseteq \B_n$, $n\geq 3$, we have (*)
$$succ(\Sg^{\B_n}X)\cong \Nr_{1,2,\ldots, n}\Sg^{\B_{n+1}}succ(X).$$
\end{theorem}

\begin{proof} The operations are respected by \cite{Andras} theorem 5.1. 
The last condition follows  because of the presence of the 
functional element $pred$, since we have $suc(pred x)=x$ and $pred(sucx)=x$, when $c_0x=x$, \cite{Andras} 
lemmas 4.6-4.10. 
\end{proof}
\begin{theorem}
Let $n\geq 3$. Let ${\C}_n$ be the algebra obtained from ${\B}_n$ by reshuffling the indices as follows; 
set $c_0^{{\C}_n}=c_n^{{\B}_n}$ and $c_n^{{\C}_n}=c_0^{{\cal B}_n}$. Then ${\C}_n$ is a cylindric algebra,
and $suc: {\C}_n\to \Nr_n{\C}_{n+1}$ is an isomorphism for all $n$. 
Furthermore, for all $X\subseteq \C_n$ we have
$$suc(\Sg^{\C_n}X)\cong \Nr_n\Sg^{\C_{n+1}}suc(X).$$ 
\end{theorem}

\begin{proof} immediate from \ref{neat}
\end{proof}  
\begin{theorem} Let ${\C}_n$ be as above. Then $succ^{m}:{\C_n}\to \Nr_n\C_m$ is an isomophism, such that 
for all $X\subseteq A$, we have
$$suc^{m}(\Sg^{\C_n}X)=\Nr_n\Sg^{\C_m}suc^{n-1}(X).$$
\end{theorem} 
\begin{proof} By induction on $n$.
\end{proof}
Now we want to neatly embed our $\QRA$ in $\omega$ extra dimensions. At the same we do not want to lose, our control over the streching;
we still need the commutativing of taking, now $\Ra$  reducts with forming subalgebras; we call this property the $\Ra S$ property.
To construct the big $\omega$ dimensional algebra, we use a standard ultraproduct construction.
So here we go.
For $n\geq 3$, let  ${\C}_n^+$ be an algebra obtained by adding $c_i$ and $d_{ij}$'s for $\omega>i,j\geq n$ arbitrarity and with 
$\Rd_n^+\C_{n^+}={\B}_n$. Let ${\C}=\prod_{n\geq 3} {\C}_n^+/G$, where $G$ is a non-principal ultrafilter
on $\omega$. 
In our next theorem, we show that the algebra $\A$ can be neatly embedded in a locally finite algebra $\omega$ dimensional algebra
and we retain our $\Ra S$ property. 

\begin{theorem} Let $$i: {\A}\to \Ra\C$$
be defined by
$$x\mapsto (x,  suc(x),\ldots suc^{n-1}(x),\dots n\geq 3, x\in B_n)/G.$$ 
Then $i$ is an embedding ,
and for any $X\subseteq A$, we have 
$$i(\Sg^{\A}X)=\Ra\Sg^{\C}i(X).$$
\end{theorem}
\begin{proof} The idea is that if this does not happen, then it will not happen in a fnite reduct, and this impossible \cite{Sayed}.

\end{proof}

\begin{theorem} Let $\bold Q\in {\RA}$. Then for all $n\geq 4$, there exists a unique 
$\A\in S\Nr_3\CA_n$ such that $\bold Q=\Ra\A$, 
such that for all $X\subseteq A$, $\Sg^{\bf Q}X=\Ra\Sg^{\A}X.$
\end{theorem}
\begin{proof} This follows from the previous theorem together with $\Ra S$ property.
\end{proof}

\begin{corollary} Assume that $Q=\Ra\A\cong \Ra\B$ then this lifts to an isomorphism from $\A$ to $\B$.
\end{corollary}
The previous theorem says that $\Ra$ as a functor establishes an equivalence between ${\QRA}$ 
and a reflective subcategory of $\Lf_{\omega}.$
We say that $\A$ is the $\omega$ dilation of ${\bf Q}$.
Now we are ready for:

\begin{theorem} $\QRA$ has $SUPAP$.
\end{theorem}
\begin{proof}  We form the unique dilatons of the given algebras required to be superamalgamated. 
These are locally finite so we can find a superamalgam $\D$. Then $\Ra\D$ will be required superamalgam; it contains quasiprojections because the base algebras 
does.
Let $\A,\B\in \QRA$. Let $f:\C\to \A$ and $g:\C\to \B$ be injective homomorphisms .
Then there exist $\A^+, \B^+, \C^+\in \CA_{\alpha+\omega}$, $e_A:\A\to \Ra{\alpha}\A^+$ 
$e_B:\B\to  \Ra\B^+$ and $e_C:\C\to \Ra\C^+$.
We can assume, without loss,  that $\Sg^{\A^+}e_A(A)=\A^+$ and similarly for $\B^+$ and $\C^+$.
Let $f(C)^+=\Sg^{\A^+}e_A(f(C))$ and $g(C)^+=\Sg^{\B^+}e_B(g(C)).$
Since $\C$ has $UNEP$, there exist $\bar{f}:\C^+\to f(C)^+$ and $\bar{g}:\C^+\to g(C)^+$ such that 
$(e_A\upharpoonright f(C))\circ f=\bar{f}\circ e_C$ and $(e_B\upharpoonright g(C))\circ g=\bar{g}\circ e_C$. Both $\bar{f}$ and $\bar{g}$ are 
monomorphisms.
Now $Lf_{\omega}$ has $SUPAP$, hence there is a $\D^+$ in $K$ and $k:\A^+\to \D^+$ and $h:\B^+\to \D^+$ such that
$k\circ \bar{f}=h\circ \bar{g}$. $k$ and $h$ are also monomorphisms. Then $k\circ e_A:\A\to \Ra\D^+$ and
$h\circ e_B:\B\to \Ra\D^+$ are one to one and
$k\circ e_A \circ f=h\circ e_B\circ g$.
Let $\D=\Ra\D^+$. Then we obtained $\D\in \QRA$ 
and $m:\A\to \D$ $n:\B\to \D$
such that $m\circ f=n\circ g$.
Here $m=k\circ e_A$ and $n=h\circ e_B$. 
Denote $k$ by $m^+$ and $h$ by $n^+$.
Now suppose that $\C$ has $NS$. We further want to show that if $m(a) \leq n(b)$, 
for $a\in A$ and $b\in B$, then there exists $t \in C$ 
such that $ a \leq f(t)$ and $g(t) \leq b$.
So let $a$ and $b$ be as indicated. 
We have  $(m^+ \circ e_A)(a) \leq (n^+ \circ e_B)(b),$ so
$m^+ ( e_A(a)) \leq n^+ ( e_B(b)).$
Since $K$ has $SUPAP$, there exist $z \in C^+$ such that $e_A(a) \leq \bar{f}(z)$ and
$\bar{g}(z) \leq e_B(b)$.
Let $\Gamma = \Delta z \sim \alpha$ and $z' =
{\sf c}_{(\Gamma)}z$. (Note that $\Gamma$ is finite.) So, we obtain that 
$e_A({\sf c}_{(\Gamma)}a) \leq \bar{f}({\sf c}_{(\Gamma)}z)~~ \textrm{and} ~~ \bar{g}({\sf c}_{(\Gamma)}z) \leq
e_B({\sf c}_{(\Gamma)}b).$ It follows that $e_A(a) \leq \bar{f}(z')~~\textrm{and} ~~ \bar{g}(z') \leq e_B(b).$ Now by hypothesis
$$z' \in \Ra\C^+ = \Sg^{\Ra\C^+} (e_C(C)) = e_C(C).$$ 
So, there exists $t \in C$ with $ z' = e_C(t)$. Then we get
$e_A(a) \leq \bar{f}(e_C(t))$ and $\bar{g}(e_C(t)) \leq e_B(b).$ It follows that $e_A(a) \leq (e_A \circ f)(t)$ and 
$(e_B \circ g)(t) \leq
e_B(b).$ Hence, $ a \leq f(t)$ and $g(t) \leq b.$
We are done.
\end{proof} 

One can prove the theorem using the dimension restricted free algebra $B=\Fr_1^{\rho}\CA_{\omega}$, where $\rho(0)=2$.
This corresponds to a countable first order language with a sequence of variables of order type $\omega$ and one binary relation.
The idea is that $\Fr_1\QRA\cong  \Ra\Fr_1^{\rho}\CA_{\omega}$. So let 
$a, b\in \Fr_1\QRA$ be such that $a\leq b$. Then there exists $y\in \Sg^{\B}\{x\}$ were $x$ is the free generator
of both, such that
$a\leq y\leq b$.

But we need to show that pairing functions can be defined in $\Ra\Fr_{1}\CA_{\omega}$
We have one binary relation $E$ in our langauge; for convenience, 
we write $x\in y$ instead of $E(x,y)$, to remind ourselves that we are actually working in the language
of set theory. 
We define certain formulas culminating in formulating the axioms of a finitely undecidable theory, better known as Robinson's arithmetic 
in our language. These formulas are taken from N\'emeti \cite{Nemeti}. (This is not the only way to define quasi-projections)
We need to define, the quasi projections. Quoting Andr\'eka and N\'emeti in \cite{AN}, we do this by 'brute force'.

$$x=\{y\}=:y\in x\land (\forall z)(z\in x\implies z=y)$$
$$\{x\}\in y=:\exists z(z=\{x\}\land z\in y)$$
$$x=\{\{y\}\}=:\exists z(z=\{y\}\land x=\{z\})$$
$$x\in \cup y:=\exists z(x\in z\land z\in y)$$
$$pair(x)=:\exists y[\{y\}\in x\land (\forall z)(\{z\}\in x\to z=y)]\land \forall zy[(z
\in \cup x\land \{z\}\notin x\land$$
$$y\in \cup x\land \{y\}\notin x\to z=y]\land \forall z\in x\exists y
(y\in z).$$
Now we define the pairing functions:
$$p_0(x,y)=:pair(x)\land \{y\}\in x$$
$$p_1(x,y)=:pair(x)\land [x=\{\{y\}\}\lor (\{y\}\notin x\land y\in \cup x)].$$
$p_0(x,y)$ and $p_1(x,y)$ are defined.

\subsection{ Pairing functions in N\'emetis directed $\CA$s}

We recall the definition of what is called weakly higher order cylindric algebras, or directed cylindric algebras invented by N\'emeti 
and further studied by S\'agi and Simon. 
Weakly higher order cylindric algebras are natural expansions of cylindric algebras. 
They have extra operations that correspond to a certain kind of bounded existential 
quantification along a binary relation $R$. The relation $R$ is best thought of as the `element of relation' in a model of some set theory.
It is an abstraction of the membership relation. These cylindric-like algebras 
are the cylindric counterpart of quasi-projective relation algebras, introduced by Tarski. These algebras were 
studied by many authors
including Andr\'eka, Givant, N\'emeti, Maddux, S\'agi, Simon, and others. The reference \cite{Andras} is recommended for other references in the topic.
It also has reincarnations in Computer Science literature 
under the name of Fork algebras.
We start by recalling the concrete versions of directed cylindric algebras:

\begin{definition}(P--structures and extensional structures.) \\
Let $U$ be a set and let $R$ be a binary relation on $U$. The structure
$\langle U; R \rangle$ is defined to be a P--structure\footnote{``P'' stands for ``pairing'' or ``pairable''.} iff for every
elements $a,b \in U$ there exists an element $c \in U$ such that $R(d,c)$ is
equivalent with $d=a$ or $d=b$ (where $d \in U$ is arbitrary) , that is, \\
\\
\centerline{ $\langle U; R \rangle \models (\forall x,y)(\exists z)(\forall w)( R(w,z) \Leftrightarrow (w=x$ or $w=y))$.} \\
\\
The structure $\langle U; R \rangle $ is defined to be a \underline{weak P--structure} iff \\
\\
\centerline{ $ \langle U; R \rangle \models (\forall x,y)(\exists z)(R(x,z) $ and $ R(y,z))$.} \\
\\
The structure $\langle U; R \rangle$ is defined to be {extensional}
iff every two points $a,b \in U$ coincide whenever they have the same
``$R$--children'', that is, \\
\\
\centerline{ $\langle U; R \rangle \models (\forall x,y)(((\forall z) R(z,x) \Leftrightarrow R(z,y)) \Rightarrow x=y) $.}
\end{definition}

\noindent
We will see that if $\langle U; R \rangle$ is a P--structure then one can
``code'' pairs of elements of $U$ by a single element of $U$ and whenever
$\langle U; R \rangle$ is extensional then this coding is ``unique''. In fact,
in $\RCA_{3}^{\uparrow}$ (see the definition below) one can define terms similar
to quasi--projections and, as with the class of $\QRA$'s, one can equivalently
formalize many theories of first order logic as equational theories of certain
$\RCA_{3}^{\uparrow}$'s. Therefore $\RCA_{3}^{\uparrow}$ is in our main interest.
$\RCA_{\alpha}^{\uparrow}$ for bigger $\alpha$'s behave in the same way, an
explanation of this can be found in \cite{Sagi} and can be deduced from our proof, which shows that $\RCA_{3}^{\uparrow}$ has implicitly $\omega$ 
extra dimensions.
\begin{definition}
\label{canyildef}
(${\sf Cs}^{\uparrow}_{\alpha}$, $\RCA^{\uparrow}_{\alpha}$.) \\
Let $\alpha$ be an ordinal. Let $U$ be a set and let $R$ be a binary relation on $U$
such that $\langle U; R \rangle$ is a weak P--structure.
Then the
{full w--directed cylindric set algebra} of dimension $\alpha$ with base
structure $\langle U; R \rangle$ is the algebra: \\
\\
\centerline{$\langle {\cal P}({}^{\alpha}U); \cap, -, C_{i}^{\uparrow(R)}, C_{i}^{\downarrow(R)}, D_{i,j}^{U} \rangle_{i,j \in \alpha}$,} \\
\\
where $\cap$ and $-$ are set theoretical intersection and complementation (w.r.t. ${}^{\alpha}U$),
respectively, $D^{U}_{i,j} = \{ s \in {}^{\alpha}U: s_{i}=s_{j} \}$ and
$ C_{i}^{\uparrow(R)}, C_{i}^{\downarrow(R)}$ are defined as follows. For every
$X \in {\cal P}({}^{\alpha}U)$: \\
\\
\indent $ C_{i}^{\uparrow(R)}(X) = \{ s \in {}^{\alpha}U: (\exists z \in X)( R(z_{i},s_{i})$ and $(\forall j \in \alpha)(j \not=i \Rightarrow s_{j}=z_{j})) \},$ \\
\indent $ C_{i}^{\downarrow(R)}(X) = \{ s \in {}^{\alpha}U: (\exists z \in X)( R(s_{i},z_{i})$ and $(\forall j \in \alpha)(j \not=i \Rightarrow s_{j}=z_{j})) \}.$ \\
\\
The class of {w--directed cylindric set algebras} of dimension $\alpha$
and the class of {directed cylindric set algebras} of dimension $\alpha$
are defined as follows. \\
\\
\centerline{$ w-{\sf Cs}^{\uparrow}_{\alpha} ={\bf S} \{ {\cal A}: \ {\cal A}$ is a full
w--directed cylindric set algebra of dimension $\alpha$} \\
\centerline{ \indent \indent \indent \indent with base structure $\langle U; R \rangle$,
for some weak P--structure $\langle U; R \rangle \}$.} \\
\\
\centerline{$ {\sf Cs}^{\uparrow}_{\alpha} ={\bf S} \{ {\cal A}: \ {\cal A}$ is a full
w--directed cylindric set algebra of dimension $\alpha$} \\
\centerline{ \indent \indent \indent \indent with base structure $\langle U; R \rangle$,
for some extensional P--structure $\langle U; R \rangle \}$.} \\
\\
The class $\RCA^{\uparrow}_{\alpha}$ of {representable directed cylindric algebras} of
dimension $\alpha$ is defined to be $\RCA^{\uparrow}_{\alpha} = {\bf SP}{\sf Cs}^{\uparrow}_{\alpha}$.
\end{definition}

The main result of Sagi in \cite{Sagi} is a direct proof for the following: \\
\\
\begin{theorem}\label{rep}
{\em $\RCA^{\uparrow}_{\alpha}$ is a finitely axiomatizable
variety whenever $\alpha \geq 3$ and $\alpha$ is finite}
\end{theorem}

$\CA^{\uparrow}_3$ denotes the variety of directed cylindric algebras of dimension $3$ 
as defined in \cite{Sagi} definition 3.9. In \cite{Sagi}, it is proved that
$\CA^{\uparrow}_3=\RCA^{\uparrow}_3.$ A set of axioms is formulated on p. 868 in \cite{Sagi}. 
Let $\A\in \CA^{\uparrow}_3$. 
Then we have quasi-projections 
$p,q$ defined on $\A$ as defined in \cite{Sagi} p. 878, 879. We recall their definition, which is a little bit complicated because 
they are defined as formulas in the corresponding second 
order logic.
Let $\cal L$ denote the untyped logic corresponding to directed $\CA_3$'s as  defined p.876-877 in \cite{Sagi}. It has only $3$ variables. 
There is a correspondance between formulas (or formual schemes)  in this language and $\CA^{\uparrow}_3$ terms. 
This is completely analgous to the corresponance between $\RCA_n$ terms and first order formulas containing only $n$ variables.
For example $v_i=v_j$ corresponds to $d_{ij}$, $\exists^{\uparrow} v_i(v_i=v_j)$ correspond to ${\sf c}^{\uparrow}_i d_{ij}$.
In \cite{Sagi} the following formulas (terms) are defined:

\begin{definition} Let $i,j,k \in 3$ distinct elements. 
We define variable--free $RCA^{\uparrow}_{3}$ terms as follows:
\begin{tabbing}
\indent \= $ v_{i} = \{ \{ v_{j} \}_{R} \}_{R}$ \ \ \= is \indent \= $\exists v_{k}( v_{k} = \{ v_{j} \}_{R} \wedge v_{i} = \{ v_{k} \}_{R})$, \kill
\> $ v_{i} \in_R v_{j}$ \indent \> is \> $\exists^{\uparrow}v_{j}(v_{i}=v_{j})$, \\
\> $ v_{i} = \{ v_{j} \}_{R}$ \> is \> $\forall v_{k}( v_{k} \in_R v_{j} \Leftrightarrow v_{k}=v_{j}) $, \\
\> $ \{ v_{i} \}_{R} \in_R v_{j}$ \> is \> $\exists v_{k}( v_{k} \in_R v_{j} \wedge v_{k} = \{ v_{i} \}_{R})$, \\
\> $ v_{i} = \{ \{ v_{j} \}_{R} \}_{R}$ \> is \> $\exists v_{k}( v_{k} = \{ v_{j} \}_{R} \wedge v_{i} = \{ v_{k} \}_{R})$ ,\\
\> $ v_{i} \in_R \cup v_{j}$ \> is \> $\exists v_{k}(v_{i} \in_R v_{k} \wedge v_{k} \in_R v_{j})$.
\end{tabbing}
\end{definition}

Therefore $pair_{i}$ (a pairing function) can be defined as follows: \\

\indent $\exists v_{j} \forall v_{k}( \{ v_{k} \}_{R} \in_R v_{i} \Leftrightarrow v_{j} = v_{k}) \ \wedge $ \\
\indent $\forall v_{j} \exists v_{k}( v_{j} \in_R v_{i} \Rightarrow v_{k} \in_R v_{j} ) \ \wedge $ \\
\indent $\forall v_{j} \forall v_{k}( v_{j} \in_R \cup v_{i} \ \wedge \ \{ v_{j} \} \not\in_{R} v_{i} \ \wedge \ v_{k} \in_R \cup v_{i} \ \wedge \ \{ v_{k} \} \not\in_{R} v_{i} \Rightarrow v_{j} = v_{k} )$. \\

It is clear that this is a term built up of diagonal elements and directed cylindrifications. 
The first quasi-projection  $v_{i} = P(v_{j})$ can be chosen as: \\
\\
\indent $pair_{j} \ \wedge \ \forall^{\downarrow} v_{j} \exists^{\downarrow} v_{j} (v_{i} = v_{j})$. \\
\\
and the second quasiprojection  $v_{i} = Q(v_{j})$ can be chosen as: \\
\\
\centerline{$pair_{j} \ \wedge \ (( \forall v_{i} \forall v_{k} ( v_{i} \in_R v_{j} \ \wedge \ v_{k} \in_R v_{j}  \Rightarrow v_{i} = v_{k} )) \Rightarrow v_{i} = P(v_{j})) \ \wedge $} \\
\centerline{ $(\exists v_{i} \exists v_{k}( v_{i} \in_R v_{j} \ \wedge \ v_{k} \in_R v_{j} \ \wedge \  v_{i} \not= v_{k}) \Rightarrow (v_{i} \not= P(v_{j}) \ \wedge \ \exists^{\downarrow} v_{j} \exists^{\downarrow} v_{j}(v_{i} = v_{j})))$.}

\begin{theorem}
Let $\B$ be the relation algebra reduct of $\A$; then $\B$ is a relation algebra, and the 
variable free terms corresponding to the formulas  $v_i=P(v_j)$ and $v_j=Q(v_j)$ 
call  them $p$ and $q$, respectively, are quasi-projections.
\end{theorem}
\begin{proof} 
One proof is very tedious, though routine. One translates the functions as variable free terms in the language of $\CA_3$ and use 
the definition of composition and converse in the $\RA$ reduct, to verify that they are quasi-projections.
Else one can look at their meanings on set algebras, which we recall from Sagi \cite{Sagi}.
Given a cylindric set algebra $\cal A$ with base $U$ and accessibility relation $R$
$$(v_i=P(v_j))^A=\{s\in {}^3U: (\exists a,b\in U)(s_j=(a,b)_R, s_i=a\}$$
$$(v_i=Q(v_j)^A=\{s\in {}^3U: (\exists a,b\in U)(s_j=(a,b)_R, s_i=b\}.$$
First $P$ and $Q$ are functions, so they are functional elements. 
Then it is clear that in this set algebras that $P$ and $Q$ are quasi-projections. 
Since $\RCA^{\uparrow}_3$ is the variety generated by set algebras, they have the same meaning in the class $\CA^{\uparrow}_3.$
\end{proof}

Now we can turn the class around. Given a $\QRA$ one can define a directed $\CA_n$, for every finite $n\geq 2$.
This definition is given by N\'emeti and Simon in \cite{NS}. 
It is vey similar to Simon's definition above (defining $\CA$ reducts in a $\QRA$, 
except that directed cylindrifiers along a relation $R$ are 
implemented.

\begin{theorem} The concrete category $\QRA$ 
with morphisms injective homomorphisms, and that of $\CA^{\uparrow}$  with morphisms also injective homomorphisms are equivalent.
in particular $\CA^{\uparrow}$ of dimension $3$ is equivalent to $\CA^{\uparrow}$ for $n\geq 3$.
\end{theorem}
\begin{proof} Given $\A$ in $\QRA$ we can associte a directed $\CA_3$, homomorphism are restrictions and vice versa; these are inverse Functors.
However, when we pass from an $\QRA$ to a $\CA^{\uparrow}$ and then take the $\QRA$ reduct, 
we may not get back exactly to the $\QRA$ we started off with,
but the new quasi projections are definable from the old ones. 
Via this equivalence, we readily  conclude that $\RCA_3\to \RCA_n$ are also equivalent.
\end{proof}

\end{document}